\RequirePackage{fix-cm}
\documentclass[smallextended]{svjour3_no_journal_info}       
\smartqed  
\usepackage[letterpaper,top=3cm,bottom=3cm,left=3cm,right=3cm,marginparwidth=1.75cm]{geometry}

\usepackage{adjustbox}
\usepackage{lipsum}
\usepackage{amsfonts}
\usepackage{graphicx}
\usepackage{wrapfig}
\usepackage{epstopdf}
\DeclareGraphicsExtensions{.pdf,.png,.jpg,.eps}  

\usepackage[utf8]{inputenc} 
\usepackage[T1]{fontenc}    
\usepackage{url}            
\usepackage{booktabs}       
\usepackage{amsfonts}       
\usepackage{nicefrac}       
\usepackage{mathtools}
\usepackage{natbib}
\usepackage{subcaption}
\usepackage{amsmath,amssymb,amsfonts,amsxtra,bm}
\usepackage{enumerate}
\usepackage{enumitem}
\usepackage{hyperref}       
\usepackage{tcolorbox}
\usepackage{algorithm}
\usepackage{algpseudocode}
\usepackage{booktabs, 
            makecell, multirow, tabularx} 
\usepackage{wrapfig}
\usepackage{colortbl}

\clearpage{}%

\newcommand{\Oc}{\mathcal{O}}

\DeclareMathOperator{\Proj}{Proj}

\newtheorem{assumption}{Assumption}
\newtheorem{observation}{Observation}
\makeatletter
\newcommand{\leqnomode}{\tagsleft@true}
\newcommand{\reqnomode}{\tagsleft@false}
\newcounter{mycounter}
\definecolor{pythonblue}{rgb}{0.1216, 0.4667, 0.7059} 
\definecolor{pythonorange}{rgb}{1.0, 0.4980, 0.0549}
\definecolor{pythongreen}{rgb}{0.1625, 0.6175, 0.1625}  
\definecolor{lightgray}{RGB}{245,245,245}
\definecolor{lightergray}{RGB}{255,255,255}
\definecolor{lightpurple}{RGB}{237,230,250}
\definecolor{mediumlightpurple}{RGB}{225,210,245}

\definecolor{ao}{rgb}{0.0, 0.5, 0.0}
\newcommand{\cmark}{\textcolor{ao}{\ding{51}}}
\newcommand{\xmark}{\textcolor{purple}{\ding{55}}}
\definecolor{LightCyan}{rgb}{0.88,1,1} 
\definecolor{Lightpurple}{rgb}{0.9,0.9,1}

\usepackage{pifont}

\title{
Efficient Penalty-Based Bilevel Methods: Improved Analysis, Novel Updates, and Flatness Condition
}

\titlerunning{Theory-Inspired Remedies for Efficient Bilevel Optimization via Penalty Methods}        

\author{
        Liuyuan Jiang\and
        Quan Xiao \and
        Lisha Chen \and
 Tianyi Chen
}

\authorrunning{Liuyuan Jiang, Quan Xiao,  Lisha Chen, Tianyi Chen} 

\institute{Liuyuan Jiang$^*$, Quan Xiao$^\dagger$,  Lisha Chen$^*$ \and Tianyi Chen$^\dagger$ \at
University of Rochester$^*$, Cornell University$^\dagger$\\
\email{ljiang24@ur.rochester.edu; quanx1808@gmail.com; 
lishachen9577@gmail.com;
chentianyi19@gmail.com}           
}
 
\date{}

\begin{document}

\maketitle
\begin{abstract}
Penalty-based methods have become popular for solving bilevel optimization (BLO) problems with or without constraints, thanks to their effective first-order nature. However, they often require inner-loop iterations to solve the lower-level problem and small step sizes to handle the increased smoothness induced by large penalty terms, leading to suboptimal complexity. 
This work considers the general BLO problem with coupled constraints (CCs) and leverages a novel penalty reformulation that decouples the upper- and lower-level variables. This reformulation yields an improved analysis of the smoothness constant, enabling larger step sizes and reduced iteration complexity for a Penalty-Based Gradient Descent algorithm in ALTernating fashion (ALT-PBGD).
Building on the insight of reduced smoothness, we further propose PBGD-Free, an efficient algorithm that avoids solving the lower-level problem iteratively to its optimality. It is a fully single-loop algorithm for the uncoupled constraint case, requiring only one update of the lower-level variable per iteration. For the coupled constraint case, it still employs inner-loop updates but with substantially reduced iteration complexity.
Furthermore, we propose a curvature condition, which describes the "flatness" of the upper-level objective with respect to the lower-level optimal variable without relying on the small Lipschitz modulus assumption. This condition relaxes the traditional upper-level Lipschitz requirement, enables smaller penalty constant choices, and results in a negligible penalty gradient term during upper-level variable updates. 
We provide rigorous convergence analysis and validate the method's efficacy through hyperparameter optimization for support vector machines and fine-tuning of large language models (LLMs).

\keywords{Bilevel optimization \and First-order methods \and Convergence   analysis }
\subclass{90C26 \and 90C15 \and 90C06 \and 90C60 \and 49M37 \and 68Q25}
\end{abstract}

\section{Introduction}

Bi-Level Optimization (BLO) has garnered significant attention owing to its wide-ranging applications in distributed learning \cite{qin2025duet,gao2023convergence}, meta-learning \cite{franceschi2018bilevel,chen2023_fnt_meta}, Large Language Models (LLMs) \cite{shen2024seal,zangrando2025debora}, reinforcement learning \cite{tan2023bi}, financial pricing \cite{fernandez2015bilevel,wei2022bi}, transportation \cite{santos2021bilevel}, and beyond.
These applications involve both constrained and unconstrained settings, motivating us to study the BLO problem with coupled constraints (CCs):
\begin{subequations}
\label{eq: original problem 1}
    \begin{align}
     \min_{x\in \mathcal{X}} ~  & \phi(x) := \min_{ y\in S_g(x)}f(x,y)  \label{eq: original upper-level problem}
    \\
    \text{where} & \quad  
    S_g(x) := \arg \min_{y\in \mathcal{Y}(x)} g(x,y) \quad \text{and}\quad  \mathcal{Y}(x)= \{ y\in \mathcal{Y}: c(x,y)\leq 0\}.  \label{eq: lower-level problem}
\end{align}
\end{subequations}
Here, $f:\mathbb{R}^{d_x}\times \mathbb{R}^{d_y}\to\mathbb{R}$ and $g:\mathbb{R}^{d_x}\times \mathbb{R}^{d_y}\to\mathbb{R}$ are the UL and LL objectives, respectively. The UL variable $x$ is restricted to the feasible set $\mathcal{X}\subseteq\mathbb{R}^{d_x}$, onto which is typically simple and easy to project. The LL variable $y$ is constrained to $\mathcal{Y}(x)=\{y\in\mathcal{Y}:c(x,y)\le0\}$, where $\mathcal{Y}\subseteq\mathbb{R}^{d_y}$ is an easy-to-project set, such as an Euclidean ball, and $c:\mathbb{R}^{d_x}\times\mathbb{R}^{d_y}\to\mathbb{R}^{d_c}$ encodes coupled inequality constraints.

The goal of the BLO problem \eqref{eq: original problem 1} is to find an optimal $x^*$ minimizing $\phi(x)$ and the corresponding LL optimal solution $y_g^*(x^*)\in S_g(x^*)$. In the unconstrained case, i.e., $\mathcal{Y}=\mathbb{R}^{d_y},~c(x,y)=0$, Implicit Gradient Descent (IGD) methods established the differentiability of $\phi(x)$ and derived
\begin{align} \label{eq: implicit gradient}
     \nabla \phi(x)= \nabla_x f(x,y_g^*(x)) + \nabla_y f(x,y_g^*(x))\nabla_{yy}g(x,y_g^*(x))^{-1}\nabla_{yx}g(x,y_g^*(x)) 
\end{align}
under the assumption that $g(x,\cdot)$ is strongly convex, in which case the LL solution set reduces to the singleton $S_g(x)=\{y_g^*(x)\}$ \cite{chen2022single,chen2021closing,ghadimi2018approximation,hong2020two,ji2020provably,khanduri2021near,shen2022single,Li2022fully,sow2022convergence}. More recently, 
\begin{wrapfigure}{r}{0.45\textwidth}
    \centering
    \vspace{-.4cm}
    \includegraphics[width=0.45\textwidth]{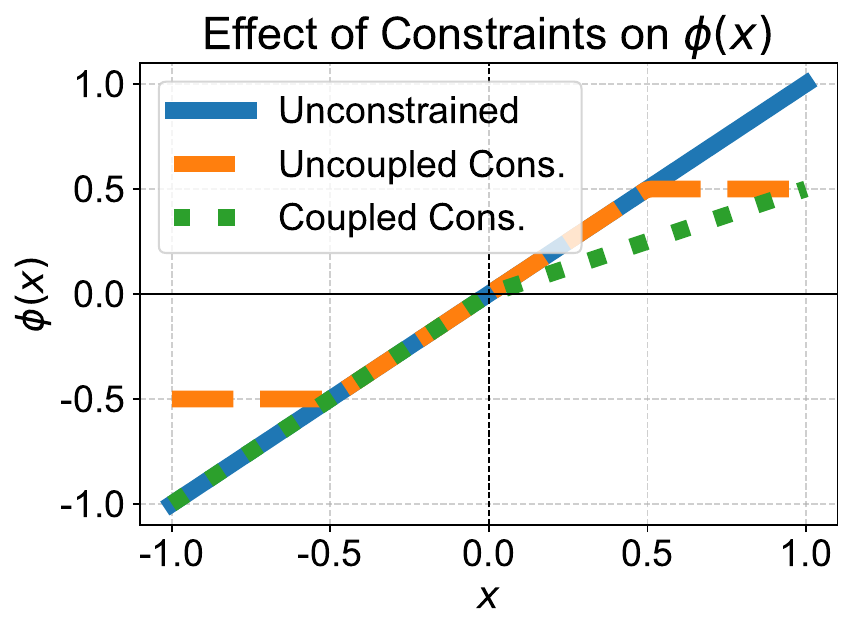}
            \vspace{-.4cm}
        \caption{Example of non-differentiable $\phi(x)$ due to the constraint set $\mathcal{Y}(x)$. Consider $f(x,y)=y$, $g(x,y)=(x-y)^2$. For \textcolor{pythonblue}{unconstrained} $\mathcal{Y}(x)=\mathbb{R}$, $\phi(x)=x$ is differentiable. For \textcolor{pythonorange}{uncoupled constrained} $\mathcal{Y}(x)=[-1/2,1/2]$, $\phi(x)$ is not differentiable at $x=\pm 1/2$; For \textcolor{pythongreen}{coupled constrained} $\mathcal{Y}(x)=\{y: y \le x/2\}$, $\phi(x)$ is not differentiable at $x=0$.}
        \label{fig:differentiability_of_phi}
    \vspace{-.5cm}
\end{wrapfigure}
\citet{kwon2023penalty} showed that similar expression \eqref{eq: implicit gradient} remains valid for any  $y_g^*(x)\in S_g(x)$ under the LL Polyak–Łojasiewicz (PL) condition. Nevertheless, IGD methods remain computationally expensive and are ill-suited for handling the non-smoothness induced by constraints.

For the uncoupled constrained BLO, the LL feasible set can be characterized by $\tilde{\mathcal{Y}} = \{y \in \mathcal{Y}: c(y)\le 0\}$, independent of $x$. Differentiability of $\phi(x)$ is generally hard to guarantee unless one assumes $\mathcal{Y} = \mathbb{R}^{d_y}$ and imposes additional restrictive conditions on $\tilde{\mathcal{Y}}$ \cite{kwon2023penalty,khanduri2023linearly}. The more challenging case is BLO with coupled constraints (CCs), where $\mathcal{Y}(x)$ itself depends on $x$ through $c(x,y)\leq 0$. In this setting, the solution mapping $x\mapsto S_g(x)$, and thus $\phi(x)$, can behave irregularly due to changing active constraints. Consequently, $\phi(x)$ may lose differentiability even when $g(x,\cdot)$ is strongly convex (cf. Fig.~\ref{fig:differentiability_of_phi}). 
This makes it difficult to implement IGD or to analyze the convergence of BLO algorithms via $\nabla \phi(x)$, as $\phi(x)$ may fail to be differentiable in the presence of constraints.

Penalty-based methods \cite{ye2022bome,liu2021value,shen2022single,kwon2023fully,kwon2023penalty,jiang2024primal,tsp2025,yao2025overcoming,mehra2021penalty,lu2024first} offer an alternative by leveraging penalty formulations. Building upon equilibrium backpropagation \cite{scellier2017equilibrium,zucchet2022beyond}, one effective branch, the value-function-based penalty methods \cite{shen2023penalty}, considers jointly minimizing $(x,y)\in \mathcal{X}\times\mathcal{Y}(x)$ over 
\begin{align}
     \tilde{F}_\gamma(x,y) := & f(x,y)+\gamma(g(x,y)-v(x)), \quad\text{where}  \quad v(x) := \min_{y_g\in \mathcal{Y}(x)} g(x,y_g). \label{eq: joint penalty problem} 
\end{align}
Under LL proximal PL condition, \citet{shen2023penalty} established that when the penalty hyper-parameter is chosen as $\gamma = \Theta(\epsilon^{-0.5})$, the local solutions to \eqref{eq: joint penalty problem} lie within $\Oc(\epsilon)$-squared-distance of those to an $\epsilon$-approximation of \eqref{eq: original problem 1}: 
\begin{align} 
    \min_{x\in \mathcal{X}, y\in \mathcal{Y}(x) } f(x,y) \quad \text{s.t.}\quad  d_{S_g(x)}(y)^2 \leq \epsilon. \label{eq: epsilon app prob} 
\end{align}
This enables the convergence analysis of BLO through the penalized objective $\tilde{F}_\gamma$ in \eqref{eq: epsilon app prob}. 
Additionally, \citet{shen2023penalty} established the closed form of the gradient of the value function $v(x)$ and found that $\tilde{F}_\gamma(x,y)$ is $l_{\tilde{F}_\gamma,1}=\Theta(\gamma)$-smooth. This enables the development of a series of fully-first-order Penalty-Based Gradient Descent (PBGD) methods \cite{kwon2023penalty,shen2023penalty,jiang2024primal,chen2024finding}, which are more efficient than IGD methods. 

However, for large-scale problems, the joint projection onto $\mathcal{X} \times \mathcal{Y}(x)$ required by projected gradient-based updates can be costly per iteration \cite{jiang2024primal}. In addition, PBGD requires an inner loop to find $y_g^*(x) \in S_g(x)$ for estimating $\nabla v(x)$. Although this inner loop converges linearly under lower-level strong convexity, it still imposes a non-negligible computational burden in large-scale settings. Moreover, the outer-loop convergence is slowed by the need for a small step size, $\eta = \Oc(\gamma^{-1})$, to satisfy $\eta \le l_{\tilde{F}_\gamma,1}^{-1}$ as required by gradient descent, since the smoothness constant $l_{\tilde{F}_\gamma,1}$ scales with $\gamma = \Theta(\epsilon^{-0.5})$. Taken together, these factors significantly increase the overall complexity, limiting the practical efficiency of standard penalty-based methods.

Consequently, this paper contributes by presenting an improved theoretical analysis of the smoothness of a variant of value-function-based penalty reformulations, developing an efficient algorithm for BLO with CCs in \eqref{eq: original problem 1}, and establishing convergence guarantees under new, mild conditions.

\subsection{Main contributions}

We highlight key contributions in developing efficient BLO algorithms as follows.

\textsf{\textbf{C1)} Improved outer-loop iteration complexity for solving BLO with CCs that matches gradient descent.} This is enabled by writing $\tilde{F}_\gamma(x,y)$ into a decoupled form:
\begin{align}
& F_\gamma(x) := \min_{y\in \mathcal{Y}(x)}  \tilde{F}_\gamma(x,y) = \gamma \min_{y_\gamma\in \mathcal{Y}(x)} \left(\gamma^{-1}f(x,y_\gamma)+ g(x,y_\gamma)\right) - \gamma \min_{y_g\in \mathcal{Y}(x)} g(x,y_g)
\label{eq: F gam function} 
\end{align}
which leads to an \emph{alternating PBGD scheme} that first optimizes $y$, then updates $x$. While the scheme avoids the expensive joint projection onto $\mathcal{X}\times\mathcal{Y}(x)$, its iteration complexity under CCs is suboptimal \cite{jiang2024primal}.
Inspired by the enhanced analysis for the unconstrained BLO problem \cite{chen2024finding}, we show that $F_\gamma(x)$ is smooth with modulus $l_{F_\gamma,1}=\Oc(1)$ independent of $\gamma$ for both non-coupled constrained setting in Sec.~\ref{sec: Improved Convergence Rate under Uncoupled Constraints} and coupled constrained setting in Sec.~\ref{sec: improved convergence under CCs}, which is the key to obtain the improved convergence rate.

\textsf{\textbf{C2)} We propose PBGD-Free, a value-function-free algorithm.} Building on the improved outer-loop complexity of the alternating scheme, PBGD-Free can avoid the inner loop for computing $y_g^*(x)\in S_g(x)$ by estimating $\nabla F_\gamma(x)$ using $\nabla_x f(x,y^\gamma)$, with $y^\gamma \approx \arg\min_{y\in \mathcal{Y}(x)}F_\gamma(x,y)$ (see Fig.~\ref{fig:PBGD flow chart}). In Sec.~\ref{sec: Developing Value-Function-Free Algorithms: PBGD-Free}, we start by demonstrating its empirical effectiveness on a LLM parameter-efficient fine-tuning (PEFT) problem. However, we also show that without additional assumptions, PBGD-Free does not necessarily converge for the worst-case performance (cf. Sec.~\ref{sec: Negative theoretical observations of PBGD-Free}).

\textsf{\textbf{C3)} We introduce a mild sufficient condition for PBGD-Free to converge.}
In Sec.~\ref{sec: Developing Value-Function-Free Algorithms: PBGD-Free}, we propose the \emph{$(\delta,\alpha)$-flatness} condition (Def.~\ref{def: flatness}), which relaxes $\alpha$-Hölder continuity and characterizes the flatness of $f(x,y)$ at $y = y_g^*(x)$. This condition serves as an alternative to the classical Lipschitz continuity of $f(x,\cdot)$ and is validated on the PEFT problem in Sec.~\ref{section: flatness of the representation learning PEFT}.
In Sec.~\ref{sec:PBGD-Free}, we show that under flatness, the update bias of PBGD-Free becomes negligible. For uncoupled problems, the fully \emph{single-loop version} of PBGD-Free converges to a stationary point of BLO, achieving the same rate as gradient descent. For BLO with coupled constraints, Sec.~\ref{sec: BLOCC-VaFF} shows that PBGD-Free with a $\Oc(\ln(\epsilon^{-1}))$-inner loop is similarly effective and requires one fewer inner loop compared with the existing alternating-update method \cite{jiang2024primal}. A detailed comparison with other algorithms is given in Table~\ref{tab:table1-comparison}.

\textsf{\textbf{C4)} We empirically validate our theoretical results.}
We first confirm the $\mathcal{O}(1)$-smoothness of $F_\gamma(x)$ on toy examples (cf. Sec.~\ref{sec: toy example for O1 smoothness}). We then evaluate PBGD-Free on real-world tasks, including LLM PEFT in Sec.~\ref{sec: representation PEFT} and SVM hyperparameter optimization in Sec.~\ref{sec: SVM exp}, which showcase the computational efficiency of our algorithm.

\begin{figure}
\centering
\includegraphics[width=0.8\linewidth]{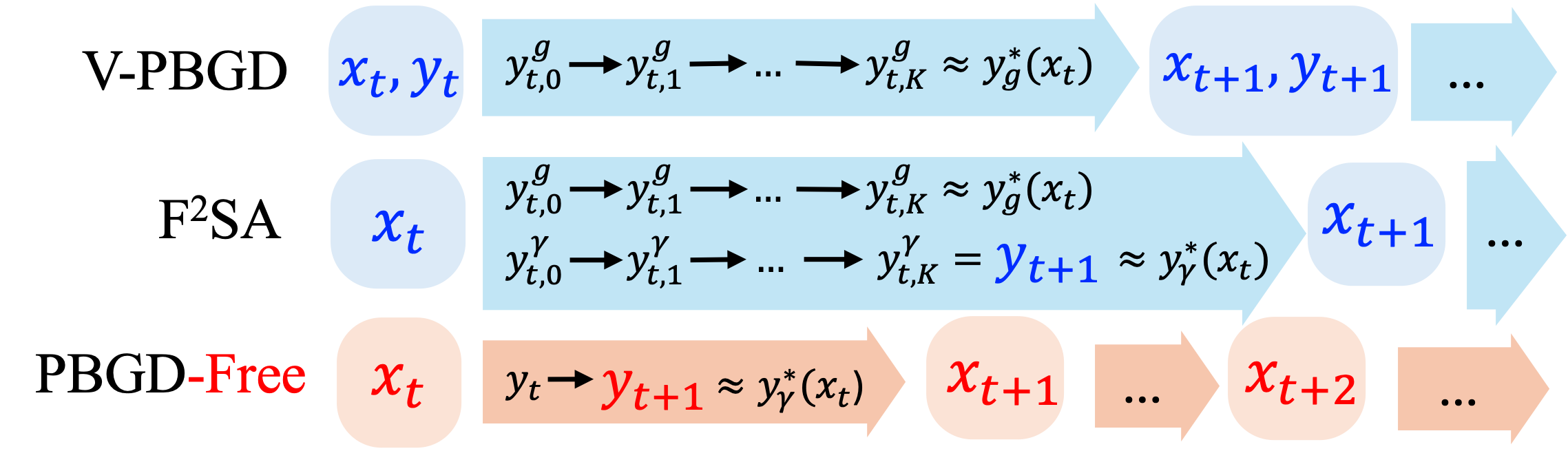}
    \vspace{0.3cm}
\caption{Comparison of V-PBGD, F$^2$SA, and PBGD-Free updates in the uncoupled constraint setting. V-PBGD \cite{shen2023penalty} (\textbf{top}) is a joint-PBGD method for \eqref{eq: joint penalty problem}, and F$^2$SA \cite{kwon2023penalty} (\textbf{middle}) is an alternating-PBGD method for \eqref{eq: F gam function}. Both refine the LL variable through multiple inner iterations before updating $x_t$. In contrast, PBGD-Free (\textbf{bottom}) performs \textbf{only a single-step update} from $y_t$ to $y_{t+1}$, providing a more efficient but potentially less accurate approximation $\nabla_x f(x_t,y_{t+1})\approx\nabla F_\gamma(x_t)$.}
\label{fig:PBGD flow chart}
\end{figure}

{\fontsize{9pt}{11pt}\selectfont
\begin{table}[t]
\centering
\begin{adjustbox}{width=0.99\textwidth}
    \begin{tabular}{l | l | l | l | l}
        \hline \hline
        \textbf{Method} & \textbf{Assumptions} & \textbf{$\Oc(1)$-smooth} & \textbf{Outer Loop}  & \textbf{Inner Loop}  \\
        \hline \hline
        \multicolumn{5}{c}{\textbf{Unconstrained $\mathcal{Y}(x)=\mathbb{R}^{d_y}$}} \\
        \hline
        F$^2$SA & $\tilde{F}_\gamma(x,\cdot)$ PL & \cmark & $\Oc(\epsilon^{-1})$ & $\Oc(2\log \epsilon^{-1})$\\
        \rowcolor{gray!20}\textbf{PBGD-Free} & $\tilde{F}_\gamma(x,\cdot)$ PL & \cmark & $\Oc(\epsilon^{-1})$ & \xmark \\
        \hline\hline
        \multicolumn{5}{c}{\textbf{Uncoupled constrained $\mathcal{Y}(x)=\tilde{\mathcal{Y}}=\{y\in \mathcal{Y}: c(y)\leq 0\}$}} \\
        \hline
        V-PBGD & $g(x,\cdot)$ PL & \xmark  & $\Oc(\epsilon^{-1.5})$  & $\Oc(\log \epsilon^{-1})$\\
        F$^2$FA & $\tilde{F}_\gamma(x,\cdot)$ PL; $\mathcal{Y}=\mathbb{R}^{d_y}$; $\tilde{Y}$ compact & \xmark & $\Oc(\epsilon^{-1.5})$  & $\Oc(2\log \epsilon^{-1})$\\
        \rowcolor{gray!10}\textbf{ALT-PBGD} & $g(x,\cdot)$ S.C. & \cmark & $\Oc(\epsilon^{-1})$ & $\Oc(2\log \epsilon^{-1})$\\
        \rowcolor{gray!20}\textbf{PBGD-Free} & $f(x,\cdot)$ flat; $g(x,\cdot)$ S.C. & \cmark & $\Oc(\epsilon^{-1})$ & \xmark\\
        \hline\hline
        \multicolumn{5}{c}{\textbf{Coupled constrained $\mathcal{Y}(x)=\{y\in \mathcal{Y}: c(x,y)\leq 0\}$}} \\
        \hline
        \multirow{2}{*}{BLOCC} & $g(x,\cdot)$ S.C.; & \multirow{2}{*}{\xmark} & \multirow{2}{*}{$\Oc(\epsilon^{-1.5})$} & $\Oc(2\epsilon^{-1})$; \\
        & $\mathcal{Y}=\mathbb{R}^{d_y}$; $c(x,y)=A(x)y+B(x)$& & &  $\Oc(2\log \epsilon^{-1})$ \\
        BiC-GAFFA & $g(x,\cdot)$ Conv. & \xmark & $\Oc(\epsilon^{-3})$ & \xmark \\
        \rowcolor{gray!10} & $g(x,\cdot)$ S.C.; &  &  & $\Oc(2\epsilon^{-1})$; \\
        \rowcolor{gray!10} \multirow{-2}{*}{\textbf{PBGD-BLOCC}}& $\mathcal{Y}=\mathbb{R}^{d_y}$; $c(x,y)=A(x)y+B(x)$ & \multirow{-2}{*}{\cmark}&\multirow{-2}{*}{$\Oc(\epsilon^{-1})$} &  $\Oc(2\log \epsilon^{-1})$ \\
        \rowcolor{gray!20}\textbf{PBGD-Free} & $f(x,\cdot)$ flat; $g(x,\cdot)$ S.C. & \cmark & $\Oc(\epsilon^{-1})$ & $\Oc(\log \epsilon^{-1})$\\
        \hline \hline
    \end{tabular}
    \end{adjustbox}
    \vspace{0.3cm}
    \caption{Comparison with V-PBGD \cite{shen2023penalty} (PBGD-JNT), F$^2$SA \cite{kwon2023penalty} (PBGD-ALT) and its improved analysis in uncoupled case~\cite{chen2023near,chen2024finding}, BLOCC \cite{jiang2024primal}, and BiC-GAFFA \cite{yao2025overcoming}.
    The convergence metric is the squared (generalized) gradient norm of the penalty reformulation whose $\epsilon$-local minimizers correspond to the $\epsilon$-approximation problem \eqref{eq: epsilon app prob}. Here, "PL", "Conv.", "S.C.", and "flat" respectively stand for (Proximal-)Polyak--Łojasiewicz condition, convex, strongly convex, and flatness under Definition~\ref{def: flatness}.}
    \label{tab:table1-comparison}
\end{table}
}

\subsection{Prior art}

\noindent\textbf{BLO without coupled constraint. } Unconstrained BLO has a long history, starting with \citet{bracken1973mathematical}. Recent advances focus on efficient gradient-based methods with finite-time guarantees. Implicit gradient descent (IGD), introduced by \citet{pedregosa2016hyperparameter}, approximates the hypergradient "$\frac{\partial}{\partial x} y_g^*(x)$" using the implicit function theorem including Hessian inverse. Several studies have explored finite-time convergence for unconstrained BLO \citep{hong2020two,ji2020provably,chen2021closing,khanduri2021near,ji2021bilevel,shen2022single, Li2022fully,sow2022convergence,chen2022single,yang2023accelerating,ghadimi2018approximation}, 
mostly focus on the strongly convex lower-level problem.
Another line of work, penalty-based methods, reformulate BLO as a single-level problem with penalty terms, which avoids using Hessian, and are able to tackle non-strongly-convex lower-level problem and are fully first-order. This approach can be dated back to \citet{ye1997exact} and has been gaining popularity recently \citep{liu2021value,mehra2021penalty,ye2022bome,kwon2023fully,kwon2023penalty,shen2023penalty,gao2022value,lu2024slm} due to its efficiency. This approach can handle the feasible-domain constraint $\mathcal{Y}$, but it cannot be applied directly to coupled constraints. 

\noindent\textbf{BLO with coupled constraint}. However, BLO with lower-level constraints has been less explored until recent years. IGD methods have been developed to address lower-level constraints that can be characterized by $c(x,y)$, progressing from linear constraints  \citep{khanduri2023linearly,xiao2022alternating} to coupled inequality constraints $c(x,y) \leq 0$ \citep{xu2023efficient}. In comparison, penalty-based methods can effectively handle the nonsmoothness introduced by the general uncoupled feasible constraints $\mathcal{Y}$ \citep{shen2023penalty}. Built upon this, recent works, such as LV-HBA \citep{yao2024constrained}, BLOCC \citep{jiang2024primal}, and BiC-CAFFA \citep{yao2025overcoming} consider both feasible constraint $\mathcal{Y}$ and coupled inequality $c(x, y) \leq 0$. However, none of them match the convergence rate of gradient descent in the unconstrained setting. 

\noindent\textbf{BLO algorithms with high complexity.} Since most BLO algorithms rely on inner loops for hyper-gradient estimation and lower-level optimization, another line of research focuses on improving the per-iteration complexity of BLO algorithms, specifically by reducing the inner loop. \citet{Li2022fully} proposed an IGD-based fully single-loop method for unconstrained BLO. For stochastic BLO with uncoupled lower-level inequality constraints, \citet{kwon2023fully} introduced a penalty-based fully single-loop method with momentum, achieving $\Oc(\epsilon^{-1.5})$ complexity for non-stochastic problems.
For coupled constraints, BiC-CAFFA \citep{yao2025overcoming} 
and BLOCC \citep{jiang2024primal} achieve $\Oc(\epsilon^{-3})$ and $\Oc(\epsilon^{-2.5})$ iteration complexity. However, these methods are still of high complexity and may be less feasible for large-scale problems. 

\noindent\textbf{Landscape-aware optimization.} Landscape-aware optimization leverages structural properties of objective functions into algorithm design to accelerate the convergence or improve the generalization. Newton-type methods, which use second-order curvature information to rescale gradients, have been utilized in BLO \citep{fang2025qnbo,ramzi2021shine,dong2025efficient} for efficient Hessian-vector calculation in IGD-based BLO methods. Sharpness-aware minimization \citep{foretsharpness}, which seeks solutions robust to local perturbations and promotes convergence to flat minima, has also been incorporated into BLO \citep{abbas2022sharp}  for improved generalization. Other landscape conditions in single-level optimization, such as relaxed smoothness \citep{zhanggradient,li2023convergence} and Hessian spectrum \citep{zhang2024transformers,ghorbani2019investigation}, are key to explaining the theoretical benefits of empirically effective algorithms like gradient clipping and Adam \citep{kingma2014adam}. However, most existing works focus on second-order BLO algorithms, and none have explored BLO tailored landscape conditions. 

\subsection{Notations and general preliminaries}

We use $\|\cdot\|$ to denote the $\ell_2$-norm. For a set $\mathcal{S} \subseteq \mathbb{R}^n$, $\text{int}(\mathcal{S})$ and $\operatorname{bd}(\mathcal{S})$ denote the interior and boundary of $\mathcal{S}$, respectively, and $d_\mathcal{S}(q)$ denotes the distance of $q\in \mathbb{R}^n$ to $\mathcal{S}$. For a square matrix $A\in \mathbb{R}^{n\times n}$, $A^{-1}$ denotes its inverse and $A^{\dagger}$ its Moore–Penrose pseudoinverse. Standard asymptotic notations are used: $\mathcal{O}(\cdot)$, $\Omega(\cdot)$, and $\Theta(\cdot)$ denote upper, lower, and tight bounds up to constant factors, and $\tilde{\mathcal{O}}(\cdot)$ suppresses logarithmic factors.

In the following, we provide some critical definitions useful in this paper.

\begin{definition}[Lipschitz Continuity and Smoothness]
Let $h:\mathcal{Q}\subseteq \mathbb{R}^n\rightarrow \mathbb{R}^m$. We say $h$ is \textit{$l_{h,0}$-Lipschitz} continuous on $\mathcal{S} \subseteq \mathcal{Q}$ if 
\begin{align*}
    \| h(q_1)-h(q_2)\| \leq l_{h,0} \|q_1-q_2\|, \quad \forall q_1, q_2 \in \mathcal{S}.
\end{align*}
If $h$ is differentiable on $\mathcal{S}$, we say $h$ is \textit{$l_{h,1}$-smooth} on $\mathcal{S}$ if its gradient $\nabla h$ is $l_{h,1}$-Lipschitz continuous on $\mathcal{S}$, or equivalently, if
\begin{align*}
    h(q_2) \leq h(q_1) + \nabla h(q_1)^\top (q_2 - q_1) + \frac{l_{h,1}}{2} \|q_2 - q_1\|^2, \quad \forall q_1, q_2 \in \mathcal{S}.
\end{align*}
Moreover, we say $h$ is \emph{locally Lipschitz} on $\mathcal{S}$ if for all $q_1 \in \mathcal{S}$, there exist constants $l_{h,1} > 0$ and $\delta > 0$ such that
\begin{align*}
    \| h(q_1) - h(q_2) \| \leq M \|q_1 - q_2\|, \quad \forall q_2 \in \mathcal{S} \text{ with } \| q_1 - q_2 \| < \delta.
\end{align*}
\end{definition}

\begin{definition}[Convexity and Strong Convexity]
Let $h:\mathcal{Q} \rightarrow \mathbb{R}$ be a function and $\mathcal{S}\subseteq \mathcal{Q}$ be a convex set.
We say $h$ is \emph{convex} on $\mathcal{S}$ if for all $q_1, q_2 \in \mathcal{S}$ and all $\lambda \in [0,1]$, 
\begin{align*}
    h(\lambda q_1 + (1 - \lambda) q_2) \leq \lambda h(q_1) + (1 - \lambda) h(q_2).
\end{align*}
We say $h$ is \emph{$\mu_g$-strongly convex} on $\mathcal{S}$ for some constant $\mu_g > 0$ if for all $q_1, q_2 \in \mathcal{S}$ and all $\lambda \in [0,1]$, 
\begin{align*}
    h(\lambda q_1 + (1 - \lambda) q_2) \leq \lambda h(q_1) + (1 - \lambda) h(q_2) - \frac{\mu_g}{2} \lambda(1 - \lambda)\|q_1 - q_2\|^2.
\end{align*}
If $h$ is differentiable on $\mathcal{S}$, then $h$ is $\mu_g$-strongly convex on $\mathcal{S}$ if and only if
\begin{align*}
    h(q_2) \geq h(q_1) + \nabla h(q_1)^\top (q_2 - q_1) + \frac{\mu_g}{2} \|q_2 - q_1\|^2, \quad \forall q_1, q_2 \in \mathcal{S}.
\end{align*}
\end{definition}

\begin{definition}[Projection]
\label{def: projection}
For any convex set $\mathcal{S} \subseteq \mathbb{R}^n$, the projection of $q \in \mathbb{R}^n$ onto $\mathcal{S}$ under a well-defined $A$-norm is defined by
\begin{align*}
    \Proj_{\mathcal{S}}^A(q) := \arg\min_{s \in \mathcal{S}} \|s - q\|_A.
\end{align*}
When $A \in \mathbb{R}^{n \times n}$ is a positive definite matrix, the $A$-norm of a vector $x \in \mathbb{R}^n$ is defined as $\|x\|_A := \sqrt{x^\top A x}$.
When $A = I$, the $A$-norm reduces to the standard $\ell_2$-norm, and we write $\Proj_{\mathcal{S}}(q) := \Proj_{\mathcal{S}}^I(q)$.
\end{definition}

\begin{definition}[Directional Derivative \citep{bonnans2013perturbation}]
\label{def: directional derivative}
Let $h:  \mathcal{Q} \subseteq \mathbb{R}^n \rightarrow \mathbb{R}$ and let $d \in \mathbb{R}^n$. The first-order \emph{directional derivative} of $h$ at $q\in \mathcal{Q}$ along $d$ is 
\begin{align*}
    D_d h(q) = \lim_{r \downarrow 0} \frac{h(q + r d) - h(q)}{r}.
\end{align*}
If $D_d h(q)$ exists and is finite for all $d$, we say $h$ is \emph{directional differential} in direction $d$ at $q$. Additionally, the \emph{second-order directional derivative} of $h$ at $q\in \mathcal{Q}$ along $d$ is
\begin{align*}
    D_{dd}^2 h(q) = \lim_{r \downarrow 0} \frac{D_d(h(q+rd))-D_d(h(q))}{r}.
\end{align*}
\end{definition}

\begin{definition}[Tangent Cone \& Critical Cone]
\label{def:tangent_critical_cone}
Let $q \in \mathcal{S}\subseteq \mathbb{R}^n$. The \emph{tangent cone} at $q$ with respect to $\mathcal{S}$ is 
\begin{align*}
    \mathcal{T}_{\mathcal{S}}(q) := \big\{ d \in \mathbb{R}^n | \exists\, t_k \downarrow 0,\, \exists\, d_k \to d \text{ such that } q + t_k d_k \in \mathcal{S}, \forall k \big\}.
\end{align*}
Let $h: \mathcal{Q}\subseteq \mathbb{R}^n \rightarrow \mathbb{R}$ be differentiable on $\mathcal{S} \subseteq \mathcal{Q}$, the \emph{critical cone} of $h$ at $q$ with respect to $\mathcal{S}$ is
\begin{align*}
    \mathcal{C}_{\mathcal{S}}(q) := \big\{ d \in \mathcal{T}_{\mathcal{S}}(q) | \langle \nabla h(q), d \rangle = 0 \big\}.
\end{align*}
\end{definition}
\begin{remark}
If $q \in \mathrm{int}(\mathcal{S})$, then $\mathcal{T}_{\mathcal{S}}(q) = \mathbb{R}^n$.
\end{remark}

\begin{definition}[Generalized (Proximal) Gradient]
\label{def: generalized gradient}
Let $h:\mathcal{Q} \rightarrow \mathbb{R}$ be differentiable on $\mathcal{S}\subseteq \mathcal{Q}$, and $\eta>0$ be some small scalar. We say the proximal gradient of $h(q)$ on $\mathcal{S}$ is
\begin{align*}
    G_{h,\mathcal{S}}(q)=\frac{1}{\eta}\Big(q-\Proj_{\mathcal{S}}(q-\eta \nabla h(q))\Big),\quad \forall q\in \mathcal{S}.
\end{align*}
\end{definition}

\begin{definition}[Proximal PL, EB, and QG]
Let $h:\mathcal{Q} \rightarrow \mathbb{R}$ be differentiable on $\mathcal{S}\subseteq \mathcal{Q}$. We say $h(q)$ satisfies \textit{proximal $\mu_g$-Polyak-Łojasiewicz (PL) condition on $\mathcal{S}$} if
\begin{align*}
    h(q)-\min_{q\in \mathcal{S}}h(q) \leq \frac{1}{2\mu_g } G_{h,\mathcal{S}}^2(q),\quad \forall q \in \mathcal{S}.
\end{align*}
Denote $S_h^* =\arg\min_{q\in \mathcal{S}}h(q)$ as the solution set. We say $h(q)$ satisfies \textit{proximal $\mu_g$-Error Bound (EB) condition on $\mathcal{S}$} if
\begin{align*}
    d_{S_h^*}(q) \leq \frac{1}{\mu_g} G_{h,\mathcal{S}}(q) ,\quad \forall q \in \mathcal{S}.
\end{align*}
We say $h(q)$ satisfies \textit{$\mu_g$-Quadratic-Growth (QG) condition on $\mathcal{S}$} if
\begin{align*}
    d_{S_h^*}(q)^2 \leq \frac{2}{\mu_g} (h(q)-\min_{q\in \mathcal{S}}h(q)),\quad \forall q \in \mathcal{S}.
\end{align*}
\end{definition}

\begin{lemma}[{Strongly Convexity and Proximal PL \cite[Appendix E]{karimi2016linear}}]
\label{lemma: SC PL}
If $h:\mathcal{Q} \rightarrow \mathbb{R}$ is $\mu_g$-strongly convex on $\mathcal{S}\subseteq \mathcal{Q}$ and $\mathcal{S}$ is a convex set, then $h$ satisfies proximal $\mu_g$-PL condition on $\mathcal{S}$.
\end{lemma}

\begin{lemma}[{Equivalence of PL, EB, and QG condition \cite[Theorem 3.1]{liao2024error}}]
\label{lemma: equiv of PL EB QG}
    Suppose $h:\mathcal{Q} \rightarrow \mathbb{R}$ is differentiable and $l_{h,1}$-smooth on convex domain $\mathcal{S}\subseteq \mathcal{Q}$. Then if $h(q)$ satisfies proximal $\mu_g$-PL condition on $\mathcal{S}$, $h(q)$ satisfies proximal $\mu_g$-EB and $\mu_g/2$-QG condition on $\mathcal{S}$.
\end{lemma}

\begin{definition}[LICQ]
    For any fixed $x$, consider the lower-level feasible set $\mathcal{Y}(x)=\{ y\in \mathcal{Y}:c(x,y)\leq 0\}$ where $\mathcal{Y}$ can be locally characterized by $\phi(y) \le 0$ near $y_g^*(x)$ and $c$ is differentiable.
    We say that $\mathcal{Y}(x)$ satisfies the \emph{Linear Independence Constraint Qualification (LICQ)} at $y_g^*(x)$ if the set of gradients of all active constraints
    \[
    \{ \nabla_y \phi_j(y_g^*(x)) : j \in \mathcal{I}_{\mathcal{Y}}^0(x) \} \cup 
    \{ \nabla_y c_i(x, y_g^*(x)) : i \in \mathcal{I}_c^0(x) \}
    \]
    is linearly independent, where 
    \[
    \mathcal{I}_{\mathcal{Y}}^0(x) := \{ j : \phi_j(y_g^*(x)) = 0 \} \quad \text{and} \quad
    \mathcal{I}_c^0(x) := \{ i : c_i(x, y_g^*(x)) = 0 \}
    \]
    denote the active constraints.
\end{definition}

\begin{lemma}[First-Order Variational Inequality \citep{rockafellar1998variational}]
\label{lem:first_order_VI}
Consider the constrained optimization problem $\min_{q \in \mathcal{S}} h(q)$,
where $h: \mathcal{Q} \rightarrow \mathbb{R}$ is differentiable on $\mathcal{S} \subseteq \mathcal{Q}$. 
A point $q^* \in \mathcal{S}$ is a first-order stationary point of this problem if and only if it satisfies the first-order variational inequality condition
\begin{align*}
    \langle \nabla h(q^*), q - q^* \rangle \geq 0, \quad \forall q \in \mathcal{S},
\end{align*}
or equivalently,
\begin{align*}
    \langle \nabla h(q^*), d \rangle \geq 0, \quad \forall d \in \mathcal{T}_{\mathcal{S}}(q^*).
\end{align*}
\end{lemma}
\begin{remark}
If $q^* \in \mathrm{int}(\mathcal{S})$, then Lemma~\ref{lem:first_order_VI} reduces to the unconstrained first-order condition $\nabla h(q^*) = 0$. Additionally, there is $\mathcal{T}_{\mathcal{S}}(q^*)=\mathcal{C}_{\mathcal{S}}(q^*)=\mathbb{R}^n$.
\end{remark}

\begin{definition}[Hausdorff-Lipschitz Continuity]
\label{def:hausdorff_lipschitz}
Let $K: \mathcal{Q} \rightrightarrows \mathbb{R}^n$ be a set-valued mapping with nonempty images. We say that $K$ is \emph{Hausdorff-Lipschitz continuous} on $\mathcal{Q}$ if there exists a constant $L \geq 0$ such that for any $q_1, q_2 \in \mathcal{X}$,
\begin{align*}
    d_H\big(S(q_1), S(q_2)\big) \leq L \|q_1 - q_2\|,
\end{align*}
where $d_H(\cdot,\cdot)$ denotes the Hausdorff distance:
\begin{align*}
    d_H(A,B) := \max\Bigg\{
        \sup_{a \in A} \inf_{b \in B} \|a-b\|, 
        \sup_{b \in B} \inf_{a \in A} \|a-b\| 
    \Bigg\}, \quad A,B \subset \mathbb{R}^n.
\end{align*}
\end{definition}

\begin{definition}[Smooth Boundary \citep{lee2003smooth}]
\label{def: smooth boundary}
Let $\mathcal{S} \subseteq \mathbb{R}^{n}$ be a closed set. We say that $\mathcal{S}$ has a \textit{smooth boundary} at $q \in \operatorname{bd}(\mathcal{S})$ if there exists a neighborhood $U \subseteq \mathbb{R}^n$ of $q$ and a continuously differentiable function 
$\varphi: U \to \mathbb{R}$ such that
\begin{align*}
    \mathcal{S} \cap U &= \{ z \in U : \varphi(z) \leq 0 \},\quad
    \operatorname{bd}(\mathcal{S}) \cap U = \{ z \in U : \varphi(z) = 0 \},\quad \text{and}\quad 
    \nabla \varphi(q) \neq 0.
\end{align*}
In this case, the tangent space with respect to $\text{bd}(\mathcal{S})$ at $q\in \text{bd}(\mathcal{S})$ is
\begin{align*}
    \mathcal{T}_{\text{bd}(\mathcal{S})}(q) := \{ d \in \mathbb{R}^n : \nabla \varphi(q)^\top d = 0 \}.
\end{align*}
\end{definition}

\begin{lemma} [{\citep[Lemma 3.1]{bubeck2015convex}}] 
\label{lemma: projection gradient update inquality}
Suppose $\mathcal{S}\subseteq \mathbb{R}^n$ is convex, closed, and nonempty. For any $q_1 \in \mathbb{R}^n$ and any $q_2 \in \mathcal{S}$, it follows that 
\begin{align*}
    \langle \Proj_{\mathcal{S}}(q_1)-q_2, \Proj_{\mathcal{S}}(q_1)- q_1 \rangle  \leq 0.
\end{align*}
In this way, take $q_1 = q_3-\eta g$ for any $q_3\in \mathcal{S}$, and denote $q_3^{\eta,g} = \Proj_{\mathcal{S}}(q_3-\eta g)$ as a projected gradient update in direction $g$ with stepsize $\eta$, we have,
\begin{align*}
    \langle g, q_3^{\eta,g}  -q_2 \rangle \leq -\frac{1}{\eta} \langle q_3^{\eta,g} -q_2, q_3^{\eta,g} -q_3\rangle, \quad \forall q_2,q_3 \in \mathcal{S}.
\end{align*}
\end{lemma}

\section{Preliminaries of Penalty-based Bilevel Methods}

In this section, we introduce preliminaries related to the penalty reformulation $F_\gamma(x)$ for BLO. 
Before proceeding, we outline the assumptions of focus in this paper. 

\begin{assumption}[Upper level]
\label{assumption: UL} Assume differentiable $f: \mathbb{R}^{d_x}\times \mathbb{R}^{d_y}\rightarrow \mathbb{R}$ is 
\refstepcounter{mycounter}
\label{ass: UL y Lipschitz}
(1) $l_{f,0}$-Lipschitz in $y\in \mathcal{Y}$, 
\refstepcounter{mycounter}
\label{ass: UL smooth}
(2) $l_{F_\gamma,1}$-smooth in $(x,y)\in \mathcal{X}\times \mathcal{Y}$, and
\refstepcounter{mycounter}
\label{ass: UL x local Lipchitz}
(3) Lipschitz in $x\in \mathcal{X}$.
\end{assumption}

\begin{assumption}[Lower level]
\label{assumption: LL}
Assume differentiable $g: \mathbb{R}^{d_x}\times \mathbb{R}^{d_y}\rightarrow \mathbb{R}$ is 
\setcounter{mycounter}{0}
\refstepcounter{mycounter}
\label{ass: LL SC}
(1) $\mu_g$-strongly-convex in $y\in \mathcal{Y}$, 
\refstepcounter{mycounter}
\label{ass: LL smooth}
and (2) $l_{g,1}$-smooth in $(x,y)\in \mathcal{X}\times \mathcal{Y}$, and
\refstepcounter{mycounter}
\label{ass: LL x local Lipchitz}
(3) Lipschitz in $x\in \mathcal{X}$.
\end{assumption}

\begin{assumption}[Constraints] 
\label{assumption: constraint}
Assume 
\setcounter{mycounter}{0}
\refstepcounter{mycounter}
\label{ass: constraint domain}
(1) $\mathcal{X}\subseteq \mathbb{R}^{d_x}$ and $\mathcal{Y}(x) \subseteq \mathbb{R}^{d_y}$ are non-empty, closed, and convex, 
\refstepcounter{mycounter}
\label{ass: constraint inequality}
(2) differentiable $c: \mathbb{R}^{d_x}\times \mathbb{R}^{d_y}\rightarrow \mathbb{R}^{d_c}$ is convex in $y\in \mathcal{Y}$, $l_{c,1}$-smooth in $(x,y)\in \mathcal{X}\times \mathcal{Y}$, 
\refstepcounter{mycounter}
\label{ass: constraint x local Lipschitz}
(3) and Lipschitz, and
\refstepcounter{mycounter}
\label{ass: licq}
(4) For any $x\in \mathcal{X}$, there exists $d_c'<\infty$ and differentiable $\varphi: \mathbb{R}^{d_y}\rightarrow \mathbb{R}^{d_c'}$ such that $\mathcal{Y} \cap U_x = \{ y \in U_x : \varphi(y) \le 0\}$ for a neighborhood $U_x$ of $y_g^*(x)$, and
$\mathcal{Y}(x)$ satisfies the LICQ in the neighborhood of optimal points $y_g^*(x)$.
\end{assumption}

The Lipschitz continuity and smoothness conditions for $f$, $g$, and $c$ in Assumptions \ref{assumption: UL}, \ref{assumption: LL}, and \ref{assumption: constraint} are standard \citep{ghadimi2018approximation,hong2020two,kwon2023fully,ji2021bilevel,chen2021closing,jiang2024primal}. The strong convexity of the lower-level problem is conventional \citep{ghadimi2018approximation,chen2023near,jiang2024primal,chen2021closing} and still presents challenges due to the imposed constraints. 
Moreover, assuming $c(x,y)$ convex in $y$ is mild and traditional for coupled constrained BLO problems \citep{yao2024constrained,khanduri2023linearly,xiao2022alternating,jiang2024primal}. The convexity and closure of $\mathcal{X}$ and $\mathcal{Y}(x)$ are standard. The $\mathcal{Y}$ local characterization condition is mild and can be satisfied by polyhedral sets, boxes, simplices, or other closed convex sets, even when $y_g^*(x)$ lies on edges or corners where the boundary may be nonsmooth, and the LICQ is a common assumption in coupled constrained BLO \citep{kwon2023penalty,xu2023efficient,jiang2024primal}.


With these conditions, it can be proved that $F_\gamma(x)$ is a good approximation to $\phi(x)$ in \eqref{eq: original problem 1} with distance controlled by $\gamma^{-1}$ so that solving $F_\gamma(x)$ is equivalent to solving to find $\epsilon$-suboptimal $\phi(x)$.
\begin{lemma} 
\label{lemma: distance of yg ygam} 
Suppose Assumption \ref{assumption: UL}.\ref{ass: UL y Lipschitz}--\ref{ass: UL smooth}, \ref{assumption: LL}.\ref{ass: LL SC}, and \ref{assumption: constraint}.\ref{ass: constraint domain}--\ref{ass: constraint inequality} hold. The $\epsilon$-suboptimal solutions in distance square metric for $\min_{x\in \mathcal{X}}F_\gamma(x)$ in \eqref{eq: F gam function} are $\epsilon$-suboptimal local solutions for the 
$\epsilon$-approximation problem \eqref{eq: epsilon app prob}
with $\gamma=\Oc(\epsilon^{-0.5})$ and $\gamma \geq \frac{l_{f,1}}{\mu_g}$. Additionally, \begin{align*} 
\|\phi (x)-F_\gamma(x)\| = \Oc(l_{f,0}^2\mu_g^{-1} \gamma^{-1}),\quad \|y_g^*(x)-y_\gamma^*(x) \| = \Theta(l_{f,0}\mu_g^{-1} \gamma^{-1}), 
\end{align*} 
where $S_g(x)=\{y_g^*(x)\}$ and $S_\gamma(x)=\{y_\gamma^*(x)\}$ are singletons with $S_g(x)$ defined in \eqref{eq: lower-level problem} and
\begin{align}
    S_\gamma(x):= \arg\min_{y\in \mathcal{Y}(x)} \gamma^{-1}f(x,y)+g(x,y). \label{eq: S gamma}
\end{align}
\end{lemma}
\begin{remark}[{Lipschitz-Continuity of $S_g(x)$ \cite[Lemma 6]{shen2023penalty}}]
\label{remark: Lipschitz of S(x)}
Suppose $\mathcal{Y}(x)=\tilde{\mathcal{Y}}$ is closed, convex and non-empty uncoupled LL constraint, $g(x,y)$ satisfies proximal $\mu$-PL condition in $y\in \mathcal{Y}$, and $\nabla_y g(x,y)$ is Lipschitz in $x$. Then there exist a constant $ L_y^g \geq 0$ such that for any $x_1,x_2\in \mathcal{X}$ and $y_1 \in S_g(x_1)=\arg\min_{y\in \mathcal{Y}}g(x_1,y)$, there exists $y_2 \in S_g(x_1)=\arg\min_{y\in \mathcal{Y}}g(x_2,y)$ such that
\begin{align*}
\| y_1-y_2\|\leq L_y^g\| x_1-x_2\|.
\end{align*}
In other words, $S_g(x)$ is Hausdorff-Lipschitz. 
\end{remark}
\begin{remark}
\label{remark: generalizaed dist}
Suppose $\mathcal{Y}(x)=\tilde{\mathcal{Y}}$ is closed, convex and non-empty uncoupled LL constraint, $f$ is $l_{f,0}$-Lipschitz and $g$ is continuous in $y\in \tilde{\mathcal{Y}}$, and there exists a finite $\gamma^*>0$ such that, for all $c\in [0,1/\gamma^*]$, the function $cf(x,y)+g(x,y)$ satisfies the proximal $\mu$-PL condition in $y\in \mathcal{Y}(x)$ for some $\mu>0$ for all $x\in \mathcal{X}$.
Then, there is
\begin{align}
    \|\phi(x)-F_\gamma(x)\|\leq &\Oc(l_{f,0}^2 \mu^{-1}\gamma^{-1}),\quad d_{S_\gamma(x)}(y_g^*(x)) ,d_{S_g(x)}(y_\gamma^*(x)) \leq  \Theta(l_{f,0}\mu^{-1}\gamma^{-1}),
\end{align}
and the $\epsilon$-suboptimal solutions equivalence between \eqref{eq: epsilon app prob} and \eqref{eq: joint penalty problem} also holds.
\end{remark}
Lemma~\ref{lemma: distance of yg ygam} and Remark~\ref{remark: generalizaed dist} directly follows \citep[Theorem 2]{shen2023penalty} and \citep[Theorem 1]{jiang2024primal}. It inspires a choice of $\gamma=\Oc(\epsilon^{-0.5}) \geq \frac{l_{f,1}}{\mu_g}$.
Here, $\gamma^{-1} f + g$ is strongly convex in $y$ when $\gamma \geq \frac{l_{f,1}}{\mu_g}$, since $l_{F_\gamma,1}$-smoothness provides a lower bound on the negative curvature of $f$. In the uncoupled constrained setting where $\mathcal{Y}(x)=\tilde{\mathcal{Y}}$, the proximal PL condition is also commonly assumed in the first-order BLO literature \citep{ghadimi2018approximation,hong2020two,ji2021bilevel,chen2021closing,dagreou2022framework,kwon2023penalty,shen2023penalty,chen2024finding}. It is a weaker condition and is implied by the smoothness of $f$ and the strongly convexity of $g(x,\cdot)$.

Moreover, $F_\gamma(x) = \gamma (v_\gamma(x)-v(x))$ features favorable properties such as differentiability and smoothness, so does the value functions. 

\begin{lemma}[Derivative of $v(x)$ {\citep[Lemma 2]{jiang2024primal}}]
\label{lemma: gradient of v}
Suppose Assumptions \ref{assumption: UL}, \ref{assumption: LL}, \ref{assumption: constraint} hold.
The  Lagrangian multiplier $\lambda_g^*(x) =\arg\max_{\lambda\in \mathbb{R}^{d_c}_+}\{\min_{y_g\in \mathcal{Y}}g(x,y_g)+\langle \lambda, c(x,y_g)\rangle\}$ is unique.
And for $\mathcal{Y}(x)=\{y\in \mathcal{Y}:c(x,y)\leq 0\}$, the value function $v(x)$ defined in \eqref{eq: joint penalty problem} is $l_{v,1}$-smooth and
\begin{align}
\nabla v(x) =  \nabla_x g(x,y_g^*(x)) + \langle \lambda_g^*(x),\nabla_x c(x,y_g^*(x)) \rangle. \label{eq: value function gradient generalized form}
\end{align}
Similarly, define
\begin{align}
    v_\gamma(x)=\min_{y_\gamma\in \mathcal{Y}(x)}\gamma^{-1}f(x,y_\gamma)+g(x,y_\gamma) \label{eq: v gamma}
\end{align}
The function $v_\gamma(x)$ is $l_{v_\gamma,1}$-smooth and
\begin{align}
    \nabla v_\gamma(x) = \gamma^{-1} \nabla_x f(x,y_\gamma^*(x))+g(x,y_\gamma^*(x))+ \langle \lambda_\gamma^*(x),\nabla_x c(x,y_\gamma^*(x)) \rangle,
\end{align}
where $(\lambda_\gamma^*(x),y_\gamma^*(x)) =\arg\max_{\lambda\in \mathbb{R}^{d_c}_+}\min_{y_\gamma\in \mathcal{Y}}\{\gamma^{-1} f(x,y_\gamma)+g(x,y_\gamma)+\langle \lambda, c(x,y_\gamma)\rangle\}$ are unique. Moreover, we have 
\begin{align}
    \nabla F_\gamma(x) = \gamma(\nabla v_\gamma(x) - \nabla v(x) ). \label{eq: nabla F gam}
\end{align}
\end{lemma}
\begin{remark}[{\citep[Proposition 4]{shen2023penalty}}]
\label{remark: gradient of v}
    When $\mathcal{Y}(x)=\tilde{\mathcal{Y}}$ is uncoupled with $x$, assumptions can be relaxed to $\tilde{\mathcal{Y}}$ being closed and convex, and $\gamma^{-1}f(x,y)+g(x,y)$ being jointly-smooth and proximal PL in $\tilde{\mathcal{Y}}$. In this case, the Lagrangian multiplier is no longer involved and $\nabla v_\gamma(x) =  \gamma^{-1}\nabla_x f(x,y_\gamma^*(x))+\nabla_x  g(x,y_\gamma^*(x))$ for any $y_\gamma \in S_\gamma(x)$.
\end{remark}

The lemma \ref{lemma: gradient of v} is the cornerstone of the implementation of a gradient-based algorithm to solve the reformulation $F_\gamma(x)$ or $\tilde{F}_\gamma(x,y)$, such as in \citep{kwon2023fully,shen2023penalty,jiang2024primal}.

\section{Improved Convergence Rate under Uncoupled Constraints}
\label{sec: Improved Convergence Rate under Uncoupled Constraints}

In this section, we start by considering the uncoupled constraint $\mathcal{Y}(x)=\tilde{\mathcal{Y}}$ independent from $x$. After revisiting the analysis of the unconstrained case (Sec.~\ref{sec: Existing work smoothness of F gam}), we will provide a dedicated analysis of the smoothness of $F_\gamma(x)$ for the uncoupled constraint in Sec.~\ref{sec: smoothness of F gam}. In Sec.~\ref{sec: Improved PBGD method}, we revisited ALT-PBGD and demonstrated that it is an optimal algorithm that matches the convergence complexity of the gradient descent algorithm. We will discuss the general coupled constraint case in Sec.~\ref{sec: BLOCC}.

\subsection{Existing work: tighter smoothness estimate of 
            \texorpdfstring{$F_\gamma(x)$}{Fγ(x)} for unconstrained BLO}
\label{sec: Existing work smoothness of F gam}

\begin{wrapfigure}{r}{0.4\textwidth}
    \centering
    \vspace{-0.6cm}
    \includegraphics[width=0.38\textwidth]{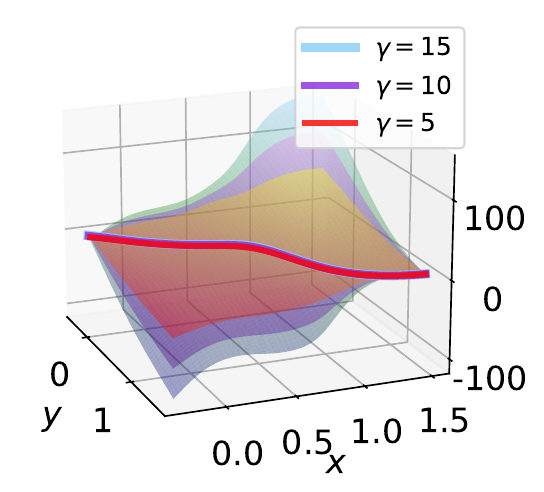}
    \caption{$\nabla_x \tilde{F}_\gamma(x, y)$ (spheres) vs. $\nabla F_\gamma(x)$ (lines) with different $\gamma$ for the problem in Examlpe~\ref{example:toy_example_1}.}
    \label{fig:gradient-plot}
    \vspace{-0.8cm}
\end{wrapfigure}

Existing literature \citep{shen2023penalty,kwon2023penalty} investigates the joint minimization of $(x, y)$ for $\tilde{F}_\gamma(x, y)$ in \eqref{eq: joint penalty problem}, whose smoothness modulus is of order $\Theta(\gamma)$. This leads to a prior estimate of the smoothness modulus for $F_\gamma(x)$ as $l_{F_\gamma,1} = \Oc(\gamma)$.
However, empirical evidence, e.g. Fig.~\ref{fig:gradient-plot}, shows that although $\nabla_x \tilde{F}_\gamma (x, y)$ will be scaled up by $\gamma$, $\nabla F_\gamma(x)$ remains at a constant value. 
This motivates a re-examination of the smoothness modulus of $F_\gamma(x)$. 

For the \textit{unconstrained} case, this has been achieved by deriving a closed-form expression for $\nabla^2 F_\gamma(x)$ and relating it to the unscaled $\nabla^2 F(x)$ \citep{chen2024finding}. Specifically, in the unconstrained case where $\mathcal{X}=\mathbb{R}^{d_x},\mathcal{Y}(x)=\mathbb{R}^{d_y}$ in \eqref{eq: original problem 1}, the lower-level stationarity $\nabla_y g(x,y_g^*(x))=0$ gives
\begin{align}
0=&\lim_{r\downarrow 0}\frac{1}{r} \big( \nabla_y g(x+rd,y_g^*(x+rd)) - \nabla_y g(x,y_g^*(x))) \big)\nonumber\\
= & \nabla_{xy}g(x,y_g^*(x))^\top d + \nabla_{yy}g(x,y_g^*(x))\frac{\partial}{\partial x} y_g^*(x) d,
\label{eq:stationarity_diff_equals_zero}
\end{align}
where the second equality is followed by Taylor's expansion and $y_g^*(x)$ being Lipschitz in $x$ \citep{shen2023penalty,kwon2023penalty,jiang2024primal,chen2025foops}.
Therefore, prior arts \citep{arbel2022nonconvex,kwon2023penalty,xiao2023} obtain
\begin{align}
\frac{\partial}{\partial x} y_g^*(x) \in -\nabla_{yy}g(x,y_g^*(x))^{\dagger}\nabla_{yx}g(x,y_g^*(x)) + \text{Ker}(\nabla_{yy}g(x,y_g^*(x))). \label{eq: partial yg(x)}
\end{align}
Moreover, the expression of $\nabla v(x)$ in Lemma~\ref{lemma: gradient of v} enables finding $\nabla^2 v(x)$: 
\begin{align}
    \nabla^2 v(x)= &\nabla_{xx} g(x,y_g^*(x)) -\nabla_{xy} g(x,y_g^*(x))\nabla_{yy}g(x,y_g^*(x))^{\dagger}\nabla_{yx}g(x,y_g^*(x)), \label{eq: hassien of v}
\end{align}
where the kernel term in \eqref{eq: partial yg(x)} is canceled out as $\nabla_{xy} g(x,y_g^*(x)) \subseteq \text{Ker}(\nabla_{yy}g(x,y_g^*(x)))$ \citep{kwon2023penalty,chen2024finding}.

Similarly, the Hessian of $v_\gamma(x)$ in \eqref{eq: v gamma} 
can be established. According to \eqref{eq: nabla F gam}, $\nabla^2 F_\gamma(x)=\gamma(\nabla^2 v_\gamma(x)-\nabla^2 v(x))$ and since $\|\nabla^2 v_\gamma(x)-\nabla^2 v(x)\|=\Oc(\|y_g^*(x)-y_\gamma^*(x)\|)=\Oc(\gamma^{-1})$ as per Lemma~\ref{lemma: distance of yg ygam}, the smoothness constant for $F_\gamma(x)$ is $l_{F_\gamma,1}=\max_x\|\nabla^2 F_{\gamma}(x)\|=\Oc(1)$.

\subsection{Tighter smoothness estimate of \texorpdfstring{$F_\gamma(x)$}{Fγ(x)} for uncoupled constrained BLO}

\label{sec: smoothness of F gam}

\begin{wrapfigure}{r}{0.45\textwidth}
    \centering
    \vspace{-0.2cm}
    \includegraphics[width=0.4\textwidth]{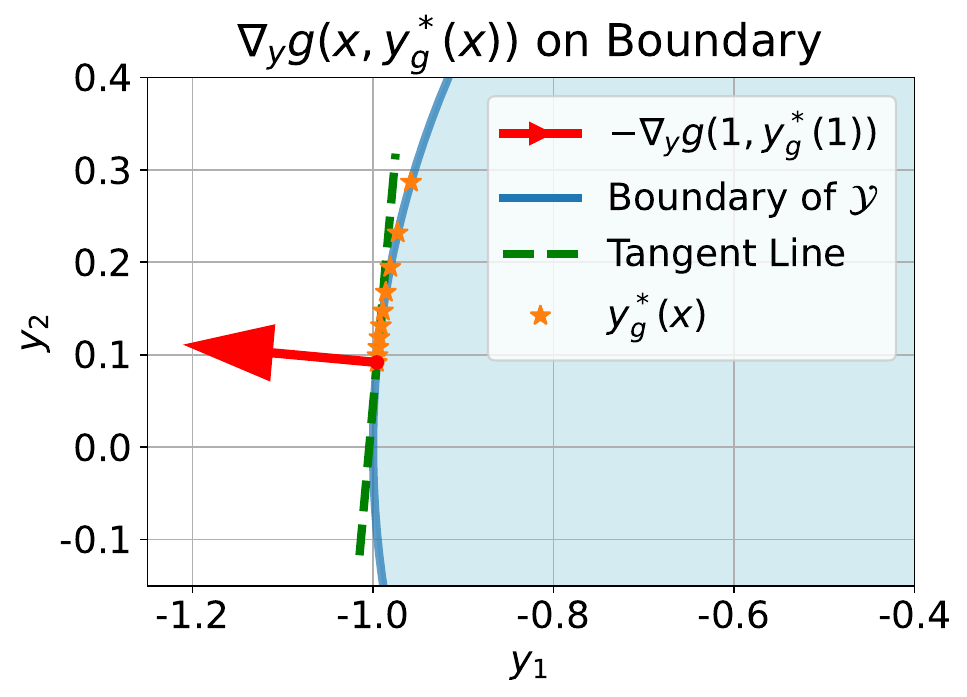}
        \caption{Implicit gradient on the boundary. The LL objective is given by $g(x,y)=xy_1+(1+x)y_2^2-y_2$, with LL domain $\mathcal{Y}$ being the unit ball. The orange stars depict the trajectory of $y_g^*(x)$ as $x$ varies from $1$ to $3$.}
        \label{fig:geographic_y_gradient}
    \vspace{-0.8cm}
\end{wrapfigure}

\textbf{Technical challenges for constrained BLO.}
For unconstrained BLO, a key step in establishing the tight $\Oc(1)$ smoothness modulus estimate for $F_\gamma(x)$ is deriving $\nabla^2 v(x)$ based on finding implicit gradient $\frac{\partial}{\partial x} y_g^*(x)$ in \eqref{eq: partial yg(x)}. In the unconstrained case, the implicit gradient can be characterized using the stationary condition $\nabla_y g(x, y_g^*(x)) = 0$. In the constrained setting, however, this condition typically does not hold due to the non-smoothness induced by the constraints, making it significantly more difficult to derive a closed-form expression for $\frac{\partial}{\partial x} y_g^*(x)$. Moreover, $y_g^*(x)$ may not admit a closed-form gradient and $F_\gamma(x)$ or $v(x)$ may fail to be twice differentiable.

To address this, we alternatively look into $D_d(y_g^*(x))=\lim_{r\downarrow 0} \frac{y_{g}^*(x+rd)-y_{g}^*(x)}{r}$.
\begin{lemma}
\label{lemma: gradient equals zero generalization}
Suppose $\mathcal{Y}(x)=\tilde{\mathcal{Y}}$ is closed, convex and non-empty uncoupled LL constraint, and Assumption~\ref{assumption: UL}.\ref{ass: UL y Lipschitz}--\ref{ass: UL smooth}, \ref{assumption: LL}.\ref{ass: LL SC}--\ref{ass: LL smooth} hold.
Then,
\begin{align*}
    \Big\langle \nabla_y g(x,y_g^*(x)),\lim_{r\downarrow 0} \frac{y_g^*(x+rd)-y_g^*(x)}{r} \Big\rangle = 0.
\end{align*}
\end{lemma}
\begin{proof}
    The gradient of the value function $\nabla v(x) =  \nabla_x g(x,y_g^*(x))$
    is an established result as per Lemma~\ref{lemma: gradient of v} and Remark~\ref{remark: gradient of v}. In this way, for any unit directions $d\in \mathbb{R}^{d_x}$ and $r>0$, there exists $y_g^*(x+rd)$ uniquely defined such that the directional derivative $v(x)$ in the unit direction $d$ is
    \begin{align*}
        D_d(v(x))\stackrel{(a)}{=}& \langle \nabla v(x), d\rangle = \langle \nabla_x g(x,y_g^*(x)), d\rangle \nonumber \\
        \stackrel{(b)}{=} & \lim_{r\downarrow 0} \frac{1}{r}(g(x+rd,y_g^*(x+rd)) -g(x,y_g^*(x)))  \nonumber \\
        \stackrel{(c)}{=} & \langle \nabla_x g(x,y_g^*(x)), d\rangle +  \langle \nabla_y g(x,y_g^*(x)),\lim_{r\downarrow 0} \frac{y_g^*(x+rd)-y_g^*(x)}{r} \rangle
    \end{align*}
    where (a) follows the $\nabla v(x) =\nabla_x g(x,y_g^*(x))$; (b) is the definition of directional derivative; and (c) applies Taylor expansion and $\| y_g^*(x+rd)-y_g^*(x)\|= \Oc(rd)$:
    \begin{align*}
       g(x+rd,y_g^*(x+rd))
       =& g(x,y_g^*(x)) + \langle\nabla_x g(x,y_g^*(x)), rd\rangle\\
       & + \langle\nabla_y g(x,y_g^*(x)),\, y_g^*(x+rd)-y_g^*(x)\rangle + \mathcal{O}(r^2).
    \end{align*}
    In this way, we can conclude the proof.
\end{proof}

Building on this, we analyze the directional derivative of the implicit gradient $y_g^*(x)$ using the first-order variational inequality condition. As a preview, we summarize the differentiability properties of solution mapping and value functions in Table \ref{tab:differentiability-summary}. 

\begin{lemma}
\label{lemma: implicit gradient}
Suppose $\mathcal{Y}(x)=\tilde{\mathcal{Y}}$ is closed, convex and non-empty uncoupled LL constraint, and Assumption \ref{assumption: UL}.\ref{ass: UL y Lipschitz}--\ref{ass: UL smooth}, \ref{assumption: LL}.\ref{ass: LL SC}--\ref{ass: LL smooth} hold. The directional derivative of $y^*_g(x)$ in direction $d$ is the projection of $-(\nabla_{yy} g(x,y_g^*(x)))^{-1} \nabla_{yx} g(x,y_g^*(x)) d$ onto the critical cone $C_\mathcal{Y}(y^*_g(x))$ in the $\nabla_{yy} g(x,y_g^*(x))$-norm, i.e. 
\begin{align}
    D_d(y^*_g(x)) 
    = & \Proj_{C_\mathcal{Y}(y^*_g(x))}^{\nabla_{yy} g(x,y^*_g(x))} \Big( - (\nabla_{yy} g(x,y^*_g(x)))^{-1} \nabla_{yx}g(x,y^*_g(x)) d \Big). \label{eq: implicit directional derivative}
\end{align}
\end{lemma}
\begin{proof}
By Taylor's expansion, there is
\begin{align*}
    y^*_g(x+rd)= &y^*_g(x)+rD_d(y^*_g(x)) +\Oc(r^2),\quad \text{and}\\
    \nabla_y g(x+rd,y^*_g(x+rd)) = & \nabla_y g(x,y^*_g(x))+r\nabla_{yy} g(x,y^*_g(x))D_d(y^*_g(x)\\
    & + r\nabla_{yx} g(x,y^*_g(x))d + \Oc(r^2),
\end{align*}

According to the first-order variational inequality in Lemma~\ref{lem:first_order_VI}, there is
\begin{align*}
\langle \nabla_y g(x,y^*_g(x)), y- y^*_g(x)\rangle \geq 0, \quad \forall y \in \tilde{\mathcal{Y}}.
\end{align*}
For arbitrary $w\in \mathcal{C}_{\tilde{\mathcal{Y}}}(y_g^*(x))$ where $\mathcal{C}_{\tilde{\mathcal{Y}}}(y_g^*(x))$ is the critical cone, choose $y =y_g^*(x)+r w'$, where $w'\rightarrow w$ as $r\downarrow 0$. There exists sufficiently small $r$ such that $y =y_g^*(x)+r w'\in \tilde{\mathcal{Y}}$ following the definition of tangent cone and the fact that $\mathcal{C}_{\mathcal{Y}}(y_g^*(x))\subseteq \mathcal{T}_{\mathcal{T}}(y_g^*(x))$. In this way, for sufficiently small $r$, there is 
\begin{align*}
    \langle \nabla_y g(x+rd,y^*_g(x)+rD_d(y_g^*(x))), (y_g^*(x)+rw)- (y^*_g(x)+rD_d(y_g^*(x)))\rangle & \geq 0, \\
    \Leftrightarrow \langle \nabla_y g(x,y^*_g(x))+r\nabla_{yy} g(x,y^*_g(x))D_d(y^*_g(x)+r\nabla_{yx} g(x,y^*_g(x))d, w- D_d(y_g^*(x))\rangle& \geq 0
\end{align*}
Moreover, we know that $\langle \nabla_y g(x,y^*_g(x)), D_d(y_g^*(x)) \rangle =0$ from Lemma~\ref{lemma: gradient equals zero generalization} and that $\langle \nabla_y g(x,y^*_g(x)), w \rangle = 0$ for all $w \in \mathcal{C}_{\tilde{\mathcal{Y}}}(y_g^*(x))$, following the definition of critical cone of $g(x,\cdot)$ with respect to $\tilde{\mathcal{Y}}$.
In this way, directional derivative $D_d(y^*_g(x))$ is characterized by the variational inequality:
\begin{align*}
    \langle \nabla_{yy} g(x,y_g^*(x)) D_d(y^*_g(x)) + \nabla_{yx} g(x,y_g^*(x))d, ~ w- D_d(y^*_g(x)) \rangle \geq 0 \quad\forall w\in \mathcal{C}_{\tilde{\mathcal{Y}}}(y^*_g(x)).
\end{align*}
As $D_d(y^*_g(x))\in  \mathcal{C}_{\tilde{\mathcal{Y}}}(y^*_g(x))$ according to Lemma~\ref{lemma: gradient equals zero generalization}, the variational inequality is equivalent to
\begin{align}
    D_d(y_g^*(x)) = \arg\min_{\eta \in \mathcal{C}_{\tilde{\mathcal{Y}}}(y_g^*(x))} \frac{1}{2} \eta^\top \nabla_{yy} g(x,y_g^*(x)) \eta + (\nabla_{yx} g(x,y_g^*(x)) d)^\top \eta. \label{eq: minimization directional derivative}
\end{align}
This gives a unique solution as the objective is strongly convex in $\eta$ and the constraint $\mathcal{C}_{\tilde{\mathcal{Y}}}(y_g^*(x))$ is closed and convex:
\begin{align*}
    D_d(y^*_g(x)) = \Proj_{\mathcal{C}_{\tilde{\mathcal{Y}}}(y^*_g(x))}^{\nabla_{yy} g(x,y^*_g(x))} \Big( - (\nabla_{yy} g(x,y^*_g(x)))^{-1} \nabla_{yx}g(x,y^*_g(x)) d \Big).
\end{align*}
In other words, $D_d(y^*_g(x)) $ is the projection of $-(\nabla_{yy} g(x,y_g^*(x)))^{-1} \nabla_{yx} g(x,y_g^*(x)) d$ onto the critical cone $\mathcal{C}_{\tilde{\mathcal{Y}}}(y_g^*(x))$ in the $\nabla_{yy} g(x,y_g^*(x))$-norm, which is given by $\|v\|_{\nabla_{yy} g(x,y_g^*(x))}^2 := v^\top \nabla_{yy} g(x,y_g^*(x)) v$. Here, $\nabla_{yy} g(x,y_g^*(x))$-norm is well defined as $\nabla_{yy} g(x,y_g^*(x))$ is positive definite, as guaranteed by strong convexity.

\end{proof}
Lemma~\ref{lemma: implicit gradient} prepares to bridge the connection of $v_\gamma(x)$ and $v(x)$ for the constrained LL setting. Building on this, we seek to provide a tighter estimate for $l_{F_\gamma,1}$ with the following conventional assumption \citep{chen2024finding,kwon2023fully}.
\begin{assumption}
\label{assumption: Hessian Lipschitz} 
    Assume $f$, $g$ are twice differentiable on $\mathcal{X}\times \mathcal{Y}$, and $\nabla^2 f$, $\nabla^2 g$ are respectively $l_{f,2}$, $l_{g,2}$-Lipschitz in $y\in \mathcal{Y}$.
\end{assumption}

\begin{theorem}[$\Oc(1)$-smoothness of $F_\gamma(x)$]
\label{theorem: smoothness}
    Suppose $\mathcal{Y}(x)=\tilde{\mathcal{Y}}$ is closed, convex and non-empty uncoupled LL constraint, and Assumption~\ref{assumption: UL}.\ref{ass: UL y Lipschitz}--\ref{ass: UL smooth}, \ref{assumption: LL}.\ref{ass: LL SC}--\ref{ass: LL smooth}, \ref{assumption: Hessian Lipschitz} hold. For all unit direction $d\in \mathbb{R}^{d_x}$, we have $ \| D_{dd}^2( F_\gamma(x))\| \leq l_{F_\gamma,1}= \Oc(1)$ non-scalable with $\gamma \geq \frac{l_{f,1}}{\mu_g}$; i.e., $F_\gamma(x)$ is $\Oc(1)$-smooth.
\end{theorem}

The key idea of the proof of Theorem~\ref{theorem: smoothness} is to observe that 
\begin{align}
    D_{dd}^2(F_\gamma(x)) = \gamma \left(D_{dd}^2(v_\gamma(x)) - D_{dd}^2(v(x))\right), \label{eq: 2nd order directional derivative of F}
\end{align}
following \eqref{eq: F gam function}, and that $\|y_g^*(x) - y_\gamma^*(x)\| = \Oc(\gamma^{-1})$ in Lemma~\ref{lemma: distance of yg ygam}. Therefore, using the Lipschitz continuity of the gradient and Hessian of both $f$ and $g$, we can bound the difference in norm of directional derivatives by $\Oc(\gamma^{-1})$ and achieve that the smoothness parameter $l_{F_\gamma,1} = \mathcal{O}(1)$ non-scalable with $\gamma$. This aligns with the empirical observation in Fig.~\ref{fig:gradient-plot} and also showcases the benefits of using alternating updates. The full proof of Theorem~\ref{theorem: smoothness} is available at Appendix~\ref{appendix: Theorem smoothness}.

\textbf{Technical insight.} 
Theorem~\ref{theorem: smoothness} reveals a crucial distinction: \textsf{while the joint update objective $\tilde{F}_\gamma(x,y)$ becomes \textsl{less smooth} as $\gamma$ grows ($l_{\tilde{F}_\gamma,1} = \Oc(\gamma)$), the alternating formulation $F_\gamma(x)$ remains uniformly smooth with $l_{F_\gamma,1} = \Oc(1)$, independent of $\gamma$}. This independence allows us to use a constant step size $\eta =\Oc(1)$ to satisfy $\eta \leq l_{F_\gamma,1}^{-1}$, the gradient descent condition.
Additionally, \textsf{directional derivatives are only needed for the smoothness analysis and not for algorithm design}. Although $y_g^*(x)$ and the second-order derivatives of $F_\gamma(x)$ are in general only directionally differentiable and may be nonsmooth, with directional derivatives difficult to compute, $F_\gamma(x)$ is differentiable with a closed-form gradient according to Lemma~\ref{eq: value function gradient generalized form}. 
We provide the summary of differentiability $y_g^*(x)$, $F_\gamma(x)$, and $\nabla F_\gamma(x)$ in Table~\ref{tab:differentiability-summary}.

\begin{table}[t]
    \centering
    \begin{adjustbox}{width=0.99\textwidth}
    \begin{tabular}{l | l}
        \hline \hline
        \textbf{Properties} & \textbf{Assumptions} \\
        \hline \hline
        \multicolumn{2}{c}{\textbf{Uncoupled constrained case} $\tilde{\mathcal{Y}}=\{y\in \mathcal{Y}: c(y)\leq 0\}$} \\
        \hline
        $F_\gamma(x)$ diff.; $\nabla F_\gamma(x)$ depends on $y_g^*(x)$;
        & $f,g$ smooth; $\frac{1}{\gamma}(f+g)$ proximal-PL; $\tilde{\mathcal{Y}}$ closed and \\
        $y_g^*(x)$ direct.-diff.; $F_\gamma(x)$ 2nd-order direct.-diff.
        & convex. (weaker than Assumption~\ref{assumption: UL}--\ref{assumption: constraint}) \\
        \hline
        $F_\gamma(x)$ is $\Oc(1)$-smooth.
        & above and $f,g$ Hessian Lipschitz; $f(x,\cdot)$ Lipschitz. \\
        \hline \hline
        \multicolumn{2}{c}{\textbf{Coupled constrained case} $\mathcal{Y}(x)=\{y\in \mathcal{Y}: c(x,y)\leq 0\}$} \\
        \hline
        $F_\gamma(x)$ diff; $\nabla F_\gamma(x)$ depends on $y_g^*(x)$ and $\lambda_g^*(x)$.
        & Assumption~\ref{assumption: UL}--\ref{assumption: constraint} \\
        \hline
        $y_g^*(x), \lambda_g^*(x)$ diff.; gradient unavail.; $y_g^*(x)$ direct.-diff. avail.;
        & above, and $\mathcal{Y}$ smooth on boundary;  \\
        $F_\gamma^\lambda(x)$ diff.; $\nabla F_\gamma^\lambda(x)$ depends on $\nabla \lambda_g^*(x)$ (unavail.).
        & strict complementarity holds \\
        \hline
        $F_\gamma(x)$ is $\Oc(1)$-smooth.
        & above and $f,g$ Hessian Lipschitz; $f(x,\cdot)$ Lipschitz.  \\
        \hline \hline
    \end{tabular}
    \end{adjustbox}
    \vspace{0.2cm}
    \caption{Summary of differentiability of $F_\gamma(x)$, $F_\gamma^\lambda(x)$, and the solution mappings $y_g^*(x)$, $\lambda_g^*(x)$ under uncoupled and coupled constraints. Here, ``diff.", ``direct.-diff.", ``avail.", and ``unavail." respectively stand for ``differentiable", ``directional differentiable", ``available", ``unavailable". The unavailability comes from the unknown local characterization of $\mathcal{Y}$ by some $\varphi(y)\leq 0$.} 
    \label{tab:differentiability-summary}
\end{table}

\subsection{ALT-PBGD: an improved PBGD method}
\label{sec: Improved PBGD method}

\begin{algorithm}[t]
\caption{JNT-PBGD}
\label{alg: JNT-PBGD}
\begin{algorithmic}[1]
\State\textbf{Inputs:} initial point $x_0$; stepsize $\eta^{\text{JNT}}$; number of iterations $T$; inner \textsf{Min Solver}
\For{$t = 0$ \textbf{to} $T-1$}
    \State Update $y_{t}^g$ as \eqref{eq: yg update} using \textsf{Min Solver}
    \State Update $[x_{t+1};y_{t+1}]$ via \eqref{eq: joint update}
\EndFor
\State\textbf{Outputs:} $(x_T, y_{T})$
\end{algorithmic}
\end{algorithm}

In this section, we first revisit two variants of the PBGD method: the Joint update version (JNT-PBGD) \citep{shen2023penalty} and the Alternating update version (ALT-PBGD) \citep{kwon2023penalty,chen2024finding,xiao2024unlocking}. We then show that ALT-PBGD has an advantage over JNT-PBGD in terms of iterational convexity. At each iteration $t$, JNT-PBGD solves
\begin{align}
y_{t}^g \approx & \arg\min_{y_g \in \mathcal{Y}(x)} g(x_t, y_g) \label{eq: yg update}
\end{align}
up to an $\epsilon$-suboptimal point with respect to a distance metric. Based on Lemma~\ref{lemma: gradient of v}, where the Lagrangian term is not involved in this setting, we estimate $\nabla v(x)$ and jointly update
\begin{align}
\begin{bmatrix} x_{t+1}\\ y_{t+1} \end{bmatrix} = \Proj_{\mathcal{X} \times \mathcal{Y}} \left(\begin{bmatrix} x_{t}\\ y_{t} \end{bmatrix}  - \eta^{\text{JNT}}
\begin{bmatrix}
\nabla_x f(x_t, y_t) + \gamma(\nabla_x g(x_t, y_t) - \nabla_x g(x, y_{t}^g))  \\
\nabla_y f(x_t, y_t) + \gamma \nabla_y g(x_t, y_t)
\end{bmatrix}
\right), \label{eq: joint update}
\end{align}
where the step size $\eta^{\text{JNT}} \leq l_{\tilde{F}_\gamma,1}^{-1}$, and $l_{\tilde{F}_\gamma,1}=\Theta(\gamma)$ denotes the smoothness constant of $\tilde{F}_\gamma(x,y)$ in \eqref{eq: joint penalty problem}.

\begin{algorithm}[t]
\caption{ALT-PBGD}
\label{alg: ALT-PBGD}
\begin{algorithmic}[1]  
\State\textbf{Inputs:} initial point $x_0$; stepsize $\eta^{\text{ALT}}$; number of iterations $T$; inner \textsf{Min Solver}
\For{$t = 0$ \textbf{to} $T-1$}
    \State Update $y_{t}^g$ as \eqref{eq: yg update} and $y_{t}^\gamma$ as \eqref{eq: y gam update} using \textsf{Min Solver}
    \State Update $x_{t+1} = \Proj_{\mathcal{X}}\big(x_t -\eta g_t\big)$, where $g_t$ is in \eqref{eq: x update gt}
\EndFor
\State\textbf{Outputs:} $(x_T, y_{T}^\gamma)$
\end{algorithmic}
\end{algorithm}

For ALT-PBGD, at each iteration $t$, in addition to computing $y_t^g$ from \eqref{eq: yg update}, it updates $y_{t}^\gamma$ via
\begin{align}
y_{t}^\gamma \approx & \arg\min_{y_\gamma \in \mathcal{Y}(x_t)} \gamma^{-1} f(x_t, y_\gamma) + g(x_t, y_\gamma)
\label{eq: y gam update}
\end{align}
to an $\epsilon$-suboptimal point to find $\nabla v_\gamma(x)$ similarly according to Lemma~\ref{lemma: gradient of v}. In this way, following Lemma~\ref{lemma: gradient of v}, we estimate $\nabla F_\gamma(x_t) = \gamma(\nabla v_\gamma(x) - \nabla v(x))$ via
\begin{align}
g_t = \nabla_x f(x, y_{t}^\gamma) + \gamma \left(\nabla_x g(x, y_{t}^\gamma) - \nabla_x g(x, y_{t}^g)\right),
\label{eq: x update gt}
\end{align}
and update $x_{t+1} = \Proj_{\mathcal{X}} \left(x_t - \eta^{\text{ALT}} g_t \right)$ with $\eta^{\text{ALT}} \leq l_{F_\gamma,1}^{-1}$, where $l_{F_\gamma,1}=\Oc(1)$ according to Theorem~\ref{theorem: smoothness}.
We outline the oracles for ALT-PBGD and JNT-PBGD in Algorithm \ref{alg: ALT-PBGD} and \ref{alg: JNT-PBGD} and present their complexity analysis in Proposition \ref{prop: ALT-PBGD}.

\begin{proposition}
\label{prop: ALT-PBGD}
    Suppose $\mathcal{Y}(x)=\tilde{\mathcal{Y}}$ is closed, convex and non-empty uncoupled LL constraint, and Assumption~\ref{assumption: UL}.\ref{ass: UL y Lipschitz}--\ref{ass: UL smooth}, \ref{assumption: LL}.\ref{ass: LL SC}--\ref{ass: LL smooth} hold.
    For the choice of $\gamma =\Oc(\epsilon^{-0.5}) \geq \frac{l_{f,1}}{\mu_g}$,
    ALT-PBGD in Algorithm \ref{alg: ALT-PBGD} with $\eta^{\text{ALT}} =\Oc(1) \leq l_{F_\gamma,1}^{-1}$ is achieved for $T={\Oc}(\epsilon^{-1})$  complexity, and JNT-PBGD in Algorithm \ref{alg: JNT-PBGD} with $\eta^{\text{JNT}} =\Oc(\gamma^{-1}) \leq l_{\tilde{g},1}^{-1}$ is achieved for $T=\Oc(\epsilon^{-1.5})$ outer-loop complexity for $\|G_{F_\gamma, \mathcal{X}}(x)\|^2<\epsilon$. 
\end{proposition} 
\begin{remark}
\label{remark: PGD complexity}
    When naive projected gradient descent is used as the \textsf{Min Solver}, the inner problem converges linearly, resulting in an overall complexity of $\tilde{\mathcal{O}}(\epsilon^{-1})$ for ALT-PBGD, and $\tilde{\mathcal{O}}(\epsilon^{-1.5})$ for JNT-PBGD. 
\end{remark}
\begin{proof}[Proof of Proposition~\ref{prop: ALT-PBGD} and Remark~\ref{remark: PGD complexity}]
\textbf{ALT-PBGD.}
According to \citep{karimi2016linear}, to achieve 
\begin{align}
    \|y_{t}-y_g^*(x_t) \|^2,\|y_{t}-y_g^*(x_t) \|^2<\epsilon^{2}, \label{eq: LL optimality}
\end{align}
we can apply PGD on 
\begin{align*}
    & \min_{y\in \tilde{\mathcal{Y}}}~g(x,y),\quad \text{and}\quad \min_{y\in \tilde{\mathcal{Y}}}~\gamma^{-1}f(x,y)+g(x,y)
\end{align*}
with algorithm complexity $\Oc(2\ln({\epsilon}^{-1}))$. In this way, the update bias $b(x_t)=\nabla F_\gamma(x_t)-g_t$ has its norm square bounded by
\begin{align*}
    \|b(x_t) \|^2= &\|\nabla F_\gamma(x_t) -g_t\|^2 \nonumber \nonumber \\
    \leq &\gamma \| \gamma^{-1}\nabla_x f(x_t,y_\gamma^*(x_t)) + \nabla_x g(x_t,y_\gamma^*(x_t))-\nabla_x g(x_t,y_g^*(x_t))  \nonumber  \\
    & ~~~ -\big(\gamma^{-1}\nabla_x f(x_t,y_t^\gamma) + \nabla_x g(x_t,y_t^\gamma)-\nabla_x g(x_t,y_t^g) \big)\|^2 \nonumber \\
    \leq & 2 \| \nabla_x f(x_t,y_\gamma^*(x_t)) + \gamma \nabla_x g(x_t,y_\gamma^*(x_t)) - \big( \nabla_x f(x_t,y_t^\gamma) + \gamma \nabla_x g(x_t,y_t^\gamma)
    \big) \|^2 \nonumber \\
    & + 2 \|\gamma\nabla_x g(x_t,y_\gamma^*(x_t))- \gamma \nabla_x g(x_t,y_t^g)\|^2 \nonumber \\
    \leq & 2(l_{F_\gamma,1}+\gamma l_{g,1} )^2 \| y_\gamma^*(x_t)-y_t^\gamma \|^2 + 2 \gamma^2 l_{g,1}^2\| y_g^*(x)-y_t^g\|^2\nonumber\\
    \leq & \Oc(\gamma^2 \epsilon^2) = \Oc(\epsilon).
\end{align*}
where the second inequality is by Young's inequality; the third is from the smoothness of $f$ and $g$; and the last is from \eqref{eq: LL optimality} and the choice of $\gamma = \Oc(\epsilon^{-0.5})$.

According to smoothness of $F_\gamma(x)$, there is
    \begin{align}
        F_\gamma(x_{t+1}) \leq & F_\gamma(x_t)+\langle \nabla F_\gamma(x_t), x_{t+1}-x_t \rangle + \frac{l_{F_\gamma,1}}{2}\| x_{t+1}-x_t  \|^2 \nonumber \\
        \leq & F_\gamma(x_t) + \langle  g_t, x_{t+1}-x_t \rangle + \frac{1}{2 \eta^{\text{ALT}} }\| x_{t+1}-x_t  \|^2 + \langle  b(x_t), x_{t+1}-x_t \rangle,
        \label{eq: smooth step 1}
    \end{align}
    where the second inequality is by $\eta^{\text{ALT}} \leq \frac{1}{l_{F_\gamma,1}}$. The projection guarantees that $x_{t+1}$ and $x_t$ are in $\mathcal{X}$. In this way, 
    Following lemma \ref{lemma: projection gradient update inquality}, we know that 
    \begin{align*}
        \langle g_t , x_{t+1}-x_t \rangle \leq -\frac{1}{\eta^{\text{ALT}}} \| x_{t+1}-x_t\|^2.
    \end{align*}
    Plugging this back to \eqref{eq: smooth step 1},
    \begin{align}
        F_\gamma(x_{t+1}) \leq & F_\gamma(x_t) - \frac{1}{2 \eta^{\text{ALT}} }\| x_{t+1}-x_t  \|^2  + \langle  b(x_t), x_{t+1}-x_t \rangle  \nonumber \\
         \leq & F_\gamma(x_t) - \frac{1}{2 \eta^{\text{ALT}} }\| x_{t+1}-x_t  \|^2  + \eta^{\text{ALT}} \|  b(x_t)\|^2 +\frac{1}{4\eta^{\text{ALT}}}\|x_{t+1}-x_t \|^2 \nonumber \\
         = & F_\gamma(x_t) - \frac{1}{4 \eta^{\text{ALT}} }\| x_{t+1}-x_t  \|^2  + \eta^{\text{ALT}} \|  b(x_t)\|^2, \label{eq: descent intermediate step 1}
    \end{align}
    where the second inequality is from Young's inequality. In this way, 
    \begin{align}
        \| G_{F_\gamma,\mathcal{X}}(x_t)\|^2 = &\Big\| \frac{x_t-\Proj_{\mathcal{X}}(x_t-\eta^{\text{ALT}} \nabla F_\gamma(x_t))}{\eta^{\text{ALT}}}\Big\|^2  \nonumber \\
        = & \Big\| \frac{x_t-\Proj_{\mathcal{X}}(x_t-\eta^{\text{ALT}} g_t -\eta^{\text{ALT}} b_t)}{\eta^{\text{ALT}}}\Big\|^2  \nonumber \\
        \leq & \Big\| \frac{x_t-\Proj_{\mathcal{X}}(x_t-\eta^{\text{ALT}} g_t ) }{\eta^{\text{ALT}}}\Big\|^2 + \| b_t\|^2 \nonumber \\
        =& \Big\|\frac{x_t-x_{t+1}}{\eta^{\text{ALT}}} \Big\|^2 + \| b_t\|^2 \nonumber \\
        \leq & \frac{4}{\eta^{\text{ALT}}}(F_\gamma(x_t)-F_\gamma(x_{t+1})) + 5\| b_t\|^2. \label{eq: G eta and x update}
    \end{align}
    Here, the first inequality follows the non-expansiveness of projection, and the second inequality is by plugging in \eqref{eq: descent intermediate step 1}.
    Telescoping therefore gives
    \begin{align}
        \frac{1}{T} \sum_{t=0}^{T-1} \| G_{F_\gamma,\mathcal{X}}(x_t)\|^2  \leq & \frac{4}{\eta^{\text{ALT}} T} (F_\gamma(x_0)-F_\gamma(x_t) ) + \frac{5}{T} \sum_{t=0}^{T-1}\|  b(x_t)\|^2 \nonumber \\
        =& \Oc({\eta^{\text{ALT}}}^{-1} T^{-1} + \epsilon) \label{eq: G eta upper bound}
    \end{align}
    As $l_{F_\gamma,1}=
    \Oc(1)$ (cf. Theorem~\ref{theorem: smoothness}), we can choose $\eta^{\text{ALT}} = O (1) \leq \Oc(l_{F_\gamma,1}^{-1})$, to achieve $\frac{1}{T} \sum_{t=0}^{T-1} \|G_{F_\gamma,\mathcal{X}}(x_t)\|^2\leq \epsilon$, we require complexity
    \begin{align*}
        T = \Oc(\epsilon^{-1}).
    \end{align*}
    Additionally, the inner complexity is of $\Oc(2\ln(\epsilon^{-1}))$ from implementing PBG ~\citep{karimi2016linear}. In this way, the overall complexity is $\tilde{O}(\epsilon^{-1})$.

\textbf{JNT-PBGD~\citep{shen2023penalty}.} Similarly, applying PGD on $\min_{y\in \{y\in \mathcal{Y}:c(y)\leq 0 \}}g(x,y)$ can achieve $\|y_{t}-y_g^*(x_t) \|^2<\epsilon^{2}$ with algorithm complexity $\Oc(\ln({\epsilon}^{-1}))$. 
Following the similar analysis as in the one for ALT-PBGD, the update bias has its norm square bounded by
\begin{align*}
    \|b(x_t) \|^2 
    = &\bigg\|\begin{bmatrix}
    \nabla_x f(x_t, y_t) + \gamma(\nabla_x g(x_t, y_t) - \nabla_x g(x, y_g^*(x_t)))  \\
    \nabla_y f(x_t, y_t) + \gamma \nabla_y g(x_t, y_t)
    \end{bmatrix} \nonumber\\
    & ~~~~ ~~~~ -\begin{bmatrix}
    \nabla_x f(x_t, y_t) + \gamma(\nabla_x g(x_t, y_t) - \nabla_x g(x, y_{t}^g))  \\
    \nabla_y f(x_t, y_t) + \gamma \nabla_y g(x_t, y_t)
    \end{bmatrix} \bigg\|^2 \nonumber \\
    =& \gamma^2 \| \nabla_x g(x_t,y_g^*(x_t))-\nabla_x g(x_t,y_t^g)\|^2 \nonumber\\
    \leq & \gamma^2 l_{g,1}^2\| y_g^*(x)-y_t^g\|^2\nonumber\\
    \leq & \Oc(\gamma^2 \epsilon^2) = \Oc(\epsilon),
\end{align*}
and choosing $\eta^{\text{JNT}}\leq l_{\tilde{F}_\gamma,1}^{-1} = \Oc(\gamma^{-1})$ gives
\begin{align*}
    \frac{1}{T} \sum_{t=0}^{T-1} \| G_{F_\gamma,\mathcal{X}}(x_t)\|^2  \leq & \Oc({\eta^{\text{JNT}}}^{-1} T^{-1} + \epsilon) = \Oc(\epsilon^{0.5} T^{-1} + \epsilon).
\end{align*}

In this way, to achieve $\frac{1}{T} \sum_{t=0}^{T-1} \|G_{F_\gamma,\mathcal{X}}(x_t)\|^2\leq \epsilon$, we require complexity $T = \Oc(\epsilon^{-1.5})$.
Taking the inner $\Oc(\ln(\epsilon^{-1}))$ complexity into account, the total iteration complexity for JNT-PBGD is $\Oc(\epsilon^{-1.5}\ln(\epsilon^{-1}))=\tilde{\Oc}(\epsilon^{-1.5})$.

\end{proof}
The definition of the convergence metric $G_{F_\gamma, \mathcal{X}}(x)$ is provided in Definition~\ref{def: generalized gradient}. It is commonly used as a convergence metric in constrained optimization \citep{ghadimi2016mini, chen2021closing, kwon2023penalty}. Notably, ALT-PBGD improves upon the $\mathcal{O}(\epsilon^{-1.5})$ complexity of JNT-PBGD, achieving $\mathcal{O}(\epsilon^{-1})$, which matches the rate of Projected Gradient Descent for non-convex smooth objectives. This gain arises from minimizing $F_\gamma(x)$ directly, rather than the joint penalty $\tilde{F}_\gamma(x, y)$ in JNT-PBGD \citep{shen2023penalty}. In the latter, the penalty objective has smoothness $l_{\tilde{g},1} = \mathcal{O}(\gamma)$, requiring a stepsize $\eta = \mathcal{O}(\epsilon^{0.5})$ and resulting in the slower $\tilde{\mathcal{O}}(\epsilon^{-1.5})$ complexity.

\section{Developing Value-Function-Free Algorithms: PBGD-Free}
\label{sec: Developing Value-Function-Free Algorithms: PBGD-Free}

In this section, we continue to consider the non-coupled constraint $\mathcal{Y}(x)=\tilde{\mathcal{Y}}$ independent from $x$. 
From the earlier section, we observe that one limitation of ALT-PBGD compared to JNT-PBGD is that it requires two inner loops to compute both $y_t^g$ in \eqref{eq: yg update} and $y_t^\gamma$ in \eqref{eq: y gam update}, whereas JNT-PBGD only requires $y_t^g$. Nevertheless, empirical evidence, for example, the instance in \eqref{eq: bilevel repre LLM}, shows that directly updating
\begin{align}
x_{t+1} = \Proj_\mathcal{X}\big(x_t - \eta \nabla_x f(x_t, y_t^\gamma)\big) \label{eq: x update VaFF}
\end{align}
can sometimes yield satisfactory results. This observation motivates the development of PBGD-Free, presented in Algorithm~\ref{alg: PBGD-Free}.

\begin{figure}

\begin{minipage}[t]{0.3\linewidth}
    \centering
\includegraphics[width=\textwidth]{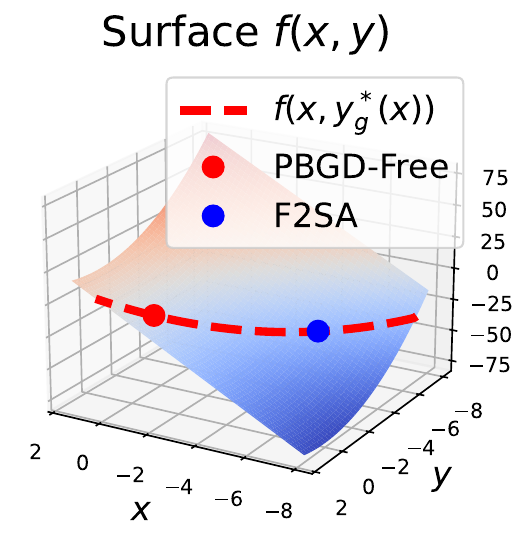}
\end{minipage}
\hfill
\begin{minipage}[t]{0.39\linewidth}
\centering
\includegraphics[width=0.95\linewidth]{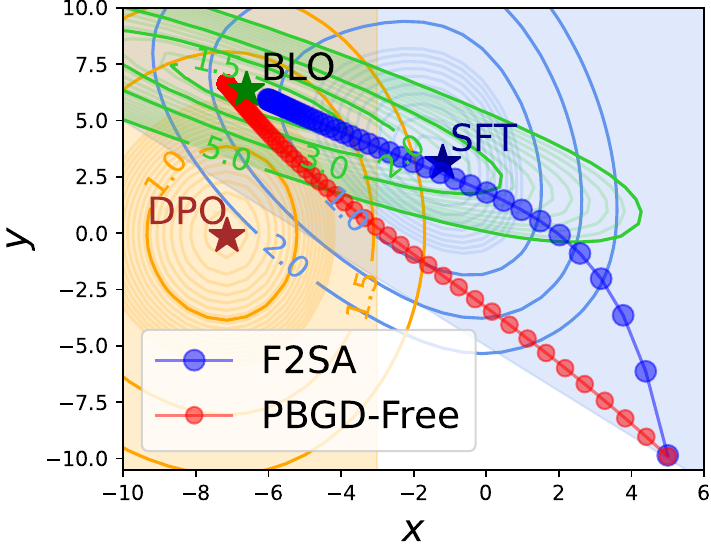}
\end{minipage}
\begin{minipage}[t]{0.29\linewidth}
    \centering
    \includegraphics[width=\linewidth]{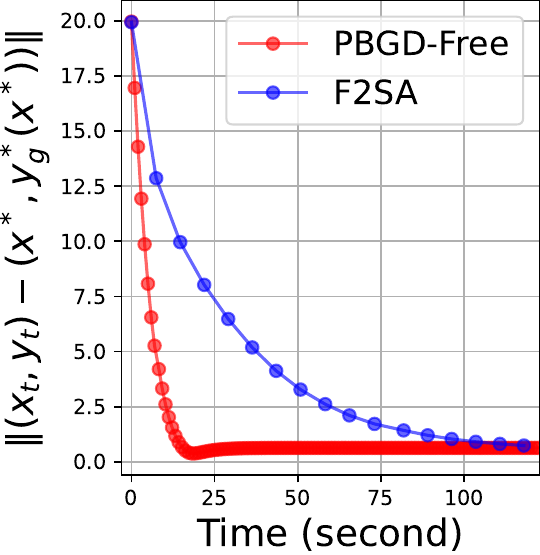}
\end{minipage}
\caption{\textbf{An Illustration to show PBGD-Free does not work in Example~\ref{example:bias}, but works well in PEFT. 
}  
The \textbf{left} plot shows the $f(x,y)$ and $f(x, y^*_g(x))$ in Example~\ref{example:bias}, with \textcolor{red}{red} and \textcolor{blue}{blue} dots as the converged points using PBGD-Free and F$^2$SA method. 
The \textbf{middle} plot shows the trajectory of updates in PEFT. The \textcolor{pythonorange}{orange}, \textcolor{pythonblue}{blue}, and \textcolor{pythongreen}{green} contours are the landscapes of $f_{\text{DPO}}(x,y)$, $g_{\text{SFT}}(x,y)$, and $\tilde{F}_\gamma (x,y)$, respectively. 
The \textbf{right} plot presents the convergence vs. time in PEFT, showing faster convergence of PBGD-Free.
(See Sec.~\ref{app:toy_example} for details.)
}
\label{fig:toy_example_convergence_results}
\end{figure}
\textbf{Positive empirical observations.} \emph{PBGD-Free largely reduces computation and memory cost while preserving the accuracy in LLM parameter efficient fine-tuning (PEFT) \citep{ren2023prepare,aghajanyan2021muppet}.} See Fig.~\ref{fig:toy_example_convergence_results} and experiments in Sec.~\ref{sec: exp}.
We consider a model parameterized by $(x,y)$ and apply a bilevel formulation for PEFT:
\begin{align} 
\min_{x,y}~ & f_{\text{DPO}}(x,y;\mathcal{D}_{\text{DPO}}) \quad 
\text{s.t.} \quad y \in \arg\min_{y'} g_{\text{SFT}}(x,y';\mathcal{D}_{\text{SFT}}), \label{eq: bilevel repre LLM}
\end{align} 
where $x$ denotes the backbone parameters (typically a pretrained LLM) and $y$ is a lightweight, easily fine-tuned head. We prioritize supervised fine-tuning (SFT) loss at the LL to ensure a capable LLM model, while we keep direct preference optimization (DPO) loss~\citep{rafailov2023direct} in the UL to keep alignment with human preferences. Details of this setting are provided in Sec.~\ref{sec: problem setting and its flatness}.

\textbf{Negative theoretical observations.}
\emph{There are some counterexamples where PBGD-Free does not converge.}
{\bf c1)} When the UL objective solely depends on the LL variable $f(x,y)=f(y)$, as in data hyper-cleaning \citep{shaban2019truncated,hong2020two,shen2024seal} and meta-learning \citep{franceschi2018bilevel,chen2023_fnt_meta,franceschi2018bilevel}, the LL penalty gradient term contains all the gradient information about the UL variable $x$, so it cannot be omitted; and, {\bf c2)} When $f(x,y)$  jointly depends on both variables, omission of the penalty gradient term can lead to a different update direction;  
see more details in Example~\ref{example:bias}
and Proposition~\ref{prop: PBGD}.

\begin{algorithm}[t]
\caption{PBGD-Free}
\label{alg: PBGD-Free}
\begin{algorithmic}[1]
\State\textbf{Inputs:} initial points $x_0$, $y_{\gamma,0}$; stepsize $\eta$; number of iterations $T$
\Comment{\textcolor{blue}{Blue} part is the naive version; \textcolor{red}{red} part is the fully single-loop version}
\For{$t = 0$ \textbf{to} $T-1$}
    \State Update $y_{t+1}^\gamma$ \textcolor{blue}{by \eqref{eq: y gam update}} \textbf{or} by \textcolor{red}{$y_{t+1}^\gamma = \Proj_{\mathcal{Y}(x)}\big(y_t^\gamma-\eta^\gamma (\nabla_y \gamma^{-1}f(x_t,y_t^\gamma) + \nabla_y g\big(x_t,y_t^\gamma)\big)\big)$}
    \State Update $x_{t+1} = \Proj_{\mathcal{X}}\Big(x_t - \eta \nabla_x f(x_t,y_t^\gamma)\Big)$
\EndFor
\State\textbf{Outputs:} $(x_T, y_T^\gamma)$
\end{algorithmic}
\end{algorithm}

\begin{example}\label{example:bias}
For the BLO problem in \eqref{eq: original problem 1} with $f(x,y) = x^2 + 10y$ and $g(x,y) = (y - x + 1)^2$, the gradients $\langle\nabla F_\gamma(x),\nabla_x f(x, y_\gamma^*(x))\rangle <0$ exhibit \textbf{opposite directions} for $x \in (-5, 0)$. As a result, $\nabla F_\gamma(x) = 2x + 10$ converges to $x = -5$ while $\nabla_x f(x, y_\gamma^*(x)) = 2x$ converges to $x=0$. 
\end{example}

In the following subsections, we first explore the negative theoretical observations and show that they do not always converge under the traditional Lipschitz condition. Then, we will introduce a new mild condition as a surrogate to the widely used Lipschitz condition of the UL objective, discuss its use cases, and establish the convergence rate of the PBGD-Free algorithm under this condition.

\subsection{Negative theoretical observations of PBGD-Free}
\label{sec: Negative theoretical observations of PBGD-Free}
Recall that Assumption \ref{assumption: UL}.\ref{ass: UL y Lipschitz} assumes $f$ being $l_{f,0}$-Lipschitz in $y \in \mathcal{Y}$ for all $x\in \mathcal{X}$. In other words, it assumes that there exists constant $l_{f,0}\geq 0$ such that for all fixed $x\in \mathcal{X}$,
\begin{align}
    |f(x,y_1)-f(x,y_2)|\leq l_{f,0}\|y_1-y_2\|,\quad \forall y_1,y_2\in \mathcal{Y}. \label{eq: lipschitzness of f in y}
\end{align}
For differentiable $f$, Assumption \ref{assumption: UL}.\ref{ass: UL y Lipschitz} implies $\|\nabla_y f(x,y_g^*(x))\|\leq l_{f,0}$. This assumption is crucial for the key results in BLO literature \citep{shen2023penalty,kwon2023penalty,chen2024finding,chen2021closing,ji2020provably,hong2020two,ji2021bilevel,ye2022bome,huang2021enhanced} and it is critical for establishing the result $\|y_g^*(x)-y_\gamma^*(x) \| = \Theta(l_{f,0}\mu_g^{-1} \gamma^{-1})$ in Lemma~\ref{lemma: distance of yg ygam}.

This insight has practical consequences for the design and analysis of PBGD algorithms. In particular, although PBGD-Free is computationally efficient by eliminating the inner loop estimates of $y_g^*(x)$, the removal of the value function part $b(x_t):=\gamma (\nabla_x g(x,y_\gamma^*(x))-\nabla_x g(x,y_g^*(x)))$ in PBGD-Free introduces a non-negligible bias shown in Example \ref{example:bias}. To see this, by Taylor's expansion, the omitted value function part $b(x_t)$ is in the order of 
\begin{align}
    & \|b(x_t)\| = \gamma \|\nabla_{xy}g(x,y_g^*(x))\big(y_\gamma^*(x) -y_g^*(x) \big) \| + \Oc(\gamma\|y_g^*(x) -y_\gamma^*(x) \|^2). \label{bias_value}
\end{align}
Here, the second term $\Oc(\gamma\|y_g^*(x) -y_\gamma^*(x) \|^2=\Oc(l_{f,0}^2 \gamma^{-1})$ can be small enough with enlarging $\gamma$ following Proposition \ref{lemma: distance of yg ygam}. For general settings where $\nabla_{xy}g(x,y_g^*(x))\neq 0$, due to the first term and according to Proposition \ref{lemma: distance of yg ygam}, the bias in \eqref{bias_value} is tight as  $\Theta(1)$. 
Therefore, in the general case where $\nabla_{xy}g(x,y_g^*(x))\neq 0$, the PBGD-Free algorithm only drives the iterates to a neighborhood of the stationary point, which we will formally quantify as follows.

\begin{proposition}[Lower bound on asymptotic error]
\label{prop: PBGD}
Suppose $\mathcal{Y}(x)=\tilde{\mathcal{Y}}$ is closed, convex and non-empty uncoupled LL constraint, and Assumption~\ref{assumption: UL}.\ref{ass: UL y Lipschitz}--\ref{ass: UL smooth}, \ref{assumption: LL}.\ref{ass: LL SC}--\ref{ass: LL smooth}, \ref{assumption: Hessian Lipschitz} hold.
There exists a BLO problem where the iterates generated by the \textcolor{blue}{naive version} of PBGD-Free (Algorithm~\ref{alg: PBGD-Free}) converge to a neighborhood of a stationary point with a non-vanishing residual even when choosing step sizes appropriately, i.e., 
\begin{equation}
\lim_{T\rightarrow \infty}\frac{1}{T}\sum_{t=0}^{T-1} \|G_{F_{\gamma, \mathcal{X}}}(x_t)\|^2 = \lim_{T\rightarrow \infty}\Theta\Bigg(\frac{1}{T}\sum_{t=0}^{T-1} \|  \nabla_y f(x_t,y_g^*(x_t))\|^2\Bigg)=\Theta\left(l_{f,0}^2\right).
\label{eq: negative result of PBGD-Free}
\end{equation}
\end{proposition}
\begin{proof}
Algorithm \ref{alg: PBGD-Free} can be viewed as a biased ALT-PBGD (Algorithm \ref{alg: ALT-PBGD}) oracle with bias:
\begin{align*}
    \| b_t\|  = & \| \nabla F_\gamma(x_t)-\nabla_x f(x_t,y_t^\gamma) \| \\
    \stackrel{(a)}{=} & \| \nabla_x f(x_t,y_\gamma^*(x)) + \gamma (\nabla_x g(x_t,y_\gamma^*(x_t))-\nabla_x g(x_t,y_g^*(x_t)))-\nabla_x f(x_t,y_t^\gamma)\| \\
    \stackrel{(b)}{\leq} & \| \nabla_x f(x_t,y_\gamma^*(x)) - \nabla_x f(x_t,y_t^\gamma)\| + \gamma\|\nabla_x g(x_t,y_\gamma^*(x_t))-\nabla_x g(x_t,y_g^*(x_t))\| \\
    \stackrel{(c)}{\leq} & l_{F_\gamma,1} \|y_\gamma^*(x_t)- y_t^\gamma \| + \gamma l_{g,1}\| y_\gamma^*(x_t)-y_g^*(x_t)\| \\
    \stackrel{(d)}{\leq} & l_{F_\gamma,1} \sqrt{\gamma^{-1}f(x_t,y_t^\gamma)-v_\gamma(x_t)}+ \Oc(\|  \nabla_y f(x,y_g^*(x_t))\|)\\
    \stackrel{(e)}{\leq} & l_{F_\gamma,1} \sqrt{(1-\eta^\gamma \mu)^K (\gamma^{-1}f(x_t,y_{t-1}^\gamma)-v_\gamma(x_t)) }+ \Oc(\|  \nabla_y f(x,y_g^*(x_t))\|)
\end{align*}
where (a) is by plugging in $\nabla F_\gamma(x_t)$ in \eqref{eq: nabla F gam}, this holds for arbitrary $y_g^*(x)$, $y_\gamma(x)$ as solutions to problems in \eqref{eq: nabla F gam}; (b) follows triangle-inequality; (c) uses the smoothness of $f$ and $g$; the first term in (d) is obtained by the QG property as ensured by proximal-PL condition and smoothness as per Lemma~\ref{lemma: SC PL}, \ref{lemma: equiv of PL EB QG}, 
and the second term follows Lemma~\ref{lemma: distance of yg ygam}; 
and (e) follows the linear convergence result of PL function \citep{karimi2016linear} as $y_t^\gamma$ is the results from $K$-step inner update starting at $y_{t-1}^\gamma$. In this way, when taking $K$ sufficiently large, there is 
\begin{align*}
    \| b_t\|\leq \Oc(\| \nabla_y f(x,y_g^*(x_t))\| = \Oc(l_{f,0}).
\end{align*}

Moreover, $F_\gamma(x)$ is $l_{F_\gamma,1}=\Oc(1)$-smooth. Therefore, following \eqref{eq: G eta upper bound},
\begin{align*}
    \frac{1}{T}\sum_{t=0}^{T-1} \|G_{F_\gamma, \mathcal{X}}(x_t)\|^2 &\leq  \Oc (\eta^{-1} T^{-1})+ \Oc\Big(  \frac{1}{T} \sum_{t=0}^{T-1}\|b_t\|^2 \Big) \nonumber\\
    &\leq \Oc(T^{-1}) + \Oc\Big(\frac{1}{T}\sum_{t=0}^T \|  \nabla_y f(x,y_g^*(x_t))\|^2\Big)=\Oc (T^{-1})+\Oc(l_{f,0}^2). 
\end{align*}
where $\Oc(1)$-smooth allows choosing a $\eta = \Oc(1)$ that satisfies $\eta\leq l_{F_\gamma,1}^{-1}$
    In this way, for sufficiently large $T$, $\frac{1}{T}\sum_{t=0}^{T-1} \|\nabla G_{F_\gamma, \mathcal{X}}(x_t)\|^2 \leq \Oc(l_{f,0}^2)$. 
    
Next, we will prove the lower bound of Algorithm \ref{alg: PBGD-Free} by constructing a counterexample, and show that the upper bound is tight.
Consider $f(x,y) = x^2+l_{f,0}y$ and $g(x,y)=(y-x+1)^2$, $\mathcal{X}=\mathbb{R}^{d_x}$ and $\mathcal{Y}(x)=\mathbb{R}^{d_y}$. In this problem, $\nabla G_{F_\gamma, \mathcal{X}}(x)=\nabla F_\gamma(x)=2x+l_{f,0}$ while $\nabla_x f(x,y_\gamma^*(x))=2x$. Using fixed stepsize $\eta$, $x_{t+1}=x_t-\eta x_t = (1-\eta)x_t$, implying $\| \nabla_x f(x_{t+1},y_\gamma^*(x_{t+1}))\|=\| 2x_{t+1}\|=2(1-\eta)\|x_t\|=2(1-\eta)^{t+1}\|x_0\|$.
Therefore, for arbitrary small $\epsilon>0$, there exists some $T_0=\Oc(\ln(\epsilon^{-1}))$ such that Algorithm \ref{alg: PBGD-Free} converges to $\|x_t\| < \epsilon$ for all $t\geq T_0$, whereas $\nabla F_\gamma(x_t) = l_{f,0}$. In this way, we have
\begin{align*}
    \frac{1}{T-T_0}\sum_{t=T_0}^T \|\nabla F_\gamma(x_t)\|^2  = \Oc(\epsilon)+l_{f,0}^2. 
\end{align*}
In this case, $\nabla G_{F_\gamma, \mathcal{X}}(x_t)=\nabla F_\gamma(x_t)$. Therefore, it indicates that $\Theta(l_{f,0}^2)$ is a tight bound.
\end{proof}
Proposition \ref{prop: PBGD} illustrates that
PBGD-Free converges to the $\epsilon$ stationary point only when the bound for $\| \nabla_y f(x, y_g^*(x_t)) \|$ (a.k.a $\ell_{f,0}$) is small. However, this is difficult to guarantee even in scenarios where PBGD‑Free is effective, such as in representation learning based PEFT \eqref{eq: bilevel repre LLM}. This motivates us to explore a weaker condition than the small Lipschitz assumption on  $f(x,\cdot)$, one that is more likely to hold in practice.

\subsection{A relaxed condition: UL \texorpdfstring{$(\delta,\alpha)$}{(delta,alpha)}-flatness}
\label{sec: improved analysis under flatness cond}

In this section, We introduce a relaxed condition that is less restrictive, and therefore more general, than the uniform small-Lipschitz-constant assumption on $f(x,\cdot)$. 

\begin{definition}[{$(\delta,\alpha)$-flatness}]
\label{def: flatness}
Function $f(x,\cdot):\mathcal{Y}\rightarrow \mathbb{R}$ is called $(\delta,\alpha)$-flat with modulus $c\geq 0$ at $y_g^*(x) \in S_g(x)$ with $\delta\geq 0 , \alpha \geq 1$ if 
\begin{equation}
    |f(x,y_g^*(x))-f(x,y)| \leq c \|y_g^*(x) -y \|^\alpha+\delta,~~~\forall y\in \mathcal{Y}.
\end{equation}
\end{definition}
When $\delta = 0$ and $y_g^*(x)$ is replaced by an arbitrary $y'$, Definition~\ref{def: flatness} reduces to the standard Hölder condition. Under the Lipschitzness condition in 
Assumption~\ref{assumption: UL}.\ref{ass: UL y Lipschitz}, the function $f(x,\cdot)$ satisfies $(0,1)$-flatness with modulus $c = l_{f,0}$. However, setting $\delta = 0$ naturally imposes the constraint $\alpha \leq 1$ whenever $\nabla_y f(x, y_g^*(x)) \neq 0$. Unless otherwise specified, we assume $\delta>0$ and $\alpha >1$ when referring to flatness afterwards.

In the following, we discuss the relations of Lipschitz condition in Assumption \ref{assumption: UL}.\ref{ass: UL y Lipschitz} and the new flatness condition in Definition \ref{def: flatness} through several observations. 

\begin{observation}
\label{obs:small_grad_gives_flat}
Under the $l_{F_\gamma,1}$-smoothness condition of $f(x, \cdot)$, if $\| \nabla_y f(x, y_g^*(x))\| = \delta^{\frac{1}{\alpha}}$, then $f(x,\cdot)$ is $(\delta, \alpha)$-flat with some modulus $0\leq c \leq \Oc(l_{F_\gamma,1})$. 
\end{observation}
\begin{proof}
    For $\|y-y_g^*(x)\|\geq 1$, by the $l_{f,0}=\delta^{\frac{1}{\alpha}}$-Lipschitzness of $f(x,y)$ in $y$, there is 
\begin{align*}
    \| f(x,y)-f(x,y_g^*(x))\|\leq l_{f,0}\|y-y_g^*(x)\| \leq \delta^{\frac{1}{\alpha}}\|y-y_g^*(x)\| ^\alpha. 
\end{align*}
For small $\|y-y_g^*(x)\|< 1$, Taylor's expansion gives
\begin{align*}
    f(x,y)-f(x,y_g^*(x))= \langle \nabla_y f(x,y_g^*(x)), y-y_g^*(x)\rangle+ R(x,y). 
\end{align*}
Here, $R(x,y)$ is a remainder. By Hölder-Continuous Gradient Condition \citep{berger2020quality}, \citep[Section 2]{nesterov2013introductory}, which is implied by the smoothness, there exists some $0\leq c\leq l_{F_\gamma,1}/2$, $1<\alpha < 2$ such that $\| R(x,y)\|\leq c\|y-y_g^*(x)\|^{\alpha}$. By Cauchy-Schwarz's inequality, there is 
\begin{align*}
    \|\langle \nabla_y f(x,y_g^*(x)), y-y_g^*(x)\rangle\|\leq & \| \nabla_y f(x,y_g^*(x))\| \| y-y_g^*(x)\| \nonumber \\
    \leq & \delta^{\frac{1}{\alpha}}\|y-y_g^*(x)\| \nonumber\\
    \leq & \delta +\|y-y_g^*(x)\|^\alpha
\end{align*}
where the last inequality holds as for $a,b\in (0,1)$ and $\alpha\in(1,2)$, there is $ab\leq \max\{a,b\}^2\leq \max\{a,b\}^\alpha \leq a^\alpha +b^\alpha$ and here $a=\delta^{1/\alpha}$ and $b=\| y-y_g^*(x)\|$.
\end{proof}
Observation~\ref{obs:small_grad_gives_flat} demonstrates that assuming small $\|\nabla_y f(x, y_g^*(x))\|$ is stronger than assuming flatness since $\ell_{f,0}=\delta^{\frac{1}{\alpha}}>\delta$ when $\alpha>1$. Below, we will show that the flatness condition automatically holds near the LL optimal solution $y_g^*(x)$. 

\begin{observation}
\label{obs: local abrupt change}
Under the $l_{F_\gamma,1}$-smoothness condition of $f$, the $(\delta, \alpha)$-flatness condition holds automatically for all $y \in \{ y: |f(x,y) - f(x, y_g^*(x))| \leq \delta \}$.
\end{observation}
Since $f$ is continuous and smooth, this observation implies that the flatness condition permits abrupt, unstable changes in the $\Oc(\delta)$-neighborhood of $y_g^*(x)$. This demonstrates that the flatness condition is relatively mild and further confirms that it is strictly weaker than the small Lipschitz condition, which explicitly requires $\| \nabla_y f(x,y_g^*(x))\|$ to be small. 
Fig.~\ref{fig:flatness_ex2} visualizes Observation \ref{obs: local abrupt change} through the following example.

\begin{example}
\label{example:toy_example_3}
We consider the LL objective $g(x,y)=(y-x)^2$ and the UL objective 
\vspace{-0.1cm}
\begin{align*}
f(x, y) = &  
\left( \sin(y - x) + 2 \right) |y - x|^2 + 10 \exp\left( -\frac{(y - x)^2}{2 (0.005)^2} \right) \sin\left(100(y - x)\right).
\end{align*}
The LL problem $g(x,y)$ is strongly convex in $y$, with $y_g^*(x) = x$. 
Therefore, $\nabla_y f(x, y_g^*(x))=1000$ is extremely large, implying $l_{f,0}\geq 1000$ and leading to a loose upper bound for $\|y_g^*(x)-y_\gamma^*(x)\|$ or $\|\phi(x)-F_\gamma(x)\|$ following Lemma~\ref{lemma: distance of yg ygam}. However,
this problem is $(3e^{-3}, 1.1)$-flat with $c = 5$ at $y_g^*(x)=x$ for $x\in[-10,10]$. 
As shown in Fig.~\ref{fig:flatness_ex2}, $f(x, \cdot)$ exhibits fluctuations around $y_g^*(x)$ while remaining relatively stable elsewhere. This shows that flatness is weaker than requiring small $\|\nabla_y f(x, y)\|$.
\end{example}

Additionally, owing to the Hölder-alike nature, the following observation shows that outside of the $\Oc(\delta)$ neighborhood, the curvature of the flatness condition is also milder than the Lipschitz condition. 

\begin{figure}[t]
\centering
\begin{minipage}{0.57\textwidth}
\centering
\includegraphics[width=1\linewidth]{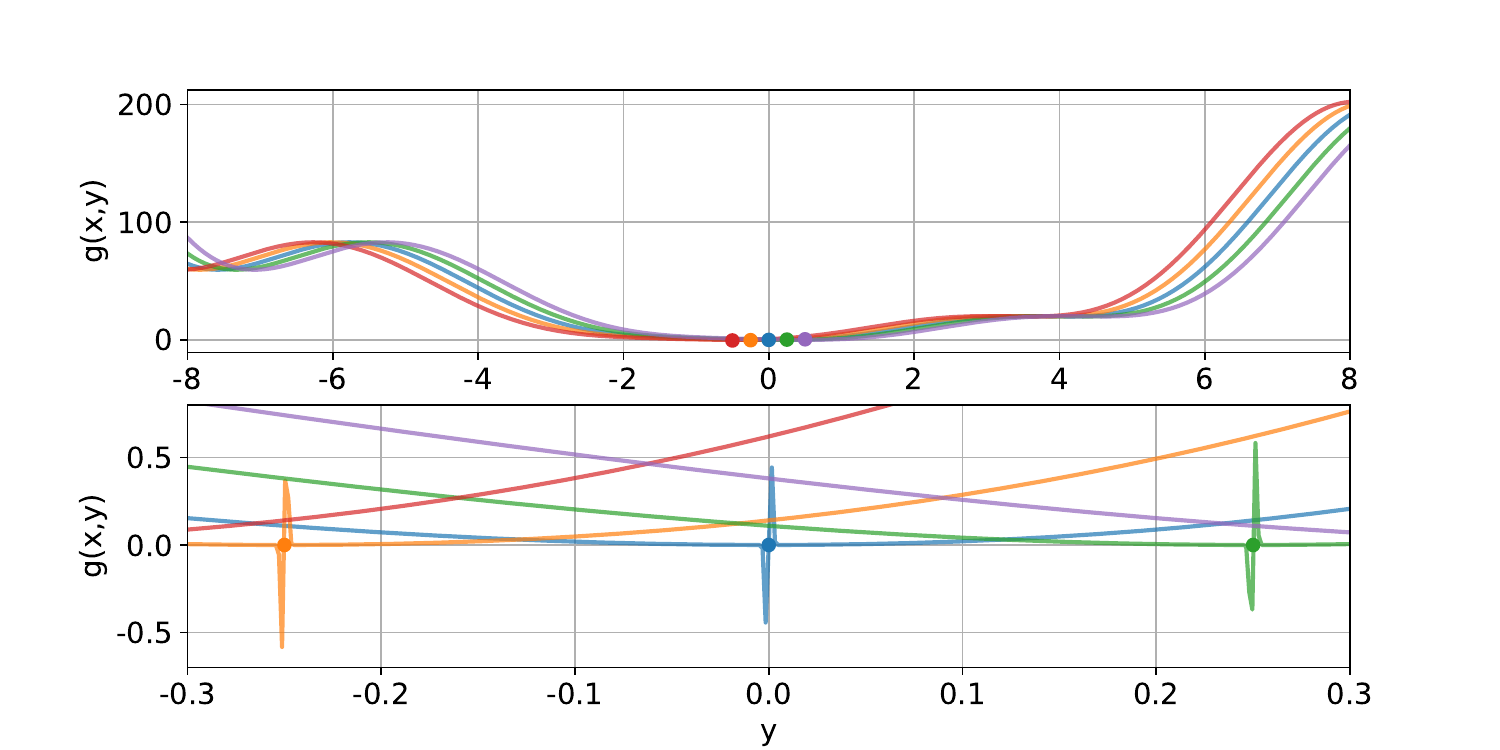}
\caption{Visualization of $f(x, y)$ in Example~\ref{example:toy_example_3}. Colored curves represent $f(x, \cdot)$ for different $x$; dots show $(y_g^*(x), f(x, y_g^*(x)))$. The \textbf{upper} plot shows $f(x,\cdot)$ on a larger scale, and the \textbf{lower} one illustrates the fluctuation around $y_g^*(x)$. }
\label{fig:flatness_ex2}
\end{minipage}%
\hfill
\begin{minipage}{0.4\textwidth}
\centering
\includegraphics[width=\linewidth]{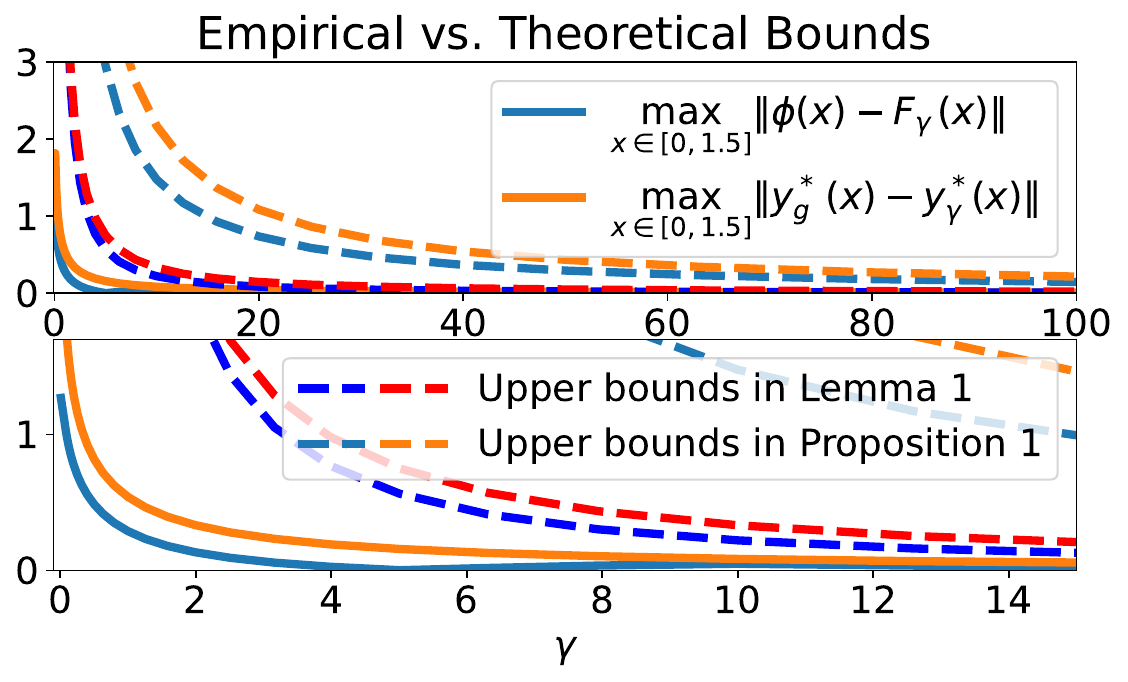}
\caption{Empirical bounds for $\|\phi(x) - F_\gamma(x)\|$ and $\|y_\gamma^*(x) - y_g^*(x)\|$ versus theoretical upper bounds in Proposition~\ref{lemma: distance of yg ygam} and Lemma~\ref{lemma: tighter bound for y gam yg} for the illustration of  representation learning PEFT \eqref{eq: bilevel repre LLM}. 
The lower plot shows a smaller scale. 
}
\label{fig:toy_example_bounds}
\end{minipage}
\end{figure}

\begin{observation}
Under $(\delta,\alpha)$-flatness, the growth rate of $f(x,\cdot)$ outside the $\Oc(\delta)$ neighborhood is
\begin{align}
\frac{|f(x, y) - f(x, y_g^*(x))|}{\| y - y_g^*(x)\|} \leq 
\begin{cases}
\Oc(1), & \text{if } \mathcal{O}(\delta) \leq \| y - y_g^*(x) \| \leq \Oc(1), \\
\mathcal{O}\left(\| y - y_g^*(x) \|^{\alpha - 1}\right), & \text{if } \| y - y_g^*(x) \| > \Oc(1).
\end{cases}
\end{align}
\end{observation}

This is obtained by dividing both sides of the flatness inequality by $\| y_g^*(x)-y\|$. For small $\| y_g^*(x)-y\|$, the second term dominates and leads to a ${\cal O}(1)$ bound, which is the same as the Lipschitz condition. However, for
large $\| y_g^*(x)-y\|$, since $\alpha >1$, the bound $\Oc(\| y_g^*(x) - y \|^{\alpha-1})$ can be larger than ${\cal O}(1)$.
This observation further demonstrates that the flatness condition relaxes the Lipschitzness of $f(x, \cdot)$ in Assumption \ref{assumption: UL}.\ref{ass: UL y Lipschitz}. Specifically, while Lipschitz continuity would require a uniform bound on the gradient, flatness allows for a higher growth rate of $\Oc(\| y - y_g^*(x) \|^{\alpha-1})$. 
For UL objective $f(x,\cdot)$ with fixed $x$, given a \textit{pre-determined $\alpha$ and modulus $c$}, the $\delta$ constant for flatness condition in Definition \ref{def: flatness} can be calculated via
\begin{align}
    \delta(x):= \max\{0,|f(x,y_g^*(x))-f(x,y_\gamma^*(x))|-c\| y_g^*(x)-y_\gamma^*(x)\|^\alpha\}. \label{eq: delta x}
\end{align}
When $\| y_\gamma^*(x) - y_g^*(x) \| > {\cal O}(1)$, the last term in \eqref{eq: delta x} dominates and $\delta(x)$ can effectively be $0$. Therefore, together with Observation \ref{obs: local abrupt change}, the flatness condition with small $\delta$ not only encompasses a broader function class than small Lipschitz continuous functions, but is easier to hold in practice. For example, modern loss functions used in deep learning, such as cross-entropy, squared error, or exponential losses, are nonlinear and locally curved. Around $y_g^*(x)$, we can write
$f(x,y') \approx f(x,y_g^*(x)) + c \|y' - y_g^*(x)\|^\alpha$ for some $\alpha > 1$ and constant $c>0$.
In such cases, the additive term in \eqref{eq: delta x} vanishes and $\delta(x)$ is effectively zero. 
This implies that the flatness condition can hold even when no Lipschitz bound on $f(x,\cdot)$ is available, particularly for locally curved objectives.
We next illustrate this behavior concretely through a parameter-efficient fine-tuning (PEFT) problem in representation learning.

\subsection{The flatness of the representation learning PEFT problem} 
\label{section: flatness of the representation learning PEFT}
In our PEFT framework in \eqref{eq: bilevel repre LLM}, the model, which can be any structure (e.g. CNN), is parameterized with $(x,y)$ by:
\begin{align}
    \pi_{x,y}(r | z) := \text{softmax}(\text{model}_{x,y}(z))_r.
\end{align}
It gives the model’s predicted probability for response $r$ given input question $z$. The DPO loss \citep{rafailov2023direct} over preference data $\mathcal{D}_{\text{DPO}}$, compares outputs $\pi_{x,y}$ against a reference $\pi_{\text{ref}}$ via 
\begin{equation}
    f_{\text{DPO}}(x, y) := -\frac{1}{|\mathcal{D}_{\text{DPO}}|} \sum_{(z, r_w, r_\ell) \in \mathcal{D}_{\text{DPO}}} 
\log \left( \sigma \left( 
q_\beta(x, y; z, r_w, r_\ell) \right) \right),
\end{equation}
where 
\begin{align}
    q_\beta(x, y; z, r_w, r_\ell) :=  
\beta \log \frac{\pi_{x,y}(r_w \mid z)}{\pi_{\text{ref}}(r_w \mid z)} 
- \beta \log \frac{\pi_{x,y}(r_\ell \mid z)}{\pi_{\text{ref}}(r_\ell \mid z)},
\end{align}
and $r_w$, $r_\ell$ are the preferred and rejected responses to input $z$.

The SFT loss operates on supervised dataset $\mathcal{D}_{\text{SFT}}$ through
\begin{align}
    g_{\text{SFT}}(x,y):= -\frac{1}{|\mathcal{D}_{\text{SFT}}|} \sum_{(z,r_{\text{SFT}}) \in \mathcal{D}_{\text{SFT}}} 
\log \left( \pi_{x,y}(r_{\text{SFT}}|z)\right) .
\end{align}
\vspace{-0.2cm}
Both objectives are differentiable with the following gradients 
\begin{align}
\nabla f_{\text{DPO}} &= -(1-\sigma(q_\beta))\nabla q_\beta,\quad  \nabla g_{\text{SFT}} = -\nabla\pi/\pi,
\end{align}
where $\nabla$ is short for gradient with respect to model parameters $x,y$.

Although the Lipschitz constant $\ell_{f,0}$ for this problem is not negligible, the objective still satisfies the flatness condition with a small $\delta$. To illustrate this point, we revisit the example summarized in Fig.~\ref{fig:toy_example_convergence_results} and detailed in Sec.~\ref{app:toy_example}. In this setting, for $c=0.5$ and $\alpha=1.5$, the values of $\delta(x)$ plotted over the iterations of $x$ (shown in Figure~\ref{fig:toy_example_lipschitz}) remain below $0.0003$ throughout the optimization, even though the Lipschitz constant of $f(x,\cdot)$ is large. 

\begin{wrapfigure}{r}{0.4\textwidth}  
\centering  
\vspace{-0.4cm}
\includegraphics[width=0.95\linewidth]{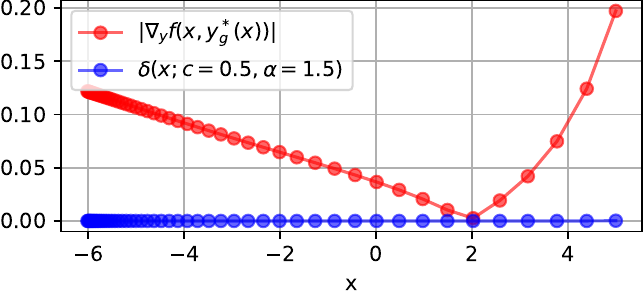}  
\caption{Comparisons of $\delta(x)$ and $\nabla_y f(x,y_g^*(x))$ during PBGD-Free updates. The Lipschitz constant $\ell_{f,0}=\max_x \|\nabla_y f(x,y_g^*(x))\|$ is large but $\delta(x)$ is small. }  
\label{fig:toy_example_lipschitz}
\vspace{-0.8cm}
\end{wrapfigure}
This shows that an analysis based solely on Lipschitz continuity can be overly conservative: the constant $l_{f,0}$ stays non-negligible even in local neighborhoods of the converged $\lim_{t\rightarrow \infty}x_t$, whereas the flatness condition provides a tighter characterization. The small $\delta(x)$ values observed along the PBGD-Free trajectory confirm that \textit{the PEFT problem \eqref{eq: bilevel repre LLM} satisfies the flatness condition}, which in turn motivates our enhanced analysis of the PBGD-Free algorithm under this condition. In practical PEFT problems, such as those reported in Sec.~\ref{sec: exp}, $\delta(x)$ in \eqref{eq: delta x} is typically small because $y_g^*(x)$ and $y_\gamma^*(x)$ are often more than one unit apart, while their effect on $f_{\text{DPO}}$ is marginal.

\subsection{Convergence analysis for PBGD-Free under flatness}
\label{sec:PBGD-Free}

As shown in \eqref{bias_value}, the first term is the major bottleneck of the divergence issue of the PBGD-Free algorithm under the Lipschitz condition in Proposition \ref{prop: PBGD}. The key to establishing the convergence guarantee for the PBGD-Free algorithm is the tighter bound of $\|y_g^*(x) - y_\gamma^*(x)\|$ and $\| \phi(x)-F_\gamma(x)\|$ under the flatness condition, compared to the results in Lemma~\ref{lemma: distance of yg ygam}. We highlight the results as follows. 

\begin{lemma}[\textbf{Tighter analysis on function value gap}]
\label{lemma: tighter bound for y gam yg} 
Suppose $\mathcal{Y}(x)=\tilde{\mathcal{Y}}$ is a closed, convex, and non-empty uncoupled LL constraint, and  Assumption \ref{assumption: UL}.\ref{ass: UL smooth}, \ref{assumption: LL}.\ref{ass: LL SC} hold. For fixed $x$, suppose $f(x,\cdot)$ is $(\delta,\alpha)$-flat at any $y_g^*(x) \in S_g(x)$ with $\alpha\in (1,1.5]$. Then, for $\gamma \geq \frac{l_{F_\gamma,1}}{\mu_g}$,
\begin{align}
\| \phi(x)-F_\gamma(x) \| =& \Oc(\gamma^{-\frac{\alpha}{2-\alpha}}+\delta), \quad  \text{and}\quad  
\|y_g^*(x)-y_\gamma^*(x)\| =  \Oc\big(\gamma^{-\frac{1}{2-\alpha}}+\delta^{\frac{1}{2}}\gamma^{-\frac{1}{2}}\big).
\label{eq: yg ygam tighter bound}
\end{align}
\end{lemma}
\begin{remark}
    The assumption Assumption \ref{assumption: UL}.\ref{ass: UL smooth}, \ref{assumption: LL}.\ref{ass: LL SC} and $\gamma \geq \frac{l_{F_\gamma,1}}{\mu_g}$ can be relaxed to $\gamma^{-1}f(x,y)+g(x,y)$ being proximal $\mu$-PL in $y\in \tilde{\mathcal{Y}}$.
\end{remark}
\begin{proof} 
For any $y_\gamma^*(x)\in S_\gamma(x)$, $y_g^*(x)\in S_g(x)$, there is
\begin{align}
    \gamma^{-1}f(x,y_\gamma^*(x))+g(x,y_\gamma^*(x)) \leq & \gamma^{-1}f(x,y_g^*(x))+g(x,y_g^*(x)) \nonumber \\
    \Rightarrow \quad \gamma^{-1}f(x,y_\gamma^*(x))+g(x,y_\gamma^*(x))-v(x) \leq & \gamma^{-1}f(x,y_g^*(x))+g(x,y_g^*(x))-v(x) \nonumber \\
    \Rightarrow \quad \gamma^{-1}f(x,y_\gamma^*(x))+g(x,y_\gamma^*(x))-v(x) \leq & \gamma^{-1}f(x,y_g^*(x)) \nonumber \\
    \Rightarrow \quad f(x,y_\gamma^*(x))\leq & f(x,y_g^*(x)). \label{eq: f x y gam small}
\end{align}
In this way, according to the definition of $\phi(x)$ and $F_\gamma(x)$, we have 
\begin{align}
    \|\phi (x)-F_\gamma(x)\|= & \min_{z\in S_g(x)}f(x,z)-\left(f(x,y_\gamma ^*(x))+\gamma(g(x,y_\gamma ^*(x))-v(x))\right) \nonumber \\
     \stackrel{(a)}{\leq} & f(x,y_g^*(x))-\left(f(x,y_\gamma ^*(x))+\gamma(g(x,y_\gamma ^*(x))-v(x))\right) \nonumber \\
    \stackrel{(b)}{\leq} & f(x,y_g^*(x))-\left(f(x,y_\gamma ^*(x))+\gamma\frac{\mu_g}{2}\|y_g^*(x) -y_\gamma^*(x) \|^2\right) \nonumber \\
    \stackrel{(c)}{=}  &\|f(x,y_g^*(x))-f(x,y_\gamma ^*(x))\|-\gamma\frac{\mu_g}{2}\|y_g^*(x) -y_\gamma^*(x) \|^2 \nonumber\\
    \stackrel{(d)}{\leq} & c\| y_g^*(x) -y_\gamma^*(x) \|^\alpha +\delta-\gamma\frac{\mu}{2}\|y_g^*(x) -y_\gamma^*(x) \|^2 \nonumber \\
    \stackrel{(e)}{\leq} & \max_{z:z\geq0} c z^\alpha -\gamma\frac{\mu}{2} z^2 \nonumber +\delta \\
    \stackrel{(f)}{=} & c^{\frac{2}{2-\alpha}}(2\alpha)^{\frac{\alpha}{2-\alpha}}(1-\alpha/2)(\mu\gamma)^{-\frac{\alpha}{2-\alpha}}+\delta = \Oc(\gamma^{-\frac{\alpha}{2-\alpha}}+\delta) \label{eq: first part of tighter bound proof}
\end{align}
Here, (a) holds for arbitrary $y_g^*(x)\in S_g(x), y_\gamma^*(x)\in S_\gamma(x)$ by \eqref{eq: f x y gam small}; 
(b) is from the $\mu_g$-QG condition of $g(x,\cdot)$ which is implied by $\mu_g$-PL according to Lemma \ref{lemma: equiv of PL EB QG}, via choosing $y_g^*(x)=\arg\min_{z\in S_g(x)}\|z-y_\gamma^*(x)\|$; (c) again uses \eqref{eq: f x y gam small}; (d) follows the flatness of $f(x,y)$ at $y_g^*(x)$, (e) is by formulating the problem as a maximization problem over $z=\| y_g^*(x) -y_\gamma^*(x) \|$; and (f) is the solution to this polynomial problem. Therefore, the first part is proved.

For the second part, as $\frac{1}{\gamma}f(x,\cdot)+g(x,\cdot)$ being $\mu$-PL for $\gamma\geq \frac{l_{F_\gamma,1}}{\mu_g}$, it is also $\mu$-QG by Lemma \ref{lemma: equiv of PL EB QG}. In this way, fixed any $\gamma\geq \frac{l_{F_\gamma,1}}{\mu_g}$, for any $y_\gamma^*(x) \in S_\gamma(x)$ and any $y_g^*(x) \in S_g(x)$, there is
\begin{align}
     &\gamma\left(  \left( \frac{1}{\gamma}f(x,y_g^*(x))+g(x,y_g^*(x)) \right) -\left(\frac{1}{\gamma}f(x,y_\gamma ^*(x))+g(x,y_\gamma ^*(x))\right)\right)
    \geq  \gamma \frac{\mu}{2} d_{S_\gamma(x)}^2(y_g^*(x)). \label{eq: QG of H}
\end{align}
Moreover, following steps (a)-(d) as in \eqref{eq: first part of tighter bound proof}, there is
\begin{align*}
    \text{left of \eqref{eq: QG of H}}= & \left( f(x,y_g^*(x))+\gamma g(x,y_g^*(x)) - \gamma v(x) \right) - \left( f(x,y_\gamma ^*(x))+\gamma g(x,y_\gamma ^*(x)) - \gamma v(x) \right) \nonumber \\
     =& f(x,y_g^*(x))-f(x,y_\gamma ^*(x))-\gamma \big(g(x,y_\gamma ^*(x))-v(x)\big) \nonumber \nonumber\\
     \leq & c\| y_g^*(x) -y_\gamma^*(x) \|^\alpha +\delta-\gamma\frac{\mu}{2} d_{S_\gamma(x)}^2(y_\gamma^*(x)) 
\end{align*}
Combining \eqref{eq: QG of H} and the above, there is
\begin{align*}
    c\| y_g^*(x) -y_\gamma^*(x) \|^\alpha +\delta-\gamma\frac{\mu}{2} d_{S_g(x)}^2(y_\gamma^*(x)) \geq \gamma \frac{\mu}{2} d_{S_\gamma(x)}^2(y_g^*(x))
\end{align*}
for any $y_g^*(x)\in S_g(x)$ and $y_\gamma^*(x)\in S_\gamma(x)$. In this way, for any $y_g^*(x)\in S_g(x)$, choose $y_\gamma^*(x)=\arg\min_{y\in S_\gamma(x)} \|y-y_g^*(x)\|$, there is
\begin{align*}
    cd_{S_\gamma(x)}^\alpha (y_g^*(x)) +\delta \geq & \gamma \frac{\mu}{2} d_{S_\gamma(x)}^2(y_g^*(x))
\end{align*}
Similarly, for any $y_\gamma^*(x)\in S_\gamma(x)$, choose $y_g^*(x)=\arg\min_{z\in S_g(x)} \|z-y_g^*(x)\|$, there is
\begin{align*}
    cd_{S_g(x)}^\alpha (y_\gamma^*(x)) +\delta \geq & \gamma \frac{\mu}{2} d_{S_g(x)}^2 (y_\gamma^*(x)).
\end{align*}
For simplicity, denote $x=d_{S_g(x)} (y_\gamma^*(x))$ (or $x=d_{S_\gamma(x)} (y_g^*(x))$), there is
\begin{align*}
    \quad x^{2-\alpha} \leq & 2c {\mu}^{-1}\gamma^{-1} + 2\delta {\mu}^{-1}\gamma^{-1} x^{-\alpha}.
\end{align*}

As $\alpha\in(1,1.5]$, for $x \geq \sqrt{\frac{\delta}{\gamma}}$,
\begin{align*}
x^{2-\alpha} \leq & 2c {\mu}^{-1}\gamma^{-1} + 2\delta {\mu}^{-1}\gamma^{-1} \left(\frac{\delta}{\gamma} \right)^{-\frac{\alpha}{2}}.
\end{align*}

Since $| a+b|^p \leq 2^{p-1}(|a|^p+|b|^p)$ for all $p\geq 1$ (as $|\cdot|^p$ is convex), there is
\begin{align*}
x =&  (x^{2-\alpha})^{\frac{1}{2-\alpha}}\leq  \left(2c {\mu}^{-1}\gamma^{-1} + 2\delta {\mu}^{-1}\gamma^{-1} \left(\frac{\delta}{\gamma} \right)^{-\frac{\alpha}{2}} \right)^{\frac{1}{2-\alpha}} \nonumber\\
\leq & 2^{\frac{1}{2-\alpha}-1} \left((2 c\mu^{-1})^{\frac{1}{2-\alpha}} \gamma^{-\frac{1}{2-\alpha}} + (2\mu^{-1})^{\frac{1}{2-\alpha}}\delta^{\frac{1}{2}}\gamma^{-\frac{1}{2}}\right)=Oc(\gamma^{-\frac{1}{2-\alpha}}+\delta^{\frac{1}{2}}\gamma^{-\frac{1}{2}})
\end{align*}

In this way, we can conclude the following to include the scenario that $x \leq \sqrt{\frac{\delta}{\gamma}}$.
\begin{align*}
    x =\Oc(\gamma^{-\frac{1}{2-\alpha}}+\delta^{\frac{1}{2}}\gamma^{-\frac{1}{2}}).
\end{align*}
\end{proof}

When $\delta$ is smaller than target accuracy $\epsilon$, achieving $\| \phi(x)-F_\gamma(x)\|,\|y_g^*(x)- y_\gamma^*(x) \|^2 =\Oc( \epsilon)$ only requires $\gamma = \Oc(\epsilon^{-\frac{2-\alpha}{2}})$, which is strictly smaller than the choice of $\gamma = \Oc(\epsilon^{-\frac{1}{2}})$ in previous literature \citep{shen2023penalty,chen2023bilevel,kwon2023penalty,kwon2023fully}. This also aligns with common practice, where the penalty constant $\gamma$ does not need to be excessively large. For instance, the UL objective in Example \ref{example:toy_example_3} is $(10^{-3},1.1)$-flat and therefore choosing $\gamma = 15$
gives desired accuracy, supporting the rule of thumb: $\gamma \approx 15$ is a reasonable choice. In Fig.~\ref{fig:toy_example_bounds}, we also show that our bound under the flatness condition in Lemma~\ref{lemma: tighter bound for y gam yg} is tighter than the one under the Lipschitz condition in Proposition \ref{lemma: distance of yg ygam} for the representation learning PEFT \eqref{eq: bilevel repre LLM}.

Since Lemma~\ref{lemma: tighter bound for y gam yg} provides a per-iterate bound with fixed $x$, the next step is to analyze the Lipschitz continuity of the flatness constant $\delta (x)$ with respect to $x$, enabling a uniform bound across iterations. 

\begin{lemma}[\textbf{Lipschitz continuity of flatness constant $\delta(x)$}]
\label{lemma: Lipschitz delta x}
Suppose $\mathcal{Y}(x)=\tilde{\mathcal{Y}}$ is closed, convex and non-empty uncoupled LL constraint, and Assumption \ref{assumption: UL}.\ref{ass: UL y Lipschitz}--\ref{ass: UL smooth}, \ref{assumption: LL}.\ref{ass: LL SC}--\ref{ass: LL smooth}, \ref{assumption: Hessian Lipschitz} hold. Then fixing some $c\geq 0$ and $\alpha\in (1,2)$, there exists some trajectory of $y_g^*(x)$, $y_\gamma^*(x)$ such that the flatness constant of $f(x,\cdot)$, 
$\delta(x)$ defined in \eqref{eq: delta x}, is $\Oc(c\gamma^{-(\alpha-1)})$-Lipschitz-continuous.
\end{lemma}
\begin{proof}
    
We firstly show that $f(x,y_g^*(x))-f(x,y_\gamma^*(x))$ and $c\| y_g^*(x)-y_\gamma^*(x)\|^\alpha$ are both Lipschitz continuous.

For any $x$ and unit direction $d\in \mathbb{R}^{d_x}$, the directional derivative of the first term bounded
\begin{align*}
    & \Big \| \lim_{r\downarrow 0}\frac{f(x+rd,y_g^*(x+rd))-f(x,y_g^*(x))}{r} - \lim_{r\downarrow 0}\frac{f(x+rd,y_\gamma^*(x+rd))-f(x,y_\gamma ^*(x))}{r}\Big \| \nonumber \\
    \stackrel{(a)}{=} & \big \| \nabla_x f(x,y_g^*(x)) d  + \nabla_y f(x,y_g^*(x)) D_d(y_g^*(x)) \\
    & ~ - \left( \nabla_x f(x,y_\gamma^*(x)) d  + \nabla_y f(x,y_\gamma^*(x)) D_d(y_\gamma^*(x)) \right)\big \| \nonumber \\
    \stackrel{(b)}{\leq} & \big \| \nabla_x f(x,y_g^*(x)) d   - \nabla_x f(x,y_\gamma^*(x)) d \big\| + \big \| \nabla_y f(x,y_g^*(x)) \| \| D_d(y_g^*(x))-   D_d(y_\gamma^*(x)) \big \| \nonumber\\
    &+ \|D_d(y_g^*(x)) \|\|\nabla_y f(x,y_g^*(x))-\nabla_y f(x,y_\gamma^*(x))\|\nonumber \\
    \stackrel{(c)}{\leq} & \Oc(\gamma^{-1}).
\end{align*}
Here, (a) follows Taylor's expansion; (b) uses the triangle inequality and Cauchy-Schwartz inequality; and (c) applies smoothness of $f$, Lipschitzness of $f(x,\cdot)$, and Lemma \ref{lemma: distance of yg ygam}, \ref{lemma: implicit directional derivative lipschitz}, \ref{lemma: directional derivative bound}.

For any $x,x'\in \mathcal{X}$,
\begin{align*}
    & \big| c\| y_g^*(x)-y_\gamma^*(x)\|^\alpha - c\| y_g^*(x')-y_\gamma^*(x')\|^\alpha \big| \nonumber \\ 
    \stackrel{(a)}{\leq} & c \max_{z \in [ \| y_g^*(x)-y_\gamma^*(x)\|, \| y_g^*(x')-y_\gamma^*(x')\|]} z^{\alpha-1} \big|  \| y_g^*(x)-y_\gamma^*(x)\| - \| y_g^*(x')-y_\gamma^*(x')\| \big| \nonumber \\
    \stackrel{(b)}{\leq} & \Oc(c\gamma^{-(\alpha-1)}) \big \| \big(y_g^*(x)-y_\gamma^*(x) \big) -  \big(y_g^*(x')-y_\gamma^*(x') \big) \big\| \nonumber \\
    \stackrel{(c)}{\leq} & \Oc(c\gamma^{-(\alpha-1)}) \big( \| y_g^*(x)-y_g^*(x')\| +  \|y_\gamma^*(x)-y_\gamma^*(x') \| \big) \nonumber \\
    \stackrel{(d)}{=} & \Oc(c\gamma^{-(\alpha-1)}) \| x-x'\|
\end{align*}
where (a) follows the mean value theorem, as $|\cdot|^\alpha$ is continuous; (b) is from $\| y_g^*(x)-y_\gamma^*(x)\|=\Oc(\gamma^{-1})$ for all $x$, and the $1$-Lipschitzness of the norm function; (c) uses triangle-inequality; and (d)
is achieved by knowing that $y_g^*(x)$ and $y_\gamma^*(x)$ are, respectively, $L_g$, $L_\gamma$-Lipschitz for some constant $L_g$, $L_\gamma$ \citep{shen2023penalty}. 

In this way, we can conclude that 
\[
\delta'(x)=f(x,y_g^*(x))-f(x,y_\gamma^*(x))-c\| y_g^*(x)-y_\gamma^*(x)\|^\alpha
\]
is $\Oc(c\gamma^{-(\alpha-1)})$-Lipschitz-continuous. As $\delta(x)$ is a ReLu function on $\delta'(x)$, it is also $\Oc(c\gamma^{-(\alpha-1)})$-Lipschitz-continuous.
\end{proof}
In this way, Lemma~\ref{lemma: tighter bound for y gam yg} and Lemma~\ref{lemma: Lipschitz delta x} enable the convergence of PBGD-Free to the stationary point of the penalized objective $F_\gamma(x)$ if the flatness condition of $f$ holds around the stationary point of $F_\gamma(x)$ and its corresponding LL optimal value.

\begin{theorem}[\textbf{Convergence of PBGD-Free}]
\label{thm: no value algorithm convergence}
Suppose $\mathcal{Y}(x)=\{y\in \mathcal{Y}:c(y)\leq 0 \}$, and Assumption \ref{assumption: UL}.\ref{ass: UL y Lipschitz}--\ref{ass: UL smooth}, \ref{assumption: LL}.\ref{ass: LL SC}--\ref{ass: LL smooth}, \ref{assumption: constraint}, \ref{assumption: Hessian Lipschitz} hold, and for all $x_t$ on the trajectory, $f(x_t,\cdot)$ is $(\delta(x_t),\alpha)$-flat at all $y_g(x_t)\in S_g(x_t)$ with the same $\alpha \in (1,1.5]$ and modulus $c=\Oc(1)$. For iterations generated by the \textcolor{blue}{fully single-loop version} of PBGD-Free (Algorithm~\ref{alg: PBGD-Free}) with step size $\eta \leq l_{F_\gamma,1}^{-1}$, 
suppose for target $\epsilon$, there exists $\delta$ such that $\frac{1}{T}\sum_{t=0}^{T-1}  \delta(x_t) \leq \delta$, it holds with $\gamma = \Oc(\delta^{-\frac{2-\alpha}{2}})$ that
\vspace{-0.1cm}
\begin{equation}
\frac{1}{T}\sum_{t = 0}^{T-1} \| G_{F_\gamma,\mathcal{X}}(x_t)\|^2 \leq \Oc(T^{-1} + \delta^{\frac{2(\alpha-1)}{\alpha}}).
\end{equation}
\end{theorem}
\begin{remark}
    The strong convexity assumption in Theorem~\ref{thm: no value algorithm convergence} is required to establish the $\Oc(1)$-smoothness of $F_\gamma(x)$ in the constrained setting. In the unconstrained case, this assumption can be relaxed: the same result holds under the weaker condition that there exists some $\gamma^*>0$ such that, for any $x\in\mathcal{X}$, the function $\gamma^{-1}f(x,\cdot)+g(x,\cdot)$ satisfies the $\mu_g$-PL condition for all $\gamma \geq \gamma^*$ \citep{chen2024finding}, as discussed in Sec.~\ref{sec: Existing work smoothness of F gam}.
\end{remark}
\begin{proof}
    
Denote $g_t = \nabla_x f(x_t,y_{t+1}^\gamma)$. According to $l_{F_\gamma,1}$-smoothness of $F_\gamma(x)$, we have 
\begin{align}
    F_\gamma(x_{t+1})- F_\gamma(x_t)\leq 
    -\frac{1}{4\eta}\| x_{t+1}-x_t\|^2 + \eta  \| \nabla F_\gamma(x_{t})-g_t\|^2 \label{eq: VaFF fully single-loop intermediate step}
\end{align}
following a similar analysis as in \eqref{eq: descent intermediate step 1}.

Denote $\tilde{g}_\gamma(x,y)=\gamma^{-1}f(x,y)+g(x,y)$, $y_t^{\gamma,*} = \arg \min_{y\in \tilde{\mathcal{Y}}}\tilde{g}_\gamma(x_t,y)$ and $y_t^{g,*} = \arg \min_{y\in \tilde{\mathcal{Y}}}g(x_t,y)$, and the update bias $b(x_t) = \nabla F_\gamma(x_t) - g_t$. In this way,
\begin{align}
    \|b(x_t)\|^2 = & \| \nabla_x f(x_t,y_{t}^{\gamma,*}) + \gamma(\nabla_x g(x_t,y_{t}^{\gamma,*})-\nabla_x g(x_t,y_{t}^{g,*}))- \nabla_x f(x_t,y_{t+1}^{\gamma}) \|^2 \nonumber\\
    \stackrel{(a)}{\leq} & 2\| \nabla_x f(x_t,y_{t}^{\gamma,*}) - \nabla_x f(x_t,y_{t+1}^{\gamma}) \| +2\gamma^2\|\nabla_x g(x_t,y_{t}^{\gamma,*})-\nabla_x g(x_t,y_{t}^{g,*})\|^2 \nonumber\\
    \stackrel{(b)}{\leq}  & 2l_{F_\gamma,1}^2\|y_{t+1}-y_t^{\gamma} \|^2 + \Oc\big(\gamma^{-\frac{2(\alpha-1)}{2-\alpha}}+\delta \gamma \big)\nonumber\\
    \stackrel{(c)}{\leq}  & \frac{4}{\mu_\gamma^*}l_{F_\gamma,1}^2(\tilde{g}_\gamma(x_t,y_{t+1}^\gamma)-v_\gamma(x_t)) + \Oc\big(\gamma^{-\frac{2(\alpha-1)}{2-\alpha}}+\delta \gamma \big)\nonumber\\
    \stackrel{(d)}{\leq}  & \frac{4}{\mu_\gamma^*}l_{F_\gamma,1}^2 (1-\eta^\gamma \mu_g)(\tilde{g}_\gamma(x_t,y_{t}^\gamma)-v_\gamma(x_t))+ \Oc\big(\gamma^{-\frac{2(\alpha-1)}{2-\alpha}}+\delta \gamma \big) \label{eq: VaFF bias bound intermediate step}
\end{align}
where (a) applies the Young's inequality; (b) follows the smoothness of $f$ and Lemma~\ref{lemma: tighter bound for y gam yg}; (c) employs the property of strong convexity in Assumption \ref{assumption: LL}; and (d) is by the descent theory for applying projected gradient descent on problems satisfying PL condition, see e.g. \citep[Theorem 5]{karimi2016linear}.

Plugging \eqref{eq: VaFF bias bound intermediate step} in \eqref{eq: VaFF fully single-loop intermediate step}, there is
\begin{align*}
    & F_\gamma(x_{t+1})- F_\gamma(x_t)\nonumber\\
    \leq & -\frac{\eta}{4}  \Big \|\frac{x_{t+1}-x_t}{\eta} \Big \|^2 + \eta\frac{4}{\mu_\gamma^*}l_{F_\gamma,1}^2 (1-\eta^\gamma \mu_g)(\tilde{g}_\gamma(x_t,y_{t}^\gamma)-v_\gamma(x_t))+ \eta \Oc\big(\gamma^{-\frac{2(\alpha-1)}{2-\alpha}}+\delta \gamma \big). 
\end{align*}

Moreover, as $\tilde{g}_\gamma(x,y)$ is $l_{\tilde{g},1}$-smooth and $v_\gamma(x)$ is $l_{v_\gamma,1}$-smooth, we obtain
\begin{align}
    &\tilde{g}_\gamma(x_{t+1},y_{t+1}^\gamma)-v_\gamma(x_{t+1}) \nonumber \\
    \stackrel{(a)}{\leq} & \tilde{g}_\gamma(x_t,y_{t+1}^\gamma)-v_\gamma(x_t) + \langle \nabla_x \tilde{g}_\gamma(x_t,y_{t+1}^\gamma)-\nabla v_\gamma(x_t),x_{t+1}-x_t\rangle  \nonumber\\
    & +\frac{\eta^2(l_{\tilde{g},1}+l_{v_\gamma,1})}{2}\|\frac{ x_{t+1}-x_t}{\eta}\|^2 \nonumber \\
    \stackrel{(b)}{\leq} & \tilde{g}_\gamma(x_t,y_{t+1}^\gamma)-v_\gamma(x_t) + \eta l_{\tilde{g},1}\|y_{t+1}^\gamma-y_t^{\gamma,*}\| \|\frac{ x_{t+1}-x_t}{\eta}\| +\frac{\eta^2(l_{\tilde{g},1}+l_{v_\gamma,1})}{2}\|\frac{ x_{t+1}-x_t}{\eta}\|^2 \nonumber \\
    \stackrel{(c)}{\leq} & \tilde{g}_\gamma(x_t,y_{t+1}^\gamma)-v_\gamma(x_t) + \eta l_{\tilde{g},1}\frac{z}{2}\|y_{t+1}^\gamma-y_t^{\gamma,*}\|^2+ \frac{\eta l_{\tilde{g},1}}{2z}\|\frac{ x_{t+1}-x_t}{\eta}\|^2  \nonumber\\
    & +\frac{\eta^2(l_{\tilde{g},1}+l_{v_\gamma,1})}{2}\|\frac{ x_{t+1}-x_t}{\eta}\|^2 \nonumber \\
    \stackrel{(d)}{\leq} & (1+\frac{\eta l_{\tilde{g},1}z}{2})(\tilde{g}_\gamma(x_t,y_{t+1}^\gamma)-v_\gamma(x_t)) + (\frac{\eta l_{\tilde{g},1}}{2z}+ \frac{\eta^2(l_{\tilde{g},1}+l_{v_\gamma,1})}{2})\|\frac{ x_{t+1}-x_t}{\eta}\|^2 \nonumber \\
    \stackrel{(e)}{\leq} & (1+\frac{\eta l_{\tilde{g},1}z}{2})(1-\eta^\gamma \mu_g)(\tilde{g}_\gamma(x_t,y_t^\gamma)-v_\gamma(x_t))  \nonumber\\
    & + (\frac{\eta l_{\tilde{g},1}}{2z}+ \frac{\eta^2(l_{\tilde{g},1}+l_{v_\gamma,1})}{2})\|\frac{ x_{t+1}-x_t}{\eta}\|^2, \quad \forall z>0. \label{eq: h -vh update} 
\end{align}
Here, (a) follows the smoothness of $\tilde{g}_\gamma(x,y)+v_\gamma(x)$ in $x$; (b) applies Cauchy-Schwartz inequality and the smoothness of $h$ in $y$; (c) uses Young's inequality for any $z>0$; (d) is from the PL condition of $\tilde{g}_\gamma(x,y)$ in $y$, as guaranteed by the strong convexity of $\tilde{g}_\gamma(x,\cdot)$ when $\gamma \geq\frac{l_{F_\gamma,1}}{\mu_g}$ on closed convex set $\mathcal{Y}(x)$; (e) is similarly by the descent theory for applying projected gradient descent on $\tilde{g}_\gamma(x,\cdot)$ satisfying PL condition \citep[Theorem 5]{karimi2016linear}.

In this way, adding $c (\tilde{g}_\gamma(x_{t+1},y_{t+1}^\gamma)-v_\gamma(x_{t+1}))$
to both side of \eqref{eq: VaFF bias bound intermediate step}, there is
\begin{align*}
    & F_\gamma(x_{t+1}) + c(\tilde{g}_\gamma(x_{t+1},y_{t+1}^\gamma)-v_\gamma(x_{t+1})) \\
    \leq & F_\gamma(x_t) + \left( -\frac{\eta}{4} +c \Big(\frac{\eta l_{\tilde{g},1}}{2z}+ \frac{\eta^2(l_{\tilde{g},1}+l_{v_\gamma,1})}{2} \Big) \right)\|\frac{x_{t+1}-x_t}{\eta}\|^2 \\
    & + c\Big( \big(1+\eta \big(\frac{ l_{\tilde{g},1}z}{2} + l_{F_\gamma,1}^2 \frac{4}{\mu_\gamma^* c} \big) \big) (1-\eta^\gamma \mu_g)(\tilde{g}_\gamma(x_t,y_{t}^\gamma)-v_\gamma(x_t))  
    \Big) \\
    & +\eta \Oc\big(\gamma^{-\frac{2(\alpha-1)}{2-\alpha}}+\delta \gamma \big).
\end{align*}

In this way, choose the following hyper-parameter,
\begin{align}
    \begin{cases}
        c = \mu_g^{-\frac{1}{2}}\\
        z  = 8 c l_{\tilde{g},1} \\
        \eta^\gamma \leq l_{\tilde{g},1}^{-1}\\
        \eta \leq \min \left\{\frac{1}{8  c(l_{\tilde{g},1}+l_{v_\gamma,1})},\frac{\eta^\gamma \mu_g/(1-\eta^\gamma \mu_g)}{\frac{l_{\tilde{g},1}z}{2}+ \frac{4\ l_{F_\gamma,1}^2}{\mu_g c}} \right\}
    \end{cases} \label{eq: hyper-parameter choice for FSL VaFF}
\end{align}
i.e. $c  = \Oc(1)$, $\eta= \Oc(1)$, there is
\begin{align*}
    & F_\gamma(x_{t+1})- F_\gamma(x_t) + c (\tilde{g}_\gamma(x_{t+1},y_{t+1}^\gamma)-v_\gamma(x_{t+1}))\\
    \leq &  -\frac{\eta}{8} \|\frac{x_{t+1}-x_t}{\eta}\|^2 +c (\tilde{g}_\gamma(x_t,y_{t}^\gamma)-v_\gamma(x_t))  
    +\eta \Oc\big(\gamma^{-\frac{2(\alpha-1)}{2-\alpha}}+\delta \gamma \big).
\end{align*}
Denote $D_1 = F_\gamma(x_0)- F_\gamma(x_T)$, $D_2 = (\tilde{g}_\gamma(x_0,y_0^\gamma)-v_\gamma(x_0)) - (\tilde{g}_\gamma(x_T,y_T^\gamma)-v_\gamma(x_T))$. Rearranging and telescoping gives 
\begin{align}
    \frac{1}{T}\sum_{t=0}^{T-1} \Big\|\frac{x_{t+1}-x_t}{\eta} \Big\|^2 \leq &  \frac{8 \left(  D_1 + c D_2 \right)}{\eta T} + \Oc\big(\gamma^{-\frac{2(\alpha-1)}{2-\alpha}}+\delta \gamma \big) \nonumber \\
    = & \Oc( T^{-1} + \delta^{\frac{2(\alpha-1)}{\alpha}}) \label{eq: x update bound}
\end{align}
where the last equality is achieved as $c  = \Oc(1)$ and $\eta= \Oc(1)$ and by setting $\gamma = \Theta(\delta^{-\frac{2-\alpha}{\alpha}})$. Moreover, the hyper-parameter choices in \eqref{eq: hyper-parameter choice for FSL VaFF} reformulate \eqref{eq: h -vh update}, which can be plugged in \eqref{eq: VaFF bias bound intermediate step} to obtain
\begin{align}
    \frac{1}{T}\sum_{t=0}^{T-1} \| b_t\|^2  \leq & \frac{4}{\mu_\gamma^*}l_{F_\gamma,1}^2 (1-\eta^\gamma \mu_g)\frac{1}{T}\sum_{t=0}^{T-1} (\tilde{g}_\gamma(x_t,y_{t}^\gamma)-v_\gamma(x_t))+ \Oc\big( \delta^{\frac{2(\alpha-1)}{\alpha}} \big) \nonumber \\
    \leq &  \Oc\big(T^{-1} + \delta^{\frac{2(\alpha-1)}{\alpha}} \big). \label{eq: bias update bound}
\end{align}
where the last inequality follows by applying Abel's summation formula on series $\sum_{k=1}^K a_k b_k$ where $a_k=\Oc((1-\eta^\gamma \mu_g/2)^k)$ and $K^{-1}\sum_{k=0}^K b_k = \Oc( T^{-1} + \delta^{\frac{2(\alpha-1)}{\alpha}})$.

In this way, following a similar analysis as in \eqref{eq: G eta and x update}, there is
\begin{align*}
   \frac{1}{T}\sum_{t=0}^{T-1}  \| G_{F_\gamma,\mathcal{X}}(x_t)\|^2 \leq & \frac{1}{T}\sum_{t=0}^{T-1}  \Big\|\frac{x_t-x_{t+1}}{\eta} \Big\|^2 + \frac{1}{T}\sum_{t=0}^{T-1} \| b_t\|^2 \nonumber \\
   \leq &  \Oc\big(T^{-1} + \delta^{\frac{2(\alpha-1)}{\alpha}} \big)
\end{align*}
where the last inequality is obtained by plugging in \eqref{eq: x update bound} and \eqref{eq: bias update bound}.
\end{proof}

Theorem~\ref{thm: no value algorithm convergence} establishes the convergence rate of the fully single-loop version of PBGD-Free in Algorithm \ref{alg: PBGD-Free}. The result shows that the algorithm converges to the neighborhood of a stationary point for $F_\gamma (x)$, where the stationary gap is controlled by the flatness parameter $(\delta,\alpha)$. 

Specifically, for a $(\delta, \alpha)$-flat function with $\alpha \in (1,1.5)$, the convergence error scales as $\mathcal{O}(\delta^{\frac{2(\alpha-1)}{\alpha}})$, ensuring that the suboptimality gap remains small. For instance, when $(\delta, \alpha) = (0.0003, 1.5)$ in the representation learning PEFT problem, the resulting solution is $\mathcal{O}(0.004)$-optimal. Moreover, the method follows a single-loop update scheme, which is computationally more efficient than other fully first-order methods \citep{kwon2023fully,kwon2023penalty,ye2022bome,shen2023penalty,chen2023bilevel}, and their fully single-loop variant, e.g. ALT-PBGD with one step (projected) gradient descent update as \textsf{Min Solver}.

\begin{theorem}[Fully single-loop version of ALT-PBGD]
\label{thm: ALT-PBGD}
    Suppose all assumptions in Proposition \ref{prop: ALT-PBGD} hold. For iterations using the fully single-loop version of Algorithm \ref{alg: ALT-PBGD} with \textsf{Min Solver} taking one-step projected gradient descent with initial point $y_{t-1}^g$ for \eqref{eq: yg update} and $y_{t-1}^\gamma$ for \eqref{eq: y gam update}, we have
\begin{align}
    \frac{1}{T}\sum_{t = 0}^{T-1} \Big\|G_{F_\gamma,\mathcal{X}}(x_t) \Big\|^2 =\Oc(\gamma^2 T^{-1}).
\end{align}
With $\gamma = \Oc(\epsilon^{-2})$, the complexity for the fully single-loop version of Algorithm \ref{alg: ALT-PBGD} is $\Oc(\epsilon^{-2})$.
\end{theorem}

The proof structure for Theorem~\ref{thm: ALT-PBGD} shares great similarity to the one for Theorem~\ref{sec: Developing Value-Function-Free Algorithms: PBGD-Free}, thus deferred to Appendix~\ref{appendix: proof of ALT-PBGD FSL}.
Theorem~\ref{thm: ALT-PBGD} highlights that the convergence of the fully single-loop ALT-PBGD is hindered by a larger $\gamma$, which regulates the lower-level violation rate. This is due to the early inexact estimate bias of the penalty gradient term $\gamma(\nabla_x f(x_t,y_t^\gamma)-\nabla_x g(x_t,y_t^g))$ is scaled by $\gamma$. In contrast, PBGD-Free does not endure such scaled bias as it does not involve the penalty gradient term. An empirical verification for this is provided in Sec.~\ref{sec: Representation learning problem on NLSY dataset}.
A comparison of the proposed algorithm with state-of-the-art fully first-order BLO methods is provided in Table \ref{tab:table1-comparison}.

\section{Extension to General Coupled Constraints Setting}
\label{sec: BLOCC}

In this section, we study the general BLO problem with coupled inequality constraints, where the lower-level feasible set $\mathcal{Y}(x) = \{y \in \mathcal{Y} : c(x, y) \leq 0\}$ is defined as in \eqref{eq: original problem 1}. In Sec.~\ref{sec: improved convergence under CCs}, we extend the smoothness analysis to this coupled constraint setting and establish an improved complexity bound for the PBGD method tailored to such problems. In Sec.~\ref{sec: BLOCC-VaFF}, we further show that Algorithm~\ref{alg: PBGD-Free} remains efficient in solving BLO problems with coupled constraints.

\subsection{Improved convergence rate under coupled inequality constraint}
\label{sec: improved convergence under CCs}

As illustrated in Lemma~\ref{lemma: gradient of v}, in the coupled-constraint setting, $\nabla v(x)$ can be characterized via the lower-level solution $y_g^*(x)$ and its corresponding Lagrange multiplier $\lambda_g^*(x)$. In convex optimization, strong duality is guaranteed when Slater’s condition holds \citep{ito2008lagrange,ruszczynski2006nonlinear}. In this setting, Mangasarian–Fromovitz Constraint Qualification (MFCQ) is equivalent to Slater’s condition, and LICQ, being stronger, implies MFCQ \citep{wachsmuth2013licq}. Hence, under LICQ, strong duality holds and the Lagrange multiplier is uniquely defined. We summarize this fact as follows. 
\begin{lemma}
\label{lemma: equiv LL objective}
Suppose Assumption \ref{assumption: LL}.\ref{ass: LL SC} and \ref{assumption: constraint} hold. Then, for any $x \in \mathcal{X}$, strong duality holds for $\min_{y\in\mathcal{Y}(x)}g(x,y)$. In particular, there exists a unique multiplier $\lambda_g^*(x) = \arg\max_{\lambda\in \mathbb{R}^{d_x}_+}\{\min_{y\in \mathcal{Y}}g(x,y)+\langle \lambda, c(x,y)\rangle\}$
which is the unique Lagrange multiplier associated with the optimal solution $y_g^*(x)$ of the lower-level problem. Equivalently, the LL solution can be written as
\begin{align}
    y_g^*(x) = \arg\min_{y \in \mathcal{Y}} \Big\{ g^\lambda(x,y) := \underbrace{g(x,y)+\langle \lambda_g^*(x), c(x,y)}_{=:L_g(x,y,\lambda_g^*(x))}\rangle \Big\}. \label{eq: g lambda x y}
\end{align}
Here, the pair $(\lambda_g^*(x), y_g^*(x)) = \max_{\lambda \in \mathbb{R}_+^{d_c}}  \min_{y \in \mathcal{Y}}  L_g(x,y,\lambda)$ forms the unique saddle-point solution. Additionally, suppose Assumption~\ref{assumption: UL}.\ref{ass: UL y Lipschitz}--\ref{ass: UL smooth} hold. Then, $\gamma^{-1}f(x,\cdot)+g(x,\cdot)$ is strongly convex for $\gamma\geq\frac{l_{F_\gamma,1}}{\mu_g}$, and LICQ will also hold as $y_\gamma^*(x)$ lies in the neighborhood of $y_g^*(x)$ for large $\gamma$. In this way, the pair $(\lambda_\gamma^*(x), y_\gamma^*(x))$ forms the unique saddle-point solution of 
     \begin{align}
         \max_{\lambda \in \mathbb{R}_+^{d_c}}  \min_{y \in \mathcal{Y}} ~\{ L_\gamma(x,y,\lambda):= \gamma^{-f}f(x,y)+g(x,y)+\langle \lambda,c(x,y)\rangle\}.
     \end{align}
\end{lemma}
Therefore, the PBGD algorithm for solving BLO with Coupled Constraints, PBGD-BLOCC \citep{jiang2024primal}, was developed. This method follows an alternative update scheme to avoid the computationally expensive joint projection onto $\mathcal{X} \times \mathcal{Y}(x)$. 
Similar to ALT-PBGD for uncoupled constrained BLO, at each iteration $t$, PBGD-BLOCC finds the $\epsilon$-suboptimal solutions and their corresponding Lagrange multipliers via
\begin{align}
    (\lambda_{t}^g, y_{t}^g) \approx & \arg\max_{\lambda\in \mathbb{R}^{d_c}_+} \min_{y\in \mathcal{Y}} L_g(x_t,y,\lambda),  \quad \text{and} 
    \label{eq: lambda y g}\\
    (\lambda_{t}^\gamma, y_{t}^\gamma) \approx &\arg\max_{\lambda\in \mathbb{R}^{d_c}_+} \min_{y\in \mathcal{Y}} L_\gamma(x_t,y,\lambda). \label{eq: lambda y gam}
\end{align}

By \eqref{eq: F gam function} and Lemma~\ref{lemma: gradient of v}, we can estimate $\nabla F_\gamma(x_t)$ by
\begin{align} 
g_t = \gamma \nabla_x L_\gamma (x_t, y_{t+1}^\gamma, \lambda_{t+1}^\gamma) - \gamma \nabla_x L_g(x_t, y_{t+1}^g, \lambda_{t+1}^g). \label{eq: gt CC}
\end{align}
The update for $x$ is then given by $x_{t+1} = \Proj_{\mathcal{X}}(x_t - \eta g_t)$, where the step size satisfies $\eta \leq l_{F_\gamma,1}^{-1}$. The full algorithm is summarized in Algorithm \ref{alg: PBGD-BLOCC}.

\begin{algorithm}[t]
    \caption{PBGD-BLOCC \citep{jiang2024primal}}
    \label{alg: PBGD-BLOCC}
    \begin{algorithmic}[1]
        \State\textbf{Inputs:} initial point $x_0$; stepsize $\eta$; counters $T$; inner \textsf{MaxMin Solver}.
        \For{$t = 0, 1, \ldots, T-1$}
            \State Update $(\lambda_{t+1}^g, y_{t+1}^g)$ as in \eqref{eq: lambda y g} via \textsf{MaxMin Solver}.
            \State Update $(\lambda_{t+1}^\gamma, y_{t+1}^\gamma)$ as in \eqref{eq: lambda y gam} via \textsf{MaxMin Solver}.
            \State Update $x_{t+1} = \Proj_{\mathcal{X}}\!\big(x_t - \eta g_t\big)$ where $g_t$ is defined in \eqref{eq: gt CC}.
        \EndFor
        \State\textbf{Outputs:} $(x_T, y_{T}^\gamma)$
    \end{algorithmic}
\end{algorithm}

In earlier work \citep{jiang2024primal}, the smoothness modulus of $F_\gamma(x)$ was estimated as $l_{F_\gamma,1} = \mathcal{O}(\gamma)$. However, this estimate is not tight, similar to the case of non-coupled constraints. To improve upon this, we extend Lemma~\ref{lemma: implicit gradient} to the coupled-constrained setting, $\mathcal{Y}(x) = \{ y \in \mathcal{Y} : c(x,y) \leq 0 \}$, which is formalized in Lemma~\ref{lemma: Hessian of v CC} under the following mild assumptions.

\begin{assumption}
\label{assumption: LL boundary assumption}
Suppose for any $x\in \mathcal{X}$, $\|\lambda_g^*(x) \|<B_\lambda$, and strict complementarity holds at the solution $y_g^*(x)$. i.e., for $\lambda_g^*(x)$ being the Lagrangian multiplier,
$[\lambda_g^*(x)]_i>0$ for all $i\in \mathcal{I}^0$, where $\mathcal{I}^0=\{i\in \mathbb{R}^{d_c}:c_i(x,y^*(x)) = 0\}$ is the active constraints index set.
Suppose the domain $\mathcal{Y}$ is smooth on the boundary, and the constraint function $c$ is twice differentiable with $\nabla^2 c$ being $l_{c,1}$-Lipschitz in $y$. 
\end{assumption}

The assumption of the existence of the upper bound for the Lagrange multiplier is a consequence of the LICQ condition \cite[Theorem 1]{wachsmuth2013licq} and is conventional \citep{jiang2024primal}.
The strict complementarity condition is traditionally assumed in BLO literature \citep{kwon2023penalty}.
The $\mathcal{Y}$ boundary smoothness and $\nabla^2 c$ Lipschitzness are often satisfied in practice, for example, in transportation and network flow problems with linear constraints \citep{santos2021bilevel}, energy systems with smooth constraints \citep{ran2024bi}, and portfolio optimization problems with smooth budget or risk limits \citep{fernandez2015bilevel,wei2022bi}.  
Under Assumption~\ref{assumption: LL boundary assumption}, Lemma~\ref{lemma: Hessian of v CC} presents a generalization to Lemma~\ref{lemma: implicit gradient}.

\begin{lemma}[Generalization to Lemma~\ref{lemma: implicit gradient}]
\label{lemma: Hessian of v CC}
Consider $\mathcal{Y}(x)=\{y\in \mathcal{Y}:c(x,y)\leq 0 \}$. Suppose Assumption \ref{assumption: UL}, \ref{assumption: LL}, \ref{assumption: constraint}, \ref{assumption: Hessian Lipschitz}, and \ref{assumption: LL boundary assumption} hold. 
For a given direction $d \in \mathbb{R}^{d_x}$, the directional derivative of $y_g^*(x)$ along $d$ is given by the projection
\begin{align}
    & D_d(y^*_g(x))\label{eq: implicit directional derivative CC}\\
    = & \Proj_{\tilde{\mathcal{C}}_\mathcal{Y}(y^*_g(x))}^{\nabla_{yy} L_g(x,y^*_g(x),\lambda_g^*(x))} \Big( - (\nabla_{yy} L_g(x,y^*_g(x),\lambda_g^*(x)))^{-1} \nabla_{yx}L_g(x,y^*_g(x),\lambda_g^*(x)) d \Big).  \nonumber
\end{align}
where the directional critical cone is defined as
\begin{align}
    \tilde{\mathcal{C}}_{\mathcal{Y}}(y_g^*(x))=\{ w\in \mathcal{C}_{\mathcal{Y}}(y_g^*(x)): \langle \nabla_y c_i(x,y_g^*(x)),w\rangle  +\langle  \nabla_x c_i(x,y_g^*(x)),d\rangle=0, \forall i \in \mathcal{I}^0\}, \label{eq: directional critical cone}
\end{align}
Here, $\mathcal{C}_{\mathcal{Y}}(y_g^*(x))$ is the standard critical cone of $L_g(x,\cdot,\lambda_g^*(x))$, and the projection is taken with respect to the $\nabla_{yy} L(x,y^*_g(x),\lambda_g^*(x))$-norm given by $\|v\|_{\nabla_{yy} L(x,y^*_g(x),\lambda_g^*(x))}^2 := v^\top \nabla_{yy} L(x,y^*_g(x),\lambda_g^*(x)) v$.
\end{lemma}

Before proceeding to the proof of Lemma~\ref{lemma: Hessian of v CC}, we present several useful lemmas. We first extend Lemma~\ref{lemma: gradient equals zero generalization} to the setting of bilevel optimization with coupled constraints.

\begin{lemma}[Generalization to Lemma~\ref{lemma: gradient equals zero generalization}]
\label{lemma: gradient equals zero generalization CC}
Suppose $\mathcal{Y}$ is smooth on its boundary, $\|\lambda_g^*(x) \|<B_\lambda$ for all $x\in \mathcal{X}$, and Assumption \ref{assumption: LL}, \ref{assumption: constraint} hold. 
Then,
\begin{align*}
    \Big\langle \nabla_y g(x,y_g^*(x)) + 
 \langle \lambda_g^*(x),\nabla_x c(x,y_g^*(x))\rangle, \lim_{r\downarrow 0} \frac{y_g^*(x+rd)-y_g^*(x)}{r}\Big\rangle = 0
\end{align*}
\begin{remark}
    This implies the directional derivative $D_d(y_g^*(x))=\lim_{r\downarrow 0} \frac{y_g^*(x+rd)-y_g^*(x)}{r}$ lies in the critical cone $\mathcal{C}_{\mathcal{Y}}(y_g^*(x))$ of $g^\lambda(x,\cdot)$.
\end{remark}
\end{lemma}
The proof of this lemma is provided in Appendix~\ref{appendix: proof of Lemma Hessian of v CC}. We then look into the lemmas associated with constructing $\tilde{\mathcal{C}}_{\mathcal{Y}}(y_g^*(x))$.
\begin{lemma}
\label{lemma: ddcs}
Suppose $\mathcal{Y}$ is smooth on its boundary, $\|\lambda_g^*(x) \|<B_\lambda$ for all $x\in \mathcal{X}$, and Assumption \ref{assumption: LL}, \ref{assumption: constraint} hold. 
Denote the inactive index set $\mathcal{I}^- = \{ i \in [d_c] : c_i(x, y_g^*(x)) < 0 \}$ and the active index set $\mathcal{I}^0 =\{i\in [d_c]: c_i(x,y_g^*(x))=0\}$.
Let $d \in \mathbb{R}^{d_x}$ be a unit direction and denote $D_d (\lambda_g^*(x))=\lim_{r\downarrow 0}\frac{\lambda_g^*(x+rd)-\lambda_g^*(x)}{r}$, $D_d (y_g^*(x))=\lim_{r\downarrow 0}\frac{y_g^*(x+rd)-y_g^*(x)}{r}$. Then,
\begin{align*}
    \langle \nabla_y c_i(x,y_g^*(x)),D_d (y_g^*(x))\rangle  +\langle  \nabla_x c_i(x,y_g^*(x)),d\rangle = & 0, \quad \forall i \in \mathcal{I}^0,\quad \\
    \text{and} \quad [D_d(\lambda_g^*(x))]_i = & 0,\quad \forall i \in \mathcal{I}^- .
\end{align*}
\end{lemma}
\begin{proof}
By complementary slackness, which holds under LICQ and convexity, we have $ [\lambda_g^*(x)]_i c_i(x,y_g^*(x))=0$ for any $i\in [d_c]$.
In this way, we obtain
\begin{align*}
    0 = & \lim_{r\downarrow 0} \frac{ [\lambda_g^*(x+rd)]_i c_i(x+rd,y_g^*(x+rd))-[\lambda_g^*(x)]_i c_i(x,y_g^*(x))}{r}\\
    =& \lim_{r\downarrow 0} \{ D_d([\lambda_g^*(x)]_i)c_i(x,y_g^*(x)) +  [\lambda_g^*(x)]_i (\nabla_x  c_i(x,y_g^*(x))^\top d+ \nabla_y  c_i(x,y_g^*(x))^\top D_d(y_g^*(x)) \},
\end{align*}
where the second equality follows Taylor's expansion and the Lipschitzness of $y_g^*(x)$ and $\lambda_g^*(x)$, which holds according to~\citet[Lemma 2]{jiang2024primal}.

For any $i$ in the active index set $\mathcal{I}^0$, we have $ D_d([\lambda_g^*(x)]_i)c_i(x,y_g^*(x)) =0$ as $c_i(x,y_g^*(x)=0$. 
Strict complementary assumption implies that $[\lambda_g^*(x)]_i >0$. 
In this way, cancellation gives
 \begin{align*}
     \nabla_x  c_i(x,y_g^*(x))^\top d+ \nabla_y  c_i(x,y_g^*(x))^\top D_d(y_g^*(x)) =0.
\end{align*}

For any $i$ in the inactive index set $\mathcal{I}^-$, by complementary slackness, we have $[\lambda_g^*(x)]_i = 0$ for all $i \in \mathcal{I}^-$.  
By continuity of $c_i(x, y_g^*(x))$ with respect to $x$, for any direction $d$ there exists $r_0>0$ such that $c_i(x + r d, y_g^*(x + r d)) < 0$ for all $0 < r \leq r_0$. i.e. $[\lambda_g^*(x+rd)]_i=0$ for all $0 < r \leq r_0$.
Hence, we obtain
\begin{align*}
    [D_d(\lambda_g^*(x))]_i = \lim_{r\downarrow 0} \frac{[\lambda_g^*(x+rd)]_i-[\lambda_g^*(x)]_i}{r}= 0, \quad \forall i \in \mathcal{I}^-.
\end{align*} 
\end{proof}

\begin{lemma}
\label{lemma: convexity of tilde C}
    Under all assumptions in Lemma~\ref{lemma: ddcs}, 
    for arbitrary $d\in \mathbb{R}^n$, $\tilde{\mathcal{C}}_{\mathcal{Y}}(y_g^*(x))$ defined in \eqref{eq: directional critical cone} is nonempty with $D_d(y_g^*(x))\in \tilde{\mathcal{C}}_{\mathcal{Y}}(y_g^*(x))$, convex, and closed.
\end{lemma}
\begin{proof}\allowdisplaybreaks
    By Lemma~\ref{lemma: gradient equals zero generalization CC} and Lemma~\ref{lemma: ddcs}, we know $D_d(y_g^*(x))\in \tilde{\mathcal{C}}_{\mathcal{Y}}(y_g^*(x))$. This proves the non-emptyness.
    Let $w_1, w_2 \in \tilde{\mathcal{C}}_{\mathcal{Y}}(x;d)$ and $\alpha \in [0,1]$. Since $\mathcal{C}_{\mathcal{Y}}(y_g^*(x))$ is convex, $w_\alpha := \alpha w_1 + (1-\alpha) w_2 \in \mathcal{C}_{\mathcal{Y}}(y_g^*(x))$.  
Moreover, for all $i \in \mathcal{I}^0$,
\begin{align*}
  0= &  \langle \nabla_y c_i(x,y_g^*(x)), w_\alpha \rangle + \langle \nabla_x c_i(x,y_g^*(x)), d \rangle \nonumber \\
= & \alpha \big(\langle \nabla_y c_i(x,y_g^*(x)), w_1 \rangle + \langle \nabla_x c_i(x,y_g^*(x)), d \rangle \big) + (1-\alpha) \big(\langle \nabla_y c_i(x,y_g^*(x)), w_2 \rangle \nonumber \\
& + \langle \nabla_x c_i(x,y_g^*(x)), d \rangle \big)  \nonumber \\
= & \alpha 0 +(1-\alpha)0,
\end{align*}
so $w_\alpha \in \tilde{\mathcal{C}}_{\mathcal{Y}}(y_g^*(x))$, proving convexity.

Furthermore, $\mathcal{C}_{\mathcal{Y}}(y_g^*(x))$ is closed by definition and the linear equalities
\[
\langle \nabla_y c_i(x,y_g^*(x)), w \rangle + \langle \nabla_x c_i(x,y_g^*(x)), d \rangle = 0, \quad \forall i \in \mathcal{I}^0,
\]
is a finite system of linear equations in $w$ whose the solution set is closed.
Hence, $\tilde{\mathcal{C}}_{\mathcal{Y}}(y_g^*(x))$ is closed as it is an interaction of two close sets.
\end{proof}

With Lemma~\ref{lemma: gradient equals zero generalization CC},~\ref{lemma: ddcs},~\ref{lemma: convexity of tilde C} prepared, we can now proceed to prove Lemma~\ref{lemma: Hessian of v CC}.

\begin{proof}[Proof of Lemma~\ref{lemma: Hessian of v CC}] 
As $g^\lambda(x,y)$ in \eqref{eq: g lambda x y} is differentiable with respect to $y$.
According to the first-order variational inequality in Lemma~\ref{lem:first_order_VI}, there is
\begin{align*}
    \langle \nabla_y g^\lambda(x,y), y- y^*_g(x)\rangle \geq 0, \quad \forall y \in \mathcal{Y}.
\end{align*}

For arbitrary $w\in \tilde{\mathcal{C}}_{\mathcal{Y}}(y_g^*(x))$, choose $y =y_g^*(x)+r w'$, where $w'\rightarrow w$ as $r\downarrow 0$. There exists sufficiently small $r$ such that $y =g^*(x)+r w'\in \tilde{\mathcal{Y}}$ following the definition of tangent cone and the fact that $\tilde{\mathcal{C}}_{\mathcal{Y}}(y_g^*(x))\subseteq \mathcal{C}_{\mathcal{Y}}(y_g^*(x))\subseteq \mathcal{T}_{\mathcal{T}}(y_g^*(x))$. In this way, for sufficiently small $r$, there is 
\begin{align*}
    \langle \nabla_y g^\lambda(x+rd,y^*_g(x)+rD_d(y_g^*(x))), (y_g^*(x)+rw)- (y^*_g(x)+rD_d(y_g^*(x)))\rangle & \geq 0.
\end{align*}
Here, by Taylor's expansion,
\begin{align*}
    & \nabla_y g^\lambda(x+rd,y^*_g(x)+rD_d(y_g^*(x)))\\
    = & \nabla_y g(x+rd,y^*_g(x)+rD_d(y_g^*(x))) + \big \langle \lambda_g^*(x+rd), \nabla_y c(x+rd,y^*_g(x)+rD_d(y_g^*(x))) \big \rangle \\
    = & \nabla_y g(x+rd,y^*_g(x)+rD_d(y_g^*(x))) + \big\langle \lambda_g^*(x), \nabla_y c(x+rd,y^*_g(x)+rD_d(y_g^*(x))) \big \rangle \\
    & +  \big\langle \lambda_g^*(x+rdl)-\lambda_g^*(x), \nabla_y c(x+rd,y^*_g(x)+rD_d(y_g^*(x))) \big \rangle \\
    = & \nabla_y g(x+rd,y^*_g(x)+rD_d(y_g^*(x))) + \big \langle \lambda_g^*(x), \nabla_y c(x+rd,y^*_g(x)+rD_d(y_g^*(x))) \big \rangle \\
    & +  \big\langle \lambda_g^*(x+rd)-\lambda_g^*(x), \nabla_y c(x+rd,y^*_g(x)+rD_d(y_g^*(x))) \big \rangle\\
    = & \nabla_y L_g \left(x,y_g^*(x),\lambda_g^*(x) \right) \nonumber \\
    &+ \nabla_{yx} L_g \left(x,y_g^*(x),\lambda_g^*(x) \right) rd+ \nabla_{yy} L_g \left(x,y_g^*(x),\lambda_g^*(x) \right) D_d(y_g^*(x)) + \Oc(r^2)\\
    & +  \big\langle \lambda_g^*(x+rd)-\lambda_g^*(x), \nabla_y c(x+rd,y^*_g(x)+rD_d(y_g^*(x))) \big \rangle.
\end{align*}
Therefore, we have
\begin{align*}
    0\leq & \lim_{r\downarrow 0} \frac{1}{r} \big \langle \nabla_y g^\lambda(x+rd,y^*_g(x)+rD_d(y_g^*(x))),w- D_d(y_g^*(x)) \big\rangle \nonumber \\
    = & \lim_{r\downarrow 0} \frac{1}{r} \big \langle \nabla_y L_g \left(x,y_g^*(x),\lambda_g^*(x) \right),w- D_d(y_g^*(x)) \big\rangle \nonumber \\
    &+\big \langle \nabla_{yx} L_g \left(x,y_g^*(x),\lambda_g^*(x) \right) d+ \nabla_{yy} L_g \left(x,y_g^*(x),\lambda_g^*(x) \right) D_d(y_g^*(x)) ,w- D_d(y_g^*(x)) \big\rangle \nonumber \\
    &+  \big \langle D_d(\lambda_g^*(x))^\top \nabla_y c(x,y^*_g(x)),w- D_d(y_g^*(x)) \big\rangle  \nonumber \\
    = & \big \langle \nabla_{yx} L_g \left(x,y_g^*(x),\lambda_g^*(x) \right) d\nonumber \\
    &+ \nabla_{yy} L_g \left(x,y_g^*(x),\lambda_g^*(x) \right) D_d(y_g^*(x)) ,w- D_d(y_g^*(x)) \big\rangle   \nonumber \\
    &+  \sum_{i\in \mathcal{I}^0} [\langle D_d(\lambda_g^*(x))]_i \langle \nabla_y c_i(x,y^*_g(x)),w- D_d(y_g^*(x)) \rangle. \nonumber\\
    = & \big \langle \nabla_{yx} L_g \left(x,y_g^*(x),\lambda_g^*(x) \right) d+ \nabla_{yy} L_g \left(x,y_g^*(x),\lambda_g^*(x) \right) D_d(y_g^*(x)) ,w- D_d(y_g^*(x)) \big\rangle.
\end{align*}
Here, the second inequality follows Lemma~\ref{lemma: gradient equals zero generalization CC},~\ref{lemma: ddcs}. and the third uses Lemma~\ref{lemma: convexity of tilde C}.
In this way, directional derivative $D_d(y^*_g(x))$ is characterized by the variational inequality:
\begin{align*}
    \langle \nabla_{yy}  L_g \left(x,y_g^*(x),\lambda_g^*(x) \right) + \nabla_{yx}  L_g \left(x,y_g^*(x),\lambda_g^*(x) \right) d, ~ w- D_d(y^*_g(x)) \rangle \geq 0 \quad\forall w\in \tilde{\mathcal{C}}_{\mathcal{Y}}(y^*_g(x)).
\end{align*}
As $D_d(y^*_g(x))\in  \tilde{\mathcal{C}}_{\mathcal{Y}}(y^*_g(x))$ according to Lemma~\ref{lemma: convexity of tilde C}, the variational inequality is equivalent to
\begin{align*}
    D_d(y_g^*(x)) = \arg\min_{\eta \in \tilde{\mathcal{C}}_{\mathcal{Y}}(y^*_g(x))} \frac{1}{2} \eta^\top \nabla_{yy} L_g \left(x,y_g^*(x),\lambda_g^*(x) \right) \eta + (\nabla_{yx} L_g \left(x,y_g^*(x),\lambda_g^*(x) \right) d)^\top \eta. 
\end{align*}
This gives a unique solution as the objective is strongly convex in $\eta$ and the constraint $\tilde{\mathcal{C}}_{\mathcal{Y}}(y^*_g(x))$ is closed and convex (cf. Lemma~\ref{lemma: convexity of tilde C}).
\end{proof}

Lemma~\ref{lemma: Hessian of v CC} facilitates the inference of $F_\gamma(x)$ being $\mathcal{O}(1)$-smooth in the coupled-constrained setting. In fact, \textsf{both $y_g^*(x)$ and $\lambda_g^*(x)$ are differentiable} under the assumptions in Lemma~\ref{lemma: Hessian of v CC} \citep[Theorem 2.1]{fiacco1976sensitivity}. However, their gradients depend on the specific local characterization $\varphi(y) \le 0$ of $\mathcal{Y}$, which is typically unavailable in practice. Therefore, we provide only the directional derivative form in Lemma~\ref{lemma: Hessian of v CC}. Similar to the uncoupled case, this \textsf{directional derivative is required solely for theoretical analysis of smoothness and does not need to be computed in algorithmic implementation}. We provide the summary of differentiability in Table~\ref{tab:differentiability-summary}.

\begin{theorem}[Generalization to Theorem~\ref{theorem: smoothness}]
\label{theorem: smoothness in CC setting}
    Consider $\mathcal{Y}(x)=\{y\in\mathcal{Y}:c(x,y)\leq 0 \}$. Suppose Assumption \ref{assumption: UL}, \ref{assumption: LL}, \ref{assumption: constraint}, \ref{assumption: Hessian Lipschitz}, \ref{assumption: LL boundary assumption} hold. Then, there exists finite $\gamma^*>0$ such that 
    $F_\gamma(x)$ is $l_{F_\gamma,1}=\Oc(1)$-smooth for all $\gamma>\gamma^*$.
\end{theorem}

\begin{proof}
According to \citep{bonnans2013perturbation}, when $v(x)$ is differentiable, the second order directional derivative of $v(x)$ in direction unit direction $d$ can be written as
\begin{align}
    D^2_{dd}(v(x))
    = & \lim_{r,r'\downarrow 0} \frac{1}{\frac{1}{2}r r'}\Big( g(x+rd+\frac{1}{2}r^2 d,y_g^*(x+rd)) -g(x,y_g^*(x)) -\langle \nabla v(x),rd \rangle \Big) \nonumber\\
    \stackrel{(a)}{=}  & \lim_{r,r'\downarrow 0} \frac{1}{\frac{1}{2}r r'}\bigg( 
    \langle \nabla_x g(x,y_g^*(x)), rd \rangle \nonumber + \langle \nabla_y g(x,y_g^*(x)), y_g^*(x+rd)-y_g^*(x)\rangle \nonumber \\
    &  + \frac{1}{2}( y_g^*(x+rd)-y_g^*(x))^\top \nabla_{yy} g(x,y_g^*(x))( y_g^*(x+rd)-y_g^*(x)) \nonumber \\
    &+ \frac{rr'}{2}d^\top \nabla_{xx} g(x,y_g^*(x)) d  +  \frac{r'}{2}d^\top \nabla_{xy} g(x,y_g^*(x))( y_g^*(x+rd)-y_g^*(x)) \nonumber \\
    &+  \frac{r}{2}( y_g^*(x+rd)-y_g^*(x))^\top \nabla_{yx}g(x,y_g^*(x)) d\nonumber\\
    & -\langle \nabla_x g(x,y_g^*(x)) + {\lambda_g^*(x)}^\top \nabla_x c(x,y_g^*(x)),rd \rangle 
    \bigg) \nonumber\\
    \stackrel{(b)}{=} & \Big(\lim_{r'\downarrow 0} \frac{y_g^*(x+r'd)-y_g^*(x)}{r'} \Big)^\top \nabla_{yy} g(x,y_g^*(x))  \Big(\lim_{r\downarrow 0} \frac{y_g^*(x+rd)-y_g^*(x)}{r} \Big) \nonumber \\
    & + d^\top \nabla_{xx} g(x,y_g^*(x)) d + \Big(\lim_{r'\downarrow 0} \frac{y_g^*(x+r'd)-y_g^*(x)}{r} \Big)^\top \nabla_{xy} g(x,y_g^*(x))d \nonumber \\
    & + d^\top \nabla_{xy} g(x,y_g^*(x))\Big(\lim_{r\downarrow 0} \frac{y_g^*(x+rd)-y_g^*(x)}{r} \Big). \nonumber\\
    = & D_d(y_g^*(x))^\top \nabla_{yy} g(x,y_g^*(x)) D_d(y_g^*(x)) + d^\top \nabla_{xx} g(x,y_g^*(x)) d \nonumber\\
    & + D_d(y_g^*(x))^\top \nabla_{xy} g(x,y_g^*(x))d + d^\top \nabla_{xy} g(x,y_g^*(x)) D_d(y_g^*(x)). \label{eq: directional Hessian intermediate step}
\end{align}
where the (a) is by 2nd order of Taylor expansion; (b) is from $\langle \nabla_x g(x,y_g^*(x)), rd \rangle + \langle \nabla_y g(x,y_g^*(x)), y_g^*(x+rd)-y_g^*(x)\rangle =\langle \nabla_x g(x,y_g^*(x)) + {\lambda_g^*(x)}^\top \nabla_x c(x,y_g^*(x)),rd \rangle$ as they are two equivalent expression of the directional $\nabla v(x)$. Here, we have prepared $D_d(y^*_g(x))$ in Lemma~\ref{lemma: Hessian of v CC}.
Similarly, denote $\tilde{g}_\gamma(x,y) =\gamma^{-1}f(x,y)+g(x,y)$, we have
\begin{align*}
    D^2_{dd}(v_\gamma(x) )= & D_d(y_\gamma^*(x))^\top \nabla_{yy} \tilde{g}_\gamma(x,y_\gamma^*(x)) D_d(y_\gamma^*(x)) + d^\top \nabla_{xy} \tilde{g}_\gamma(x,y_\gamma^*(x)) D_d(y_\gamma^*(x)) \nonumber \\
    &+ d^\top \nabla_{xx} \tilde{g}_\gamma(x,y_\gamma^*(x)) d + D_d(y_\gamma^*(x))^\top \nabla_{xy} \tilde{g}_\gamma(x,y_\gamma^*(x))d 
\end{align*}
where $D_d(y^*_\gamma(x))=$
\begin{align*}
    \Proj_{\mathcal{C}_{\mathcal{Y}}(y^*_\gamma(x))}^{\nabla_{yy} L_\gamma \left(x,y_\gamma^*(x),\lambda_\gamma^*(x) \right)} \Big( - (\nabla_{yy} L_\gamma \left(x,y_\gamma^*(x),\lambda_\gamma^*(x) \right))^{-1} \nabla_{yx}L_\gamma \left(x,y_\gamma^*(x),\lambda_\gamma^*(x) \right) d \Big).
\end{align*}
Following directly from the same arguments as in Lemma~\ref{lemma: implicit directional derivative lipschitz}, we have
\begin{align}
    \| D_d(y^*_\gamma(x))- D_d(y^*_g(x))\|=\Oc(\gamma^{-1}), \label{eq: lipschitzness of directional derive}
\end{align}
noting that all steps generalize to the coupled constraint case. Similarly, as $\tilde{\mathcal{C}}_{\mathcal{Y}}(y^*_g(x))$ is closed and convex (cf. Lemma~\ref{lemma: convexity of tilde C}), $L_g(x,y,\lambda)$ is $\mu_g$-strongly convex in $y$ and $l_{g,1}+B l_{c,1}$-smooth, we obtain similar arguments as in Lemma~\ref{lemma: directional derivative bound}:
\begin{align}
    \| D_d(y^*_g(x))\|\leq \frac{l_{g,1}+B l_{c,1}}{\mu_g}. \label{eq: boundness of Ddy}
\end{align}

In this way, following the triangle inequality and Cauchy-Schwarz inequality, there is
\begin{align*}
    \|  D_{dd}^2( F_\gamma(x)) \|= & \gamma \|   D_{dd}^2( v_\gamma(x))- D_{dd}^2( v(x)) \| \nonumber \\
    \leq &\gamma \Big(\underbrace{\|D_d(y_g^*(x))\|^2}_{\eqref{eq: boundness of Ddy}} \underbrace{\big\|\nabla_{yy} g(x,y_g^*(x)) - \nabla_{yy} \tilde{g}_\gamma(x,y_g^*(x)) \big\|}_{\eqref{eq: nabla xx diff}} \nonumber\\
    &+ 2\underbrace{\|D_d(y_g^*(x))\|}_{\eqref{eq: boundness of Ddy}}\underbrace{\|\nabla_{yy} g(x,y_g^*(x))\| }_{l_{g,1}-\text{smoothness of }g}\underbrace{\big\| D_d(y_g^*(x)) - D_d(y_\gamma^*(x))\big\|}_{\eqref{eq: lipschitzness of directional derive}}  \nonumber\\
    & + \underbrace{\big \|\nabla_{xx} g(x,y_g^*(x)) - \nabla_{xx} \tilde{g}_\gamma(x,y_g^*(x))  \big \| }_{\eqref{eq: nabla xx diff}}  \nonumber\\
    &+ \underbrace{\|D_d(y_g^*(x))\|}_{\eqref{eq: boundness of Ddy}} \underbrace{\big\| \nabla_{xy} g(x,y_g^*(x))- \nabla_{xy} \tilde{g}_\gamma(x,y_g^*(x))  \big\|}_{\eqref{eq: nabla xx diff}}   \nonumber\\
    &+ \underbrace{\| \nabla_{xy} g(x,y_g^*(x))\| }_{l_{g,1}-\text{smoothness of }g}\underbrace{\big\| D_d(y_g^*(x)) - D_d(y_\gamma^*(x))\big\| }_{\eqref{eq: lipschitzness of directional derive}}
    \Big) \nonumber\\
    \leq & \gamma \Oc(\gamma^{-1}) = \Oc(1).
\end{align*}

\end{proof}

Theorem~\ref{theorem: smoothness in CC setting} allows for $\eta =\Oc(1)$ stepsize choice of running PBGD-BLOCC~\citep{jiang2024primal} in the coupled constraint setting and results in a reduced complexity. As corroborated in Fig.~\ref{fig:BLOCC-comparison}, increasing $\gamma$ does not require decrease in $\eta$.
\begin{proposition}
\label{prop: PBGD-BLOCC}
    Suppose all assumptions in Theorem~\ref{theorem: smoothness in CC setting} hold.
    For $\gamma \geq \frac{l_{F_\gamma,1}}{\mu_g}$, Algorithm \ref{alg: PBGD-BLOCC} with $\eta =\Oc(1) \leq l_{F_\gamma,1}^{-1}$ is achieved for $T={\Oc}(\epsilon^{-1})$ outer-loop complexity to obtain $\frac{1}{T} \sum_{t=0}^{T-1} \|G_{F_\gamma,\mathcal{X}}(x_t) \|^2 <\epsilon$.
\end{proposition}
\begin{remark}
    The \textsf{MaxMin Solver} can be any efficient algorithm, such as the accelerated version of the projected gradient descent ascent algorithm \citep[Algorithm 2]{jiang2024primal}, therefore achieving $\Oc(\epsilon^{-2})$ overall complexity. For $\mathcal{Y}=\mathbb{R}^{d_y}$ and $c(x,y)=A(x)y+B(x)$ linear in $y$, the \textsf{MaxMin Solver} can be the fully single-loop version of the projected gradient descent ascent algorithm \citep[Algorithm 2]{jiang2024primal}, achieving $\tilde{\Oc}(\epsilon^{-1})$ complexity.
\end{remark}
\begin{proof}

     Assume the inner \textsf{MaxMin Solver} finds the $\Oc(\gamma^{-1}\epsilon)$ optimal value in metrics of optimal point distance. Then, the bias term $b(x_t)= \nabla F_\gamma (x_t) - g_t $ is bounded as 
    \begin{align*}
        \|b(x_t)\|
        = & \| \gamma \big( \nabla_x L_\gamma (x_t, y_{t+1}^\gamma, \lambda_{t+1}^\gamma) - \nabla_x L_g(x_t, y_{t+1}^g, \lambda_{t+1}^g)
        \big) \nonumber\\
        & -\gamma \big( \nabla_x L_\gamma (x_t, y_\gamma^*(x_t), \lambda_\gamma^*(x_t)) - \nabla_x L_g(x_t, y_g^*(x_t), \lambda_g^*(x_t))\|
        \big)\\
        \leq & \gamma (\| \nabla_{xy} L_\gamma (x_t, y_\gamma^*(x_t),\lambda_\gamma^*(x_t))\| \| y_\gamma^*(x_t)- y_{t+1}^\gamma\| + \| \nabla_{x} c(x_t, y_\gamma^*(x_t))\| \| \lambda_\gamma^*(x_t)- \lambda_{t+1}^\gamma\| )\\
        &+ \gamma (\| \nabla_{xy} L_g (x_t, y_g^*(x_t),\lambda_g^*(x_t))\| \| y_g^*(x_t)- y_{t+1}^g\| + \| \nabla_{x} c(x_t, y_g^*(x_t))\| \| \lambda_g^*(x_t)- \lambda_{t+1}^g\| )\\
        \leq &\gamma \big( (\gamma^{-1} l_{F_\gamma,1}+ l_{g,1}+B l_{c,1})\| y_\gamma^*(x_t)- y_{t+1}^\gamma\| + l_{c,0}\| \lambda_\gamma^*(x_t)- \lambda_{t+1}^\gamma\| \\
        &+ (l_{g,1}+B l_{c,1})\| y_g^*(x_t)- y_{t+1}^g\| + l_{c,0}\| \lambda_g^*(x_t)- \lambda_{t+1}^g\| 
        \big)
        = \Oc(\epsilon).
    \end{align*}
    where the first inequality uses triangle inequality, the second relies on the smoothness of $f$, $g$, and $c$, and the local Lipschitzness of $c(\cdot,y)$, and upper bounds for $\|\lambda_g^*(x)\|$, and the Cauchy-Schwartz inequality.
    
    According to Theorem~\ref{theorem: smoothness in CC setting}, $F_{\gamma}(x)$ is $l_{F_\gamma,1}=\Oc(1)$-smooth in $\mathcal{X}$. In this way, choose $\eta \leq \frac{1}{l_{F_\gamma,1}}=\Oc(1)$, we obtain
    \begin{align*}
        F_\gamma(x_{t+1}) \leq & F_\gamma(x_t) - \frac{1}{4 \eta }\| x_{t+1}-x_t  \|^2  + \eta \|  b(x_t)\|^2, 
    \end{align*}
    following similar analysis as in \eqref{eq: descent intermediate step 1}. Telescoping therefore gives
    \begin{align*}
        \frac{1}{T} \sum_{t=0}^{T-1} \| G_{F_\gamma,\mathcal{X}}(x_t)\|^2 \leq & \frac{4}{\eta T} (F_\gamma(x_0)-F_\gamma(x_T) ) + \frac{4}{T} \sum_{t=0}^{T-1}\|  b(x_t)\|^2
        \\
        = & {\cal O}(\eta^{-1}T^{-1}) + O(\epsilon)\\
        = &  {\cal O}(T^{-1} + \epsilon)
    \end{align*}
    To achieve $\frac{1}{T} \sum_{t=0}^{T-1} \|G_{F_\gamma,\mathcal{X}}(x_t) \|^2 <\epsilon$, we requires complicity $T=\Oc(\epsilon^{-1})$.
\end{proof}
The proof of Proposition \ref{prop: PBGD-BLOCC} follows directly from \citep{jiang2024primal}.  Our new analysis enables the stepsize choice of 
$\eta =\Oc(1)$, which improves the convergence rate from $T = (\epsilon^{-1.5})$ in \citep{jiang2024primal} to $T = \Oc(\epsilon^{-1})$. 

\begin{figure}[t]
    \centering
    \begin{minipage}{0.48\linewidth}
        \centering
        \includegraphics[width=0.98\linewidth]{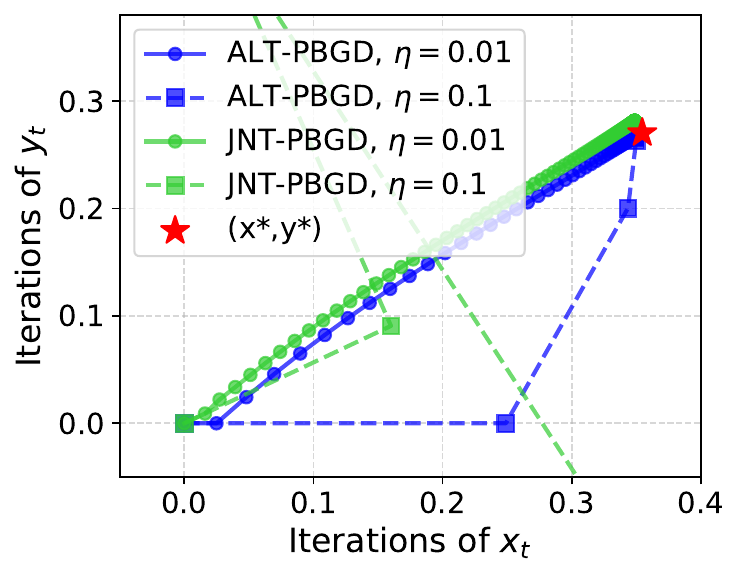}
        \caption{Iterations of solving Example \ref{example:toy_example_1} via ALT-PBGD and JNT-PBGD on $\min_{x, y} \tilde{F}_\gamma(x, y)$.}
    \label{fig:stepsize-comparison}
    \end{minipage}%
    \hfill
    \begin{minipage}{0.48\linewidth}
        \centering
        \includegraphics[width=0.98\linewidth]{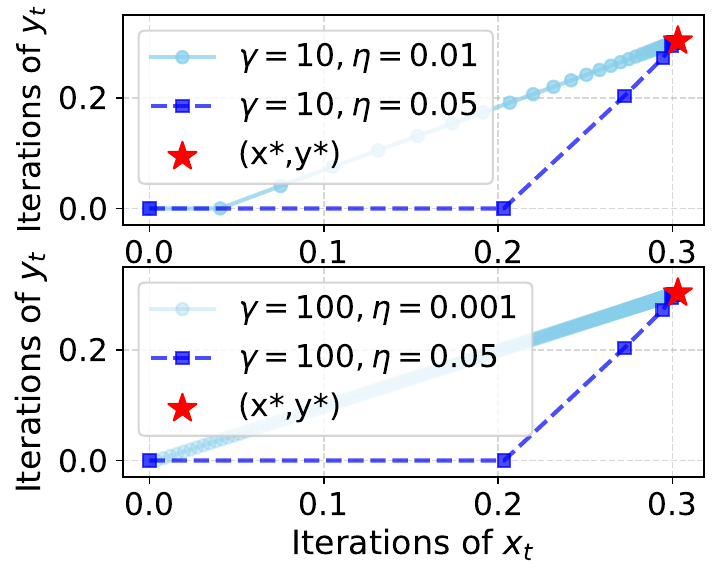}
        \caption{Trajectories of solving Example \ref{example:toy_example_2} via PBGD-BLOCC on $F_\gamma(x)$ with $\gamma = 10, 100$ and varying step-sizes.}
    \label{fig:BLOCC-comparison}
    \end{minipage}
\end{figure}

\subsection{PBGD-Free: efficient value-function-free algorithm for coupled constrained BLO}
\label{sec: BLOCC-VaFF}

By Lemma~\ref{lemma: equiv LL objective}, solving $y_g^*(x)=\arg\min_{y\in \mathcal{Y}(x)} g^\lambda(x,y)$ is equivalent to solving $y_g^*(x)=\arg\min_{y\in \mathcal{Y}} g^\lambda(x,y)$. Building on Lemma~\ref{lemma: distance of yg ygam}, we obtain the following corollary. 
\begin{corollary}
    Suppose the assumptions of Lemma~\ref{lemma: distance of yg ygam} hold. Then the conclusions of Lemma~\ref{lemma: distance of yg ygam} remain valid with $g(x,y)$ replaced by $g^\lambda(x,y)$. In particular, the $\epsilon$-suboptimal solutions of
    \begin{align}
        \min_{x\in \mathcal{X}}\Bigl\{F_\gamma^\lambda(x):= \min_{y\in \mathcal{Y}} f(x,y)
        +\gamma \bigl(g^\lambda(x,y)-\min_{y\in \mathcal{Y}}g^\lambda(x,y)\bigr)\Bigr\},
    \end{align} 
    are $\epsilon$-suboptimal local solutions with respect to the squared-distance metric for the $\epsilon$-approximation problem \eqref{eq: original problem 1}, as defined in \eqref{eq: epsilon app prob}, with $\gamma=\Oc(\epsilon^{-0.5})$ and $\gamma \geq \tfrac{l_{F_\gamma,1}}{\mu_g}$. 
    Moreover, suppose the assumptions in Lemma~\ref{lemma: tighter bound for y gam yg} hold. Then the conclusions of Lemma~\ref{lemma: distance of yg ygam} also remain valid with $g(x,y)$ replaced by $g^\lambda(x,y)$. i.e.
    \begin{align}
    \| \phi(x)-F^\lambda_\gamma(x) \| =& \Oc(\gamma^{-\frac{\alpha}{2-\alpha}}+\delta), ~~ \text{and}~~ 
    \|y_g^*(x)-y_\gamma^\lambda(x)\| = \Oc\big(\gamma^{-\frac{1}{2-\alpha}}+\delta^{\frac{1}{2}}\gamma^{-\frac{1}{2}}\big),
    \end{align}
    where $y_g^*(x)=\arg\min_{y\in \mathcal{Y}(x)} g^\lambda(x,y)$ coincides with the one in \eqref{eq: lower-level problem}, and 
    $y_\gamma^\lambda(x):= \arg\min_{y\in \mathcal{Y}(x)} \gamma^{-1}f(x,y)+g^\lambda(x,y)$.
\end{corollary}
In other words, the coupled-constraint BLO problem \eqref{eq: original problem 1} can be reformulated as an equivalent non-coupled problem with a strongly convex lower-level objective $g^\lambda(x,y)$. When $f(x,\cdot)$ is $(\delta,\alpha)$-flat, Lemma~\ref{lemma: tighter bound for y gam yg} shows $\nabla_x f(x,y_\gamma^\lambda(x))$ with $y_\gamma^\lambda (x)$ can be a good estimate for $\nabla F_\gamma^\lambda(x)$,  
with bias bounded with $\Oc(\gamma^{-\frac{\alpha-1}{2-\alpha}} + \delta^{\frac{1}{2}}\gamma^{\frac{1}{2}})$. However, finding $y_\gamma^\lambda(x)$ requires access to $\lambda_g^*(x)$, which involves another round of iteration to find the value function. In the following, we show that another surrogate can bypass this.

\begin{lemma}
\label{lemma: constrained minimizer bound}
Consider $\mathcal{Y}(x)=\{y\in\mathcal{Y}:c(x,y)\leq 0 \}$. Suppose Assumption \ref{assumption: UL}, \ref{assumption: LL}, \ref{assumption: constraint}, \ref{assumption: Hessian Lipschitz}, \ref{assumption: LL boundary assumption} hold.
For any fixed $x\in \mathcal{X}$, let us denote $y_\gamma^*(x)\in \arg\min_{y\in \mathcal{Y}(x)}\frac{1}{\gamma}f(x,y)+g(x,y)$. Then, its distance to $y_\gamma^\lambda(x)$ is bounded with $\| y_\gamma^*(x)-y_\gamma^\lambda(x)\|=\Oc(\frac{1}{\gamma})$.
\end{lemma}
\begin{proof}
Under Assumption \ref{assumption: UL}.\ref{ass: UL smooth} and \ref{assumption: LL}.\ref{ass: LL SC}, when choosing $\gamma\geq\frac{l_{F_\gamma,1}}{\mu_g}$, 
$g(x,y)-\min_{y\in \mathcal{Y}(x)}g(x,y)$ satisfies $g(x,y)-\min_{y\in \mathcal{Y}(x)}g(x,y) \geq \mu_g d^2_{S_g(x)}(y)$ and $g(x,y)-\min_{y\in \mathcal{Y}(x)}g(x,y)=0$ iff $d^2_{S_g(x)}(y)=0$, following \citep[Lemma 1]{jiang2024primal}. In this way, the conditions for Remark~\ref{remark: generalizaed dist} is satisfied. In this way, $\| y_\gamma^*(x)-y_g^*(x)\|=\Oc(\frac{1}{\gamma})$, where $y_g^*(x)$ is a singleton by the strongly convexity of $g$ and the convexity and closeness of the domain $\mathcal{Y}(x)$.
Similarly, $\| y_\gamma^\lambda(x)-y_g^*(x)\|=\Oc(\frac{1}{\gamma})$. The lemma is therefore proved by triangle inequality.
\end{proof}
In this way, we can show that Algorithm \ref{alg: PBGD-Free} works well even in the case of BLO problems with coupled constraints.

\begin{theorem}
\label{thm: no value function for BLO with CC}
    Consider $\mathcal{Y}(x)=\{y\in\mathcal{Y}:c(x,y)\leq 0 \}$. Suppose Assumption \ref{assumption: UL}, \ref{assumption: LL}, \ref{assumption: constraint}, \ref{assumption: Hessian Lipschitz}, \ref{assumption: LL boundary assumption} hold,
    $\lambda_g^*(x)$ is uniformly bounded, and for all $x\in \mathcal{X}$, $f(x,\cdot)$ is $(\delta(x),\alpha)$-flat at $y_g$ with the same $\alpha \in (1,1.5)$ and modulus $c=\Oc(1)$. For iterations using the {\color{blue}naive version of Algorithm \ref{alg: PBGD-Free}} with $\eta \leq  l_{F_\gamma,1}^{-1}$, suppose there exists $\delta$ such that $\frac{1}{T}\sum_{t=0}^{T-1}  \delta(x) \leq \delta$ along the trajectory, then by choosing $\gamma = \Oc(\delta^{-\frac{2-\alpha}{2}})$, it holds that
    \begin{align}
    \frac{1}{T} \sum_{t=0}^{T-1}\|G_{F_\gamma^\lambda,\mathcal{X}}(x_t)\|^2 \leq \Oc(T^{-1} + \delta^{\frac{2(\alpha-1)}{\alpha}}).
    \end{align}
\end{theorem}

Before proving Theorem~\ref{thm: no value function for BLO with CC}, we present a useful lemma.
\begin{lemma}
\label{lemma: lambda gradient equals zero}
Under all assumptions and notations in Lemma~\ref{lemma: gradient equals zero generalization CC}, and suppose $g(x,y)$, $c(x,y)$ is locally Lipschitz at $(x,y_g^*(x))$, there is
\begin{align}
    \Big\langle \lim_{r\downarrow 0}\frac{\lambda_g^*(x+rd)-\lambda_g^*(x)}{r},c(x,y_g^*(x))\Big\rangle = 0.
\end{align}
\end{lemma}
\begin{proof}
    By definition, the directional derivative of $v(x)$ in unit direction $d\in \mathbb{R}^{d_x}$ is 
\begin{align*}
    D_d(v) = & \lim_{r \downarrow 0} \frac{1}{r} \Big(g(x+rd,y_g^*(x+rd)) + \langle \lambda_g^*(x+rd),c(x+rd,y_g^*(x+rd)) \rangle \\
    & - \big( g(x,y_g^*(x)) + \langle \lambda_g^*(x),c(x,y_g^*(x)) \rangle \big)
    \Big)\\
    = & \lim_{r \downarrow 0} \frac{1}{r} \Big(g(x+rd,y_g^*(x+rd)) + \langle \lambda_g^*(x),c(x+rd,y_g^*(x+rd)) \rangle \\
    & + \langle \lambda_g^*(x+rd)-\lambda_g^*(x),c(x+rd,y_g^*(x+rd)) \rangle  - \big( g(x,y_g^*(x)) + \langle \lambda_g^*(x),c(x,y_g^*(x)) \rangle \big)
    \Big)\\
    \stackrel{(a)}{=} & \langle \nabla_x g(x,y_g^*(x))+\langle \lambda_g^*(x),\nabla_x c(x,y_g^*(x)) \rangle,d \rangle \\
    & + \langle \nabla_y g(x,y_g^*(x)) +\langle \lambda_g^*(x),\nabla_y c(x,y_g^*(x)) \rangle, \lim_{r\downarrow 0}\frac{y_{g}^*(x+rd)-y_g^*(x)}{r} \rangle \\
    & + \langle \lim_{r\downarrow 0}\frac{\lambda_g^*(x+rd)-\lambda_g^*(x)}{r},c(x,y_g^*(x))\rangle \\
    \stackrel{(b)}{=} & \langle \nabla_x g(x,y_g^*(x))+\langle \lambda_g^*(x),\nabla_x c(x,y_g^*(x)) \rangle,d \rangle + \langle \lim_{r\downarrow 0}\frac{\lambda_g^*(x+rd)-\lambda_g^*(x)}{r},c(x,y_g^*(x))\rangle ,
\end{align*}
where (a) follows Taylor expansion and the Lipschitzness of $y_g^*(x)$ (cf. Remark~\ref{remark: Lipschitz of S(x)}); (b) is by Lemma~\ref{lemma: gradient equals zero generalization CC}. Moreover, $v(x)$ is differentiable with its gradient in \eqref{eq: value function gradient generalized form}. We can thus conclude the proof.

\end{proof}

\begin{proof}[Proof of Theorem~\ref{thm: no value function for BLO with CC}]



Under the boundary smoothness condition, $\lambda_g^*(x)$ is differentiable~\citep[Theorem 2.1]{fiacco1976sensitivity}. However, we do not know its explicit form and therefore, the gradient of
\begin{align}
    v_\gamma^\lambda(x)=\min_{y\in \mathcal{Y}}\tilde{g}_\gamma^\lambda(x,y)= \min_{y\in \mathcal{Y}} \{ \gamma^{-1}f(x,y)+g(x,y) + \langle\lambda_g^*(x),c(x,y)\rangle \},
\end{align}
is not known. 
Denote $y_g$, $y_{g,r}$, $y_\lambda$, $y_{\lambda,r}$ respectively in short for $y_g^*(x)$, $y_g^*(x+rd)$, $y_\gamma^\lambda(x)$, $y_\gamma^\lambda(x+rd)$. 
By definition of directional derivative, there is
\begin{align}
    D_d(F_\gamma^\lambda(x))=& \lim_{r\downarrow 0}\frac{1}{r} \Big(\big(f(x+rd,y_{\lambda,r})+
    \gamma g(x+rd,y_{\lambda,r})+\gamma \langle \lambda_g^*(x+rd),c(x+rd,y_{\lambda,r}) \rangle
    \big) \nonumber \\
    & -\big(f(x,y_{\lambda})+
    \gamma g(x,y_{\lambda})+ \gamma \langle \lambda_g^*(x),c(x,y_{\lambda}) \rangle
    \big) \Big) - \gamma \langle \nabla v(x),d\rangle \nonumber\\
    = & 
    \lim_{r\downarrow 0} \frac{1}{r} \Big(\big(f(x+rd,y_{\lambda,r})+
    \gamma g(x+rd,y_{\lambda,r})+\gamma \langle \lambda_g^*(x), c(x+rd,y_{\lambda,r}) \rangle
    \big) \nonumber \\
    &~~~~~~~~~~~~ + \gamma \langle \lambda_g^*(x+rd)-\lambda_g^*(x),c(x+rd,y_{\lambda,r}) \rangle
    \big) \nonumber \\
    &~~~~~~~~~~~~ -\big(f(x,y_{\lambda})+
    \gamma g(x,y_{\lambda})+\gamma \langle \lambda_g^*(x),c(x,y_{\lambda}) \rangle
    \big) \Big)- \gamma \langle \nabla v(x),d\rangle \nonumber \\
    \stackrel{(a)}{=} & \langle \nabla_x f(x,y_{\lambda})+
    \gamma \nabla_x g(x,y_{\lambda})+\gamma \langle \lambda_g^*(x),\nabla_x c(x,y_{\lambda}) ,d \rangle \nonumber \\
    & + \langle \nabla_y f(x,y_{\lambda}) +  \gamma  \nabla_y g(x,y_{\lambda})+ \gamma \langle \lambda_g^*(x),\nabla_y c(x,y_{\lambda}) \rangle, \lim_{r\downarrow 0 } \frac{y_{\lambda,r}-y_\lambda}{r} \rangle  \nonumber \\
    & +\gamma \langle \lim_{r\downarrow 0} \frac{\lambda_g^*(x+rd)-\lambda_g^*(x)}{r}, c(x,y_{\lambda}) - c(x,y_{g}) \rangle \nonumber \\
    & -\langle 
    \gamma \nabla_x g(x,y_g^*(x))+\gamma \langle \lambda_g^*(x),\nabla_x g(x,y_g^*(x))  ,d\rangle \nonumber \\
    \stackrel{(b)}{=} & \langle \nabla_x f(x,y_{\lambda}),d\rangle +\gamma \langle D_t^C(\lambda_g^*(x)), c(x,y_{\lambda}) - c(x,y_{g}) \rangle \nonumber \\
    &+\gamma \langle \nabla_x g(x,y_{\lambda}) + \gamma \langle \lambda_g^*(x),\nabla_x c(x,y_{\lambda}) \rangle ,d \rangle  - \nabla_x g(x,y_g^*(x)) - \langle \lambda_g^*(x),\nabla_x g(x,y_g^*(x))  ,d\rangle \nonumber \\
    \stackrel{(c)}{=} & \langle \nabla_x f(x,y_{\lambda}) ,d \rangle + \Oc\big(\gamma \| y_\lambda - y_g\| \big).  \nonumber\\
    \stackrel{(d)}{=} & \langle \nabla_x f(x,y_{\lambda}) ,d \rangle + \Oc\big(\gamma^{-\frac{\alpha-1}{2-\alpha}}+\delta^{\frac{1}{2}} \gamma^{\frac{1}{2}} \big).  \label{eq: F lambda directional derivative}
\end{align}
Here, (a) follows from Taylor's expansion and the Lipschitz continuity of $y_g^*(x)$. (b) applies Lemma~\ref{lemma: gradient equals zero generalization CC} and \ref{lemma: lambda gradient equals zero}, along with the Lipschitz continuity of $g$, $\nabla g$, and $\nabla c$ in $y$ near $y_g$. The Lipschitz continuity of $g$ near $y_g$ is implied by its convexity, ensuring $\nabla_y g(x,y_g^*(x))$ is bounded. For (c), 
strict complementarity holds, which together with the strong convexity of $g$ in $y$ and the LICQ of $c$, and the Lipschitzness conditions, guarantees that the directional derivative $\lim_{r\downarrow 0} \frac{\lambda_g^*(x+rd)-\lambda_g^*(x)}{r}$ exists and is finite \citep[Theorem~2.1]{fiacco1976sensitivity}. The result in (d) then directly follows from Lemma~\ref{lemma: tighter bound for y gam yg}. 

The key takeaway in \eqref{eq: F lambda directional derivative} is that, for any directional $d$, $\langle \nabla_x f(x,y_{\lambda}) ,d \rangle$ is in $\Oc\big(\gamma^{-\frac{\alpha-1}{2-\alpha}}+\delta^{\frac{1}{2}} \gamma^{\frac{1}{2}} \big)$ distance to $D_d^C(F_\gamma^\lambda(x))$. 
Following the triangle inequality and the Lipschitz of $\nabla f(x,y)$ and $y^*_\lambda(x)$, with later guaranteed by Lemma~\ref{lemma: SC PL}, there is
\begin{align}
    \|D_d(F_\gamma^\lambda(x_1))-D_d(F_\gamma^\lambda(x_2))\| \leq&   \|\nabla_x f(x_1,y_\gamma^\lambda(x_1)) - \nabla_x f(x_2,y_\gamma^\lambda(x_2)\| + \Oc\big(\gamma^{-\frac{\alpha-1}{2-\alpha}}+\delta^{\frac{1}{2}} \gamma^{\frac{1}{2}} \big) \nonumber\\
    \leq & l_{F_\gamma,1}(1+L_y^\lambda)\|x_1-x_2\| + \Oc\big(\gamma^{-\frac{\alpha-1}{2-\alpha}}+\delta^{\frac{1}{2}} \gamma^{\frac{1}{2}} \big). \label{eq: non-smooth F Lipschitz}
\end{align}

Moreover, $F_\gamma^\lambda(x)$ is continuous as $\lambda_g^*(x)$ is Lipschitz continuous according to \citep[Theorem 2.16]{ito2008lagrange}. In this way, choosing $d_t = (x_{t+1}-x_t)/\|x_{t+1}-x_t\|$, there is
\begin{align*}
    F_\gamma^\lambda(x_{t+1})-F_\gamma^\lambda(x_t) \stackrel{(a)}{\leq} & \| x_{t+1}-x_t\|D_{d_t}(F_\gamma^\lambda(x_t)) \\
    & +  \| x_{t+1}-x_t\| \int_0^1 D_{d_t}(F_\gamma^\lambda(x_t+q(x_{t+1}-x_t)))-D_{d_t}(F_\gamma^\lambda(x_t)) dq\\
    \stackrel{(b)}{\leq} & \langle \nabla_x f(x_t,y_{\lambda}^*(x_t)) ,x_{t+1}-x_t \rangle + \Oc\big(\|x_{t+1}-x_t\|\gamma^{-\frac{2(\alpha-1)}{2-\alpha}}+\delta \gamma \big)\\
    & + \| x_{t+1}-x_t\| \int_0^1 l_{F_\gamma,1}(1+L_y^\lambda)q\|x_{t+1}-x_{t}\| + \Oc\big(\gamma^{-\frac{\alpha-1}{2-\alpha}}+\delta^{\frac{1}{2}} \gamma^{\frac{1}{2}} \big)dq\\
    \stackrel{(c)}{=} & \langle \nabla_x f(x_t, y_{t+1}^\gamma ,x_{t+1}-x_t \rangle +\frac{l_{F_\gamma,1}(1+L_y^\lambda)}{2}\| x_{t+1}-x_t\|^2 \\
    & + \Oc\big(\|x_{t+1}-x_t\|\gamma^{-\frac{\alpha-1}{2-\alpha}}+\delta^{\frac{1}{2}} \gamma^{\frac{1}{2}} \big)\\
    & +  \langle \nabla_x f(x_t, y_{\lambda}^*(x_t) ) -\nabla_x f(x_t, y_{t+1}^\gamma) ,x_{t+1}-x_t \rangle \\
    \stackrel{(d)}{\leq} & \langle \nabla_x f(x_t, y_{t+1}^\gamma ),x_{t+1}-x_t \rangle +\frac{l_{F_\gamma,1}(1+L_y^\lambda)}{2}\| x_{t+1}-x_t\|^2 \\
    & + \Oc\big(\|x_{t+1}-x_t\|\gamma^{-\frac{\alpha-1}{2-\alpha}}+\delta^{\frac{1}{2}} \gamma^{\frac{1}{2}} \big)\\
    \stackrel{(e)}{\leq}  & -\frac{1}{\eta} \| x_{t+1}-x_t\|^2 + \frac{1}{2\eta} \| x_{t+1}-x_t\|^2 + \frac{1}{4\eta}\| x_{t+1}-x_t\|^2 \\
    & + \Oc\big(\gamma^{-\frac{\alpha-1}{2-\alpha}}+\delta^{\frac{1}{2}} \gamma^{\frac{1}{2}} +\gamma^{-1}\big)\\
    = & -\frac{\eta}{4} \Big\| \frac{x_{t+1}-x_t}{\eta}\Big\|^2+ \Oc\big(\eta(\gamma^{-\frac{2(\alpha-1)}{2-\alpha}}+\delta \gamma ) \big).
\end{align*}
Here, (a) follows from first-order Taylor expansion with an integral remainder in finite-dimensional Euclidean spaces; (b) is by \eqref{eq: non-smooth F Lipschitz}; (c) is the results of calculating integral over $q$; (d) uses Cauchy-Schwartz inequality and the Lipschitz of $\nabla f$; (e) follows similar analysis as in \eqref{eq: VaFF fully single-loop intermediate step}, it follows $\eta \leq (l_{F_\gamma,1}(1+L_y^\lambda))^{-1}$, $\langle \nabla_x f(x,y_{t+1}^\gamma),x_{t+1}-x_t\rangle \leq -\frac{1}{\eta}\| x_{t+1}-x_t\|^2$ by \citep[Lemma 3.1]{bubeck2015convex} and Young's inequality.

In this way, by telescoping and rearranging, there is
\begin{align*}
    \frac{1}{T}\sum_{t=0}^{T-1}\Big\| \frac{x_{t+1}-x_t}{\eta}\Big\|^2 \leq \tfrac{4(F_\gamma^\lambda(x_0)-F_\gamma^\lambda(x_T))}{\eta T}+ \Oc\big(\gamma^{-\frac{2(\alpha-1)}{2-\alpha}}+\delta \gamma \big)=\Oc(T^{-1}+\gamma^{-\frac{2(\alpha-1)}{2-\alpha}}+\delta \gamma ).
\end{align*}

According to \citet{bonnans2013perturbation}, the first order stationarity condition gives
\begin{align*}
    0 \in \partial F_\gamma^\lambda(x)+N_{\mathcal{X}}(x)
\end{align*}
where $N_{\mathcal{X}}(x)$ is the normal cone to $\mathcal{X}$ at $x$, and $\partial F(x)=\{g\in \mathbb{R}^{d_x}: \langle g,d\rangle \leq D_d(F(x)),~ \forall d\}$ is the subdifferential. Here, the subdifferential $\partial F(x)$ is non-empty due to the differentiability of $\lambda_g^*(x)$ \citep{fiacco1976sensitivity}, even though its explicit form is not known. Additionally, as $F_\gamma^\lambda (x)$ is differentiable, we know that $G_{F_\gamma^\lambda,\mathcal{X}}(x_t)=d_{\partial F_\gamma^\lambda(x_t)+N_{\mathcal{X}}(x_t)}(0)$.

Denote $g_t= \nabla_x g(x_t,y_t^\gamma)$ and $\tilde{g_t}$ as some elements in $\partial F(x_t)$, for any unit direction $d$, 
\begin{align*}
    \langle g_t+\tilde{g_t}-g_t,d \rangle \leq & D_d(F(x))= \langle g_t,d \rangle + \Oc\big(\gamma^{-\frac{\alpha-1}{2-\alpha}}+\delta^{\frac{1}{2}} \gamma^{\frac{1}{2}} \big).
\end{align*}

In this way, choose $d=\frac{\tilde{g_t}-g_t}{\|\tilde{g_t}-g_t\|}$, we obtain
\begin{align*}
    \|\tilde{g_t}-g_t\|\leq & \Oc\big(\gamma^{-\frac{\alpha-1}{2-\alpha}}+\delta^{\frac{1}{2}} \gamma^{\frac{1}{2}} \big).
\end{align*}

Additionally, by Lemma~\ref{lemma: projection gradient update inquality}, we know that $\left(\frac{x_t-x_{t+1}}{\eta}-g_t\right)^\top (x'-x_{t+1})\leq 0$ for any $x'\in \mathcal{X}$.
Hence, 
\begin{align*}
    n_{t}:=\frac{x_t-x_{t+1}}{\eta}-g_t\in N_{\mathcal{X}}
\end{align*}

by the definition of normal cone. In this way, we know that 
\begin{align*}
    \frac{1}{T} \sum_{t=0}^{T-1}\|G_{F_\gamma^\lambda,\mathcal{X}}(x_t)\|^2=&\frac{1}{T} \sum_{t=0}^{T-1} d_{\partial F_\gamma^\lambda(x_t)+N_{\mathcal{X}}(x_t)}^2(0)  \nonumber\\
    \leq & \frac{1}{T} \sum_{t=0}^{T-1} \|\tilde{g}_t+n_t\|^2 \nonumber\\
    = & \frac{1}{T} \sum_{t=0}^{T-1} \|\tilde{g}_t+n_t -g_t+g_t\|^2\nonumber\\
    \leq & \frac{1}{T}\sum_{t=0}^{T-1}\Big\| \frac{x_{t+1}-x_t}{\eta}\Big\|^2 + \frac{1}{T} \sum_{t=0}^{T-1} \|g_t-\tilde{g}_t\|^2 \nonumber\\
    \leq & \Oc(T^{-1}+\gamma^{-\frac{2(\alpha-1)}{2-\alpha}}+\delta \gamma ).
\end{align*}
In this way, we can conclude the proof via choosing $\gamma = \Oc(\delta^{-\frac{2-\alpha}{\alpha}})$.
\end{proof}

Notice that only \textsf{Min Solver}, rather than \textsf{MaxMin Solver}, is required in PBGD-Free (Algorithm~\ref{alg: PBGD-Free}). In this way, even for BLO with CCs, its inner-loop converges linearly when PBG is applied under the LL strongly-convexity assumption. This improves upon the general $\Oc(\epsilon^{-1})$ inner complexity in \citet{jiang2024primal}. A detailed comparison of algorithms is provided in Table~\ref{tab:table1-comparison}.

\section{Numerical Experiments}
\label{sec: exp}

In this section, we first empirically validate our improved smoothness analysis for ALT-PBGD and PBGD-BLOCC, demonstrating the $\mathcal{O}(1)$-smoothness of $F_\gamma(x)$ using illustrative toy examples in Sec.~\ref{sec: toy example for O1 smoothness}. We then evaluate the effectiveness of our proposed PBGD-Free algorithm on real-world tasks. Specifically, we consider the LLM PEFT problem~\eqref{eq: bilevel repre LLM} as a representative case of the uncoupled constraint setting in Sec.~\ref{sec: representation PEFT}, and the SVM hyperparameter optimization problem in Sec.~\ref{sec: SVM exp} as a representative of the coupled constraint setting. 

\subsection{Toy examples verifying improved convergence analysis}
\label{sec: toy example for O1 smoothness}
\subsubsection{Comparison of ALT-PBGD and JNT-PBGD on fixed step sizes}
In the experiments, we set $\gamma = 10$ and run Algorithm \ref{alg: ALT-PBGD}, \ref{alg: JNT-PBGD} both with an initial point of $(x_0 = 0, y_0 = 0)$ to solve Example \ref{example:toy_example_1}. 
\begin{example}
\label{example:toy_example_1}
    Consider the BLO problem in \eqref{eq: original problem 1} with $\mathcal{X}=\mathcal{Y}(x)=[0,3]$ and the objectives:
    \begin{align*}
    f(x,y) = & \frac{e^{-y+1}}{2+\cos(4x)} +\frac{1}{2} \ln \Big( (4x-2)^2 +1 \Big) + x^2,\\ 
    g(x,y) = & 2(y-x)^2  +\frac{x}{2} \sin^2(x+y).
    \end{align*}
\end{example}
Both ALT-PBGD and JNT-PBGD are applied with a smaller step-size choice of $\eta = 0.01$, and a larger step-size, $\eta = 0.1$. As shown in Fig.~\ref{fig:stepsize-comparison}, both algorithms exhibit similar convergence rates for the smaller step-size of $\eta = 0.01$. However, when using $\eta = 0.1$, JNT-PBGD fails to converge, while ALT-PBGD performs well, achieving faster convergence within $5$ steps. Since the upper bound of the stepsize is inversely proportional to the smoothness constant, this observation aligns with the theoretical results in Lemma~\ref{lemma: implicit gradient} and Proposition \ref{prop: ALT-PBGD}, which indicate that ALT-PBGD can tolerate a larger $\gamma$ without requiring a smaller stepsize, as the smoothness constant is not scaled with $\gamma$. 

\subsubsection{Comparison of BLOCC with different \texorpdfstring{$\gamma$}{gamma}}

We illustrate Proposition \ref{prop: PBGD-BLOCC} using Example \ref{example:toy_example_2} by applying BLOCC from Algorithm \ref{alg: PBGD-BLOCC} \citep{jiang2024primal} to solve $ F_\gamma(x)$ with $\gamma = 10$ and $\gamma = 100$. The step-size choices include a safe option $\eta = \frac{1}{10\gamma}$ for both $\gamma$ choices compared to a larger value $\eta = 0.05$. 

\begin{example}
\label{example:toy_example_2}
    Consider the coupled constrained BLO problem in \eqref{eq: original problem 1} with $\mathcal{X}=[0,3]$, $\mathcal{Y}(x)=\{y\in [0,3]: y-x\leq 0 \}$ and the objectives:
    \begin{align*}
        f(x,y)= \frac{\exp(-y+2)}{(2+\cos(4x)} + \frac{1}{2}\ln((4x-2)^2+1) + x^2,~~~~~~~~
        g(x,y)=   (y-2x)^2.
    \end{align*}
\end{example}
As shown in Fig.~\ref{fig:BLOCC-comparison}, the PBGD-BLOCC converges for both $\gamma = 10$ and $\gamma = 100$ using either the large step-size $\eta = 0.05$ or the safe step-size $\eta = \frac{1}{10\gamma}$ for each $\gamma = 10$ and $\gamma = 100$. Notably, using the larger step-size $\eta = 0.05$ results in faster convergence to the optimal solution, requiring fewer iterations compared to the literature. This supports the conclusion that $F_\gamma(x)$ exhibits constant-level smoothness, leading to improved complexity as stated in Proposition \ref{prop: PBGD-BLOCC}.

\begin{figure}[t]
    \centering
    \small
    \begin{minipage}{0.44\linewidth}
        \centering
        \includegraphics[width=0.97\textwidth]{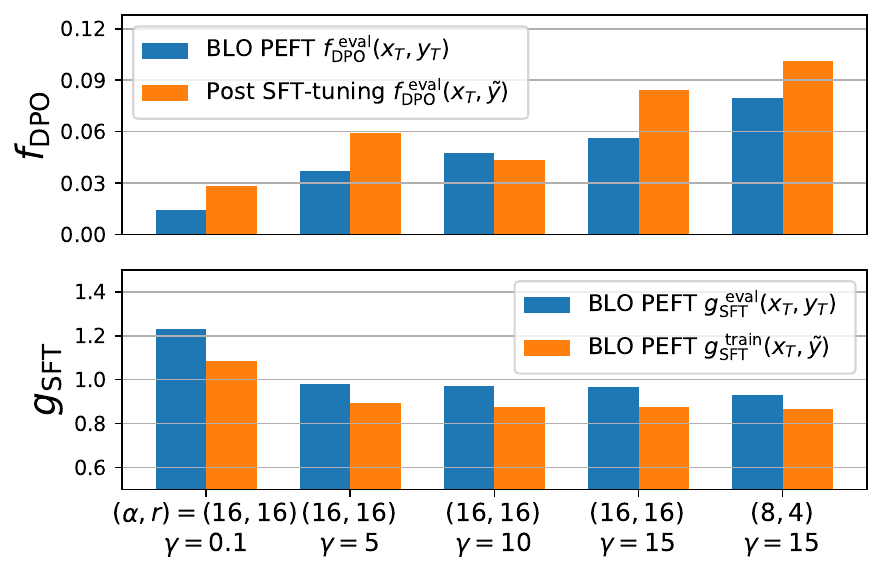}  
        \caption{Ablation study on penalty $\gamma$ and LoRA configuration \citep{hu2022lora} for \textsc{Pythia}-1b \citep{biderman2023pythia}. 
        }
        \label{fig: ablation_study}
    \end{minipage}
    \hfill
    \begin{minipage}{0.53\linewidth}
    \centering
    \includegraphics[width=0.89\linewidth]{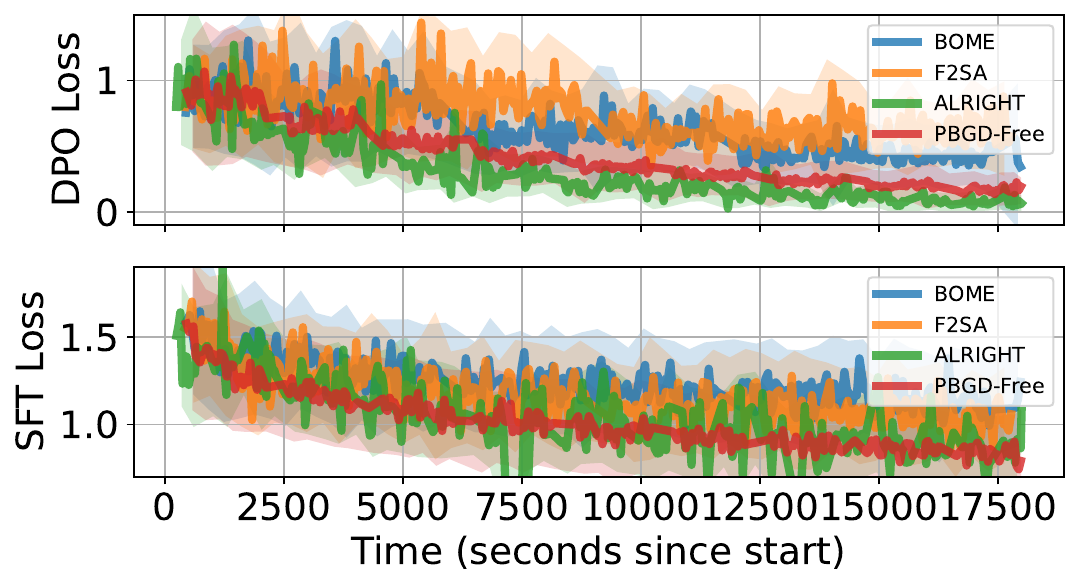}
    \caption{Train losses vs. time for different algorithms in solving \eqref{eq: bilevel repre LLM} (or bi-objective learning for ALRIGHT) on \textsc{Llama-3-3b} \citep{grattafiori2024llama}.}
    \label{fig:DPO_SFT_loss_time}
\end{minipage}
\vspace{-0.08cm}
\end{figure}

\begin{table}[t]
    \centering
    \begin{tabular}{c||c|c|c|c}
    \hline\hline
    Methods & $f_{\text{DPO}}^{\text{eval}}(x_T,y_T)$ & $g_{\text{SFT}}^{\text{eval}}(x_T,y_T)$ & $f_{\text{DPO}}^{\text{eval}}(x_T,\tilde{y})$& $g_{\text{SFT}}^{\text{train}}(x_T,\tilde{y})$\\
    \hline\hline
    \centering {V-PBGD \citep{kwon2023penalty}}  & 0.818 & 1.0309 & 0.8423 & 0.9533\\
    \hline
    \centering {BOME \citep{ye2022bome}} & 0.8332 & 1.1552 & 0.8402 & 0.9842\\
    \hline
    \centering {ALRIGHT \citep{fernando2024mitigating}}  & 0.8055 &  0.8656 & 0.8201 & 0.7855\\
    \hline
    \centering {\textbf{PBGD-Free}}   & \textbf{0.7837} & \textbf{0.8516} &\textbf{0.8088} & \textbf{0.6688}\\
    \hline\hline
    \end{tabular}
    \caption{Comparison of different algorithms for PEFT \textsc{Llama-3-3b} \citep{grattafiori2024llama}. Results show the DPO Loss $\downarrow$, SFT Loss $\downarrow$ for both the outcome $(x_T,y_T)$ trained on solving \eqref{eq: bilevel repre LLM} for different methods or and the outcome $(x_T,\tilde{y})$ from post-SFT-tuning on another dataset with fixed-backbone, using the same dataset fixed time of training for each. } \label{tab: LLM Training Results}
\end{table}

\subsection{Representation learning based problem}
\label{sec: representation PEFT}

In this section, we first compare our PBGD-Free methods with the fully single-loop version of ALT-PBGD (F$^2$SA~\cite{kwon2023fully}) on a representation learning problem using the NLSY dataset (cf. Sec.~\ref{sec: Representation learning problem on NLSY dataset}). Next, we reexamine the PEFT problem in \eqref{eq: bilevel repre LLM}. We provide additional details for the toy example in Fig.~\ref{fig:toy_example_convergence_results} in Sec.~\ref{app:toy_example} and present further experimental results on larger-scale problems in Sec.~\ref{exp:detail_peft_repre}.

\subsubsection{Representation learning problem on NLSY dataset \texorpdfstring{\citep{rothstein2019cohort}}{}}
\label{sec: Representation learning problem on NLSY dataset}

BLO has proven effective in representation learning for obtaining a joint backbone model $x$ that captures unified task features and generalizes well to downstream tasks by only tuning the head $y$ \citep{arora2020provable,xu2021deep,stock2003retrospectives,hu2023contextual,shui2022fair}.
We test our algorithm on a representation learning problem on the National Longitudinal Survey of Youth (NLSY) dataset~\citep{rothstein2019cohort}, following the experimental setup in \citep{shui2022fair}. This problem aims to learn representations to predict normalized income via 
\begin{align}
    \min_{x,y}f_{\text{MSE}}(x,y;D_1)\quad  \text{s.t.} \quad y\in \arg\min_{y}f_{\text{MSE}}(x,y;D_2)
\end{align}
where $D_1,D_2$ are datasets partitioned by gender. The representation model, parameterized $x$, consists of two fully connected layers (hidden size 200, ReLU activation), and the predictor, parameterized $y$, is a linear classification head.

We compare our PBGD-Free with 1 step inner iteration against ALT-PBGD (F$^2$SA \citep{kwon2023fully}) with 2 step inner iteration and the ITD algorithm from~\citep{shui2022fair}, following the experimental setup in~\citep{shui2022fair}. As shown in Table~\ref{tab: NLSY Training Results}, 
the performance gap is particularly notable in efficiency, where PBGD-Free is over twice as fast as F$^2$SA \citep{kwon2023penalty} and more than $30$ times faster than the ITD-based approach~\citep{shui2022fair}, primarily because it omits the value-function part and the inner loop of $y_g^*(x)$. Moreover, PBGD-Free achieves lower MSE than ALT-PBGD \citep{kwon2023fully}. This improvement stems from PBGD-Free's ability to avoid the bias $\gamma$-propagation inherent in the design for ALT-PBGD. When both algorithms are single-loop (or nearly single-loop for 2-step inner iteration), ALT-PBGD's reliance on a fixed penalty parameter $\gamma$ amplifies initial inner update biases throughout training, slowing convergence, detailed in Theorem~\ref{thm: ALT-PBGD}, while PBGD-Free eliminates these $\gamma$-dependent value function terms.

\begin{table}[t]
    \centering
    \begin{tabular}{c||c|c}
    \hline\hline
    Methods & MSE & Time (s) \\
    \hline\hline
    \centering {ALT-PBGD (F$^2$SA~\citep{kwon2023fully})} & $1.9331 \pm 0.0794$ & $12.33 \pm 0.34$ \\
    \hline
    \centering {Implicit \citep{shui2022fair}} & $2.1530 \pm 0.0455$ & $169.69 \pm 0.36$ \\
    \hline
    \centering {\textbf{PBGD-Free}} & $\bf{1.8916 \pm 0.1245}$ & $\bf{5.15 \pm 0.06}$ \\
    \hline\hline
    \end{tabular}
    \vspace{0.2cm}
    \caption{Performance results for different training methods on representation learning problem on NLSY-7k Dataset \citep{rothstein2019cohort}. The mean $\pm$ standard deviation is reported for both the mean MSE and the mean time over 5 random experiments on the test dataset.} \label{tab: NLSY Training Results}
\end{table}

\subsubsection{Additional details for toy example in Fig.~\ref{fig:toy_example_convergence_results}}
\label{app:toy_example}

In this section, we provide details for the toy example of PEFT BLO problem \eqref{eq: bilevel repre LLM} in Fig.~\ref{fig:toy_example_convergence_results}. We consider a binary classification setting where the model parameters $\theta = (x, y)$ consist of the UL variable $x$ and LL variable $y$, with $\theta \in \mathbb{R}^2$.
The model implements a 1D convolutional network with softmax activation:
\begin{verbatim}
class SoftmaxNN(nn.Module):
    def __init__(self):
        super().__init__()
        self.hidden = nn.Conv1d(in_channels=1, out_channels=1, 
                      kernel_size=2, stride=2, bias=False)
        self.activation = nn.Softmax(dim=1)
        self._init_weight() 
\end{verbatim}

\begin{wraptable}{r}{0.5\textwidth} 
    \centering
    \vspace{-0.7cm}
    \begin{tabular}{c|c|c}
    \hline\hline
    Input & Output & Feature \\ \hline\hline
    $X_1$ & $y = 0$ & $[1.0, 1.0, 0.5, 0.5]^\top$ \\
    $X_1$ & $y' = 1$ & $[1.0, 0.5, 0, 0.5]^\top$ \\
    $X_2$ & $y_w = 1$ & $[1.0, 0.5, 0.5, 0.5]^\top$ \\
    $X_2$ & $y_\ell = 0$ & $[0.5, 1.0, 1.0, 1.0]^\top$ \\
    \hline\hline
    \end{tabular}
    \caption{Dataset specification for toy example in Fig.~\ref{fig:toy_example_convergence_results}.}
    \label{tab: Toy example dataset}
    \vspace{-0.6cm}
\end{wraptable}

 We specify the SFT datasets $\mathcal{D}_{\text{SFT}} = \{(X_1, y)\}$, and the DPO dataset $\mathcal{D}_{\text{DPO}} = \{(X_2, y_w, y_\ell)\}$ in Table \ref{tab: Toy example dataset}.
The BLO problem is specified in \eqref{eq: bilevel repre LLM}, where $f_{\text{DPO}}$ consists of a DPO loss with $\beta=1$~\citep{rafailov2023direct} plus an $\ell_2$ regularization term (weight 0.01) and $g_{\text{SFT}}$ consists of a negative log-likelihood loss and the same regulariztion. The reference model is obtained via learning on $g_{\text{SFT}}(x,y)$ (parameterized with $(x=-5.34 ,y=-9.94)$). 

We apply our PBGD-Free algorithm in Algorithm \ref{alg: PBGD-Free} in comparison with ALT-PBGD (F$^2$SA~\citep{kwon2023fully}) with $\gamma = 15$, and 10-step inner loop to solve \eqref{eq: nabla F gam}, and $T=5000$ outer loop for both algorithms. The algorithm performance is presented in Fig.~\ref{fig:toy_example_convergence_results}.

\subsubsection{LLM PEFT problem} 
\label{exp:detail_peft_repre} 

In this section, we consider the PEFT problem in \eqref{eq: bilevel repre LLM} experimented on open-source large language models. 

\paragraph{Problem setting and its flatness}
\label{sec: problem setting and its flatness}

SFT for pre-trained LLMs enhances their adaptation to downstream tasks, while DPO is commonly used to align the model with human preferences. To achieve both goals, a straightforward approach is to sequentially optimize both objectives. However, this often leads to a forgetting issue \citep{fernando2024mitigating}. In practical applications, it is common to fine-tune only a lightweight head on SFT while keeping the backbone fixed or lightly tuned \citep{pfeiffer2020adapterfusion,zaken2021bitfit,ren2023prepare}. This naturally leads to a decomposition of the model into a backbone $x$ (e.g., attention weights) and an output head $y$.

In this way, the bilevel PEFT formulation \eqref{eq: bilevel repre LLM} addresses the catastrophic forgetting issue by prioritizing SFT at the LL to create a more reliable base model and use DPO at the UL to guide the human preference alignment. This hierarchical structure preserves task-specific knowledge in $y$, while guiding $x$ updates toward better preference alignment. 

This formulation operates in a standard post-SFT setting, where DPO fine-tunes a pretrained backbone after SFT has adapted the model to downstream tasks. Moreover, given that user preferences are typically consistent across tasks, this approach trains a representation model that captures these preferences, allowing only the head to be fine-tuned for future tasks and achieving parameter-efficient adaptation. In this way, the head can specialize for SFT tasks, while the backbone is optimized via DPO for both alignment and generalization.

As preliminarily demonstrated in Fig.~\ref{fig:toy_example_lipschitz}, this \textit{BLO PEFT problem} in \eqref{eq: bilevel repre LLM} features flatness (small $\delta$), which is further corroborated by the observation in experiment that
the LL solution $y_g^*(x)$ and $y_\gamma^*(x)$ have $\ell_2$-distance significantly greater than target $\epsilon$, suggesting a negligible flatness constant $\delta(x)$ by \eqref{eq: delta x}. 

In our experiments, we adopt the low rank adaptation (LoRA) \citep{hu2022lora} to the backbone $x$ for PEFT on \textsc{Llama-3-3b} \citep{grattafiori2024llama} and \textsc{Pythia-1b} \citep{biderman2023pythia}, using the \texttt{Dahoas/rm-hh-rlhf} dataset for DPO and the \texttt{OpenOrca} dataset \citep{longpre2023flan} for SFT. Our code is adapted from the bilevel LLM post-training library 
\url{https://github.com/Post-LLM/BIPOST}. The detailed setup is presented as follows.

\emph{General Setup.} We evaluate our PEFT framework \eqref{eq: bilevel repre LLM} using the \texttt{Dahoas/rm-hh-rlhf} dataset for DPO loss and the \texttt{OpenOrca} dataset for SFT loss. For training, we test one \textsc{Pythia}-1b \citep{biderman2023pythia} model with $1800$ samples for each dataset (batch size $16$) and the \textsc{Llama-3-3b} \citep{fang2024llama} model with $4800$ samples (batch size $32$).
Both models are adapted with LoRA (\textsc{alpha} $16$, \textsc{rank} $16$) and we treat LoRA PEFT weights on the attention layers as $x$, the last layer linear head as $y$. The learning rate is set to $1\times 10^{-5}$, using Adam \citep{kingma2014adam} as the optimizer. All experiments were conducted on a cluster of NVIDIA A100 GPUs, each with 40 GB of memory. Training was performed using PyTorch with the DeepSpeed library \url{https://github.com/deepspeedai/DeepSpeed} to optimize memory usage and distributed training efficiency. We consider a time-limited experiment under a consistent computational budget, reflecting real-world constraints where training time is often a critical factor.

\emph{Algorithm hyperparameter.} We use a penalty constant of $\gamma=10$ for our proposed PBGD-Free algorithm (Algorithm \ref{alg: PBGD-Free}) with a single inner loop ($K=1$). For the baseline F$^2$SA algorithm \citep{chen2024finding,kwon2023penalty}, we set $\gamma=10$ with $K=3$ inner updates for training \textsc{Llama-3-3b} \citep{grattafiori2024llama}, and $K=5$ for \textsc{Pythia}-1b \citep{biderman2023pythia}. 
For the BOME algorithm, we similarly use $K=3$ and $K=5$ inner loops, adopting its hyperparameter $\eta=0.5$ for calculating the penalty constant, as suggested in \citep{ye2022bome}. For the ALRIGHT algorithm \citep{fernando2024mitigating}, we use its default setting of $\lambda = 0.5$ as suggested in literature \citep{fernando2024mitigating}. Since the ALRIGHT algorithm in \citep{fernando2024mitigating} is a bi-objective learning algorithm that does not have the representation learning capability, we examine it on an alternative formulation $\min_{x,y}[f_{\text{DPO}(x,y)},g_{\text{SFT}}(y)]$.

\paragraph{Ablation study and main experimental results}

\begin{wrapfigure}{r}{0.5\textwidth} 
        \vspace{-0.34cm}
        \centering
        \includegraphics[width=0.5\textwidth]{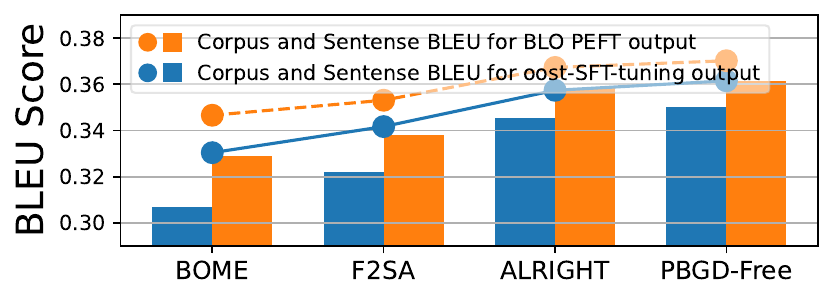}  
        \caption{BLEU-4 Corpus and BLEU-4 Sentence Score ($\uparrow$) for different algorithms for PEFT on \textsc{Llama-3-3b} \citep{grattafiori2024llama}.}
        \label{fig: bleu}
\vspace{-0.4cm}
\end{wrapfigure}

In this experiment, we consider evaluating methods on both \textbf{S1)} \textit{BLO PEFT learning phase via \eqref{eq: bilevel repre LLM}} to obtain a preference backbone $x$, and \textbf{S2)} \textit{post-SFT tuning on a new dataset} with the obtained preference backbone model $x$, to verify the representation quality and transferability of the backbone. We reports the ablation study and main experimental results as follows.

\textbf{Trade-off between DPO and SFT under different $\gamma$.} We conduct an ablation study on the \textsc{Pythia}-1b to test the impact of the penalty constant $\gamma$ and LoRA configuration on the PBGD-Free method. We report the DPO and SFT loss under different settings for both $(x_T,y_T)$ learned from \textbf{S1)} and $(x_T,\tilde{y})$ from \textbf{S2)} in Fig.~\ref{fig: ablation_study}. According to Fig.~\ref{fig: ablation_study}, increasing $\gamma$ degrades DPO performance while improving SFT for \textbf{S1)}, indicating that a larger $\gamma$ provides better LL optimality. 
Notably, the SFT improvement beyond $\gamma=10$ is marginal for \textbf{S1)}, while the DPO performance significantly deteriorates, suggesting that $\gamma \approx 10$ offers the best balance as our theory predicts. 

\textbf{Faster convergence over BLO baselines and stable training over bi-objective.} Since second-order BLO algorithms are inefficient in large-scale LLM training, we consider first-order methods F$^2$SA \citep{kwon2023penalty} and BOME \citep{ye2022bome} as BLO baselines. 
As shown in Fig.~\ref{fig:DPO_SFT_loss_time}, PBGD-Free converges faster than the BLO baselines. We additionally compare with ALRIGHT \citep{fernando2024mitigating}, an effective bi-objective algorithm, to validate the representation capability of \textit{BLO PEFT \eqref{eq: bilevel repre LLM}} formulation. ALRIGHT \citep{fernando2024mitigating} exhibits less stability during training (Fig.~\ref{fig:DPO_SFT_loss_time}), likely due to alternating between DPO and SFT objectives. %

\textbf{Transferability of preference backbone and strong SFT performance.} Compared with \textbf{S1)}, PBGD-Free in Fig.~\ref{fig: ablation_study} shows enhanced SFT with comparable DPO performance on \textbf{S2)}, suggesting it learns a transferable preference backbone $x$ through BLO \eqref{eq: bilevel repre LLM}.  Table \ref{tab: LLM Training Results} further quantifies these findings for other baselines, demonstrating that PBGD-Free achieves superior DPO and SFT performance. 
Notably, the backbone model $x$ obtained by PBGD-Free attains the lowest SFT and DPO loss on \textbf{S2)}, verifying the transferability of PBGD-Free. To further evaluate the quality of generated output, Fig.~\ref{fig: bleu} corroborates the SFT performance using the evaluation metrics BLEU score \citep{papineni2002bleu}, where our method outperforms all baselines, further justifying its superiority in learning a good representation. 

Fig.~\ref{fig:Post_SFT_tunin_DPO_SFT_loss_time_1b} further illustrates the performance of the outputs from BLO PEFT learning \eqref{eq: bilevel repre LLM} in the subsequent post-SFT-tuning phase \textbf{S2)}. The BLO baselines (F$^2$SA \citep{chen2024finding,kwon2023penalty} and BOME \citep{ye2022bome}), which did not achieve convergence in the initial PEFT phase due to their higher time complexity, tend to sacrifice DPO performance when improving SFT performance during post-SFT tuning \textbf{S2)}. 
In contrast, PBGD-Free algorithm and its bi-objective counterpart ALRIGHT \citep{fernando2024mitigating} demonstrate the ability to preserve strong preference alignment (DPO) while conducting SFT training in \textbf{S2)}. This shows that the preference backbone $x$ learned by both of them can be adapted to new task by fine-tuning only the linear head to achieve strong SFT performance. 
Notably, the BLO PEFT outputs trained by PBGD-Free achieves better SFT performance with substantially lower SFT loss, highlighting the advantage of the prioritization of SFT in our BLO formulation \eqref{eq: bilevel repre LLM}. This structure allows for a more powerful SFT tuning head, whereas bi-objective training methods tend to oscillate between potentially conflicting objectives, thereby limiting their post-SFT performance. 

\textbf{Better SFT performance while maintaining preference learning.} As illustrated in Table \ref{tab:sft-eval-examples}, the outputs generated by PBGD-Free demonstrate more precise and semantically accurate extraction, highlighting its superior SFT performance. In Fig.~\ref{fig:Post_SFT_tunin_DPO_SFT_loss_time_3b}, we present the loss metrics performance throughout post-SFT-tuning phase for \textsc{Llama-3-3b} \citep{grattafiori2024llama} in addition to the results presented in Sec.~\ref{sec: exp}. We observe that our backbone $x$ trained on BLO PEFT \eqref{eq: bilevel repre LLM} via PBGD-Free retains its lowest DPO rates throughout post-SFT-tuning

Fig.~\ref{fig: semantic_DPO} shows the quantitative results of the preference alignment using average reward gap and win rate. Together with Fig.~\ref{fig: bleu}, they indicate that the backbone preference model $x$ by the PBGD-Free maintains the first-tier DPO performance for preference alignment while enhancing the SFT performance by only fine-tuning the linear head. The slight DPO drop in Fig.~\ref{fig: semantic_DPO} of PBGD-Free compared with ALRIGHT is because it prioritizes better SFT performance, which restricts the feasible search space of representation model optimizing at the UL. However, since the representation evaluation criterion prioritizes strong SFT performance achieved by fine‑tuning only the linear head, and treats preference alignment as a secondary goal, PBGD‑Free remains the top‑performing method. Moreover, according to Fig.~\ref{fig:DPO_SFT_loss_time}, PBGD-Free is more stable during the training compared with ALRIGHT. Additionally, Table \ref{tab:sft-eval-examples} provides the SFT output comparison given by PBGD-Free and F$^2$SA on \textsc{Pythia}-1b \citep{biderman2023pythia}, \textsc{Llama-3-3b} \citep{grattafiori2024llama}, from which we can see that both methods improve the response quality over the pre-trained model through BLO PEFT \eqref{eq: bilevel repre LLM}, while PBGD-Free generates better responses and follows the human instructions well.

\textbf{Higher-rank LoRA enables finding better preference backbone via PBGD-Free. } The last $2$ columns in Fig.~\ref{fig: ablation_study} show that a higher-rank LoRA better preserves DPO with comparable SFT performance. It is likely because a higher-rank LoRA provides more over-parameterization, which ensures a more benign optimization landscape for the representation parameter $x$ \citep{yaras2024compressible,ward2023convergence,liuoptimization} and thus enables globally finding a better representation model $x$ \citep{xiao2024unlocking}.

\begin{table}[t]
    
    \begin{minipage}{0.96\linewidth}
    \renewcommand{\arraystretch}{1.6} 
    \begin{tabular}{p{0.31\linewidth} p{0.31\linewidth} p{0.31\linewidth}}
    \hline\hline
    \multicolumn{3}{c}{\textbf{Example of SFT Evaluation Performance}} \\
    \hline\hline
    \multicolumn{3}{p{\linewidth}}{\textbf{Human}: Generate an approximately fifteen-word sentence that describes all this data: Midsummer House eatType restaurant; Midsummer House food Chinese; Midsummer House priceRange moderate; Midsummer House customer rating 3 out of 5; Midsummer House near All Bar One} \\
    \textbf{Pythia-1b \citep{biderman2023pythia}}:\newline
    Midsummer House \textcolor{red}{staff} a restaurant \textcolor{red}{priced restaurant} restaurant. a \textcolor{red}{good-5 star} rating. and \textcolor{red}{in} All Bar One.
    & \textbf{BLO-PEFT (F$^2$SA \citep{kwon2023penalty})}:\newline
    Midsummer House is a \textcolor{red}{restaurant priced restaurant} restaurant with a \textcolor{orange}{3-5 customer rating}. located \textcolor{green}{near} All Bar One.  
    & \textbf{BLO-PEFT (PBGD-Free)}:\newline
    Midsummer House is a \textcolor{green}{moderately priced} \textcolor{green}{Chinese restaurant} with a \textcolor{green}{3/5 customer rating}. located \textcolor{green}{near} All Bar One. 
    \\
    \hline
    \multicolumn{3}{p{\linewidth}}{\textbf{Human}: You will be given a definition of a task first, then some input of the task. This task is about using the specified sentence and converting the sentence to Resource Description Framework (RDF) triplets of the form (subject, predicate, object). 
    The RDF triplets generated must be such that the triplets accurately capture the structure and semantics of the input sentence. The input is a sentence and the output is a list of triplets of the form [subject, predicate, object] that capture the relationships present in the sentence. When a sentence has more than 1 RDF triplet possible, the output must contain all of them. AFC Ajax (amateurs)'s ground is Sportpark De Toekomst where Ajax Youth Academy also play.} \\
    \textbf{Llama-3-3b~\citep{grattafiori2024llama}}:\newline
    [["\textcolor{red}{AjaxFC Ajax (amateurs)}", \textcolor{red}{"playsGround"}, "Sportpark De Toekomst"], 
    ["Ajax Youth Academy",  \textcolor{orange}{"has at"}, "Sportpark De Toekomst"]] 
    & \textbf{BLO-PEFT (F$^2$SA \citep{kwon2023penalty})}:\newline
    [["\textcolor{red}{AjaxFC Ajax (amateurs)}", \textcolor{orange}{"plays ground"}, "Sportpark De Toekomst"], 
    ["Ajax Youth Academy", \textcolor{orange}{"has at"}, "Sportpark De Toekomst"]] 
    & \textbf{BLO-PEFT (PBGD-Free)}:\newline
    [["\textcolor{green}{AFC Ajax (amateurs)}", \textcolor{green}{"has ground"}, "Sportpark De Toekomst"], 
    ["Ajax Youth Academy", \textcolor{green}{"plays at"}, "Sportpark De Toekomst"]] 
    \\
    \hline
    \hline
    \end{tabular}
    \vspace{0.1cm}
    \end{minipage}
    \caption{Examples of SFT evaluation performance for \textsc{Pythia}-1b \citep{biderman2023pythia}, \textsc{Llama-3-3b} \citep{grattafiori2024llama} and their corresponding BLO-PEFT \eqref{eq: bilevel repre LLM} results via our PBGD-Free Algorithm \ref{alg: PBGD-Free} and basline F$^2$SA \citep{kwon2023penalty}. Text marked in \textcolor{red}{red} indicates incorrect outputs, \textcolor{orange}{orange} indicates partially correct outputs that follow some of the instructions, and \textcolor{green}{green} indicates fully correct outputs that match the expected instructions.}
    \label{tab:sft-eval-examples}
\end{table}

\begin{figure}[t]
    \centering
    \small
    \begin{minipage}{0.5\linewidth}
        \centering
        \vspace{-0.2cm}
        \centering
        \includegraphics[width=0.9\textwidth]{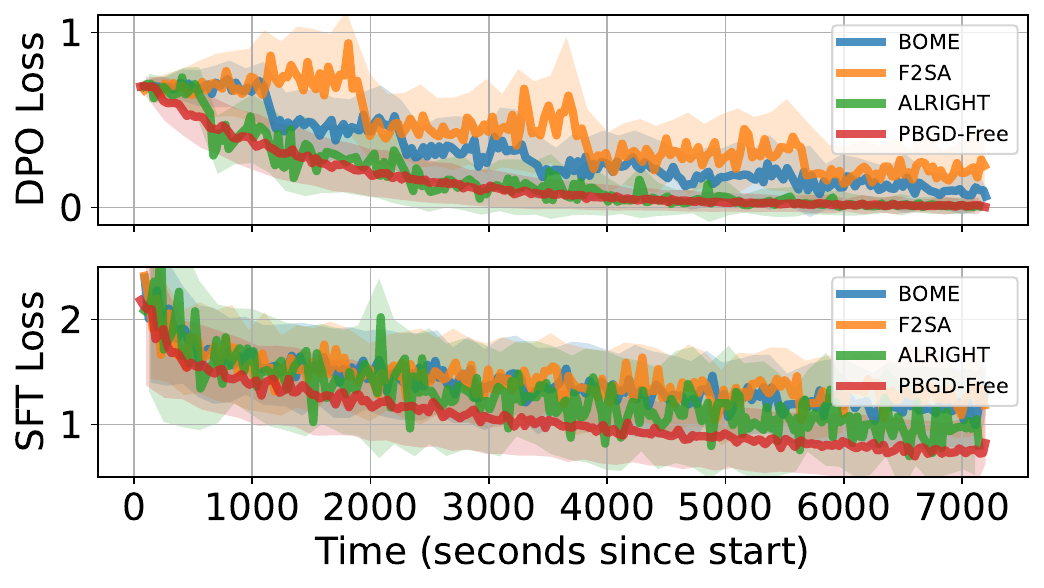}  
        \caption{Train losses vs. time with \textsc{stride = 50} for different algorithms in solving \eqref{eq: bilevel repre LLM} (or biobjective learning for ALRIGHT \citep{fernando2024mitigating}) on \textsc{Pythia}-1b \citep{biderman2023pythia}.}
        \label{fig: DPO_SFT_loss_time_1b}
    \end{minipage}
    \hfill
    \begin{minipage}{0.46\linewidth}
        \centering
        \vspace{-0.4cm}
        \includegraphics[width=0.99\linewidth]{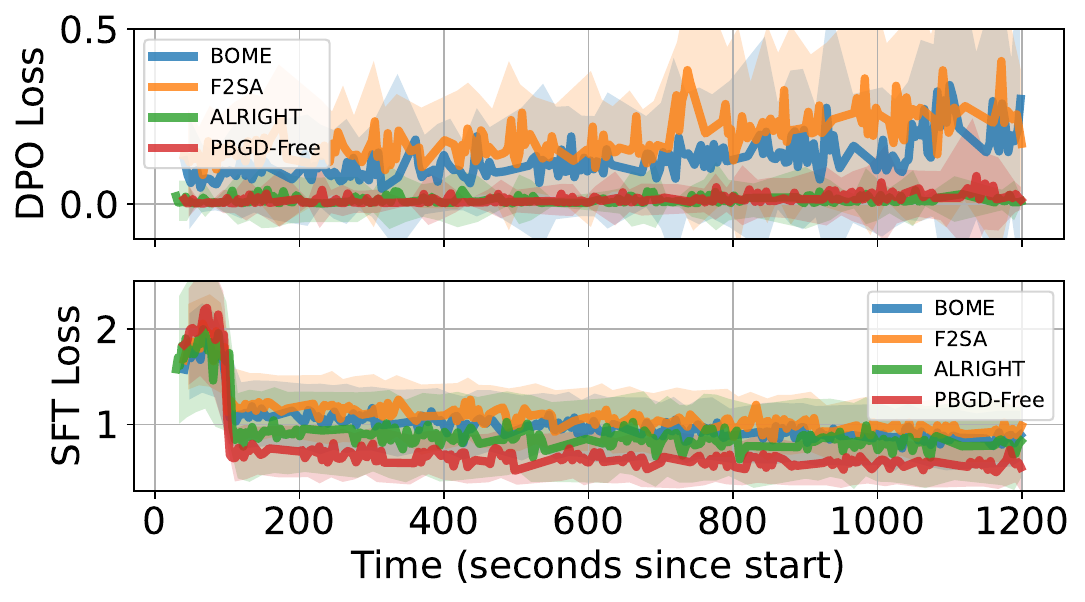}
        \vspace{-.2cm}
        \caption{Train losses vs. time with \textsc{stride = 50} for different algorithms in Post SFT-tuning phase on \textsc{Pythia}-1b \citep{biderman2023pythia}.}
        \label{fig:Post_SFT_tunin_DPO_SFT_loss_time_1b}
    \end{minipage}
    \vspace{-0.1cm}
\end{figure}

\begin{figure}[t]
    \centering
    \small
    \begin{minipage}{0.5\linewidth}
        \centering
        \centering
        \includegraphics[width=0.9\textwidth]{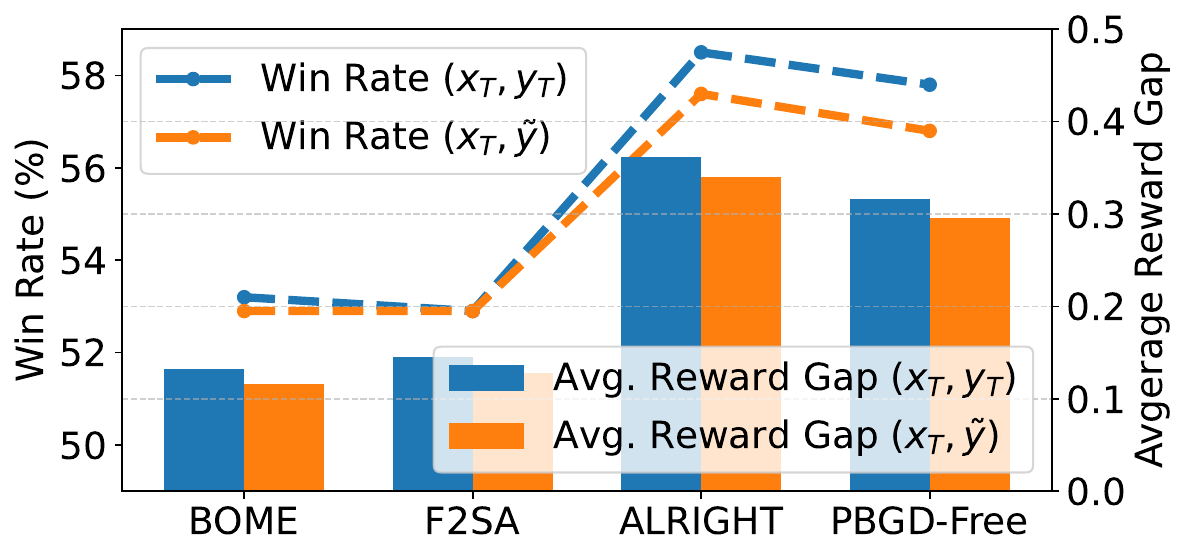}  
        \vspace{-0.2cm}
        \caption{Average Reward Gap ($\uparrow$) and Win Rate ($\uparrow$) for different algorithms for PEFT \textsc{Llamma-3-3b} on \citep{grattafiori2024llama} with the output $(x_T,y_T)$ via each method in \textbf{S1)} and the outcome $(x_T,\tilde{y})$ from post-SFT-tuning on another dataset with fixed-backbone in \textbf{S2)}. }
        \label{fig: semantic_DPO}
    \end{minipage}
    \hfill
    \begin{minipage}{0.46\linewidth}
        \centering
        \includegraphics[width=0.98\linewidth]{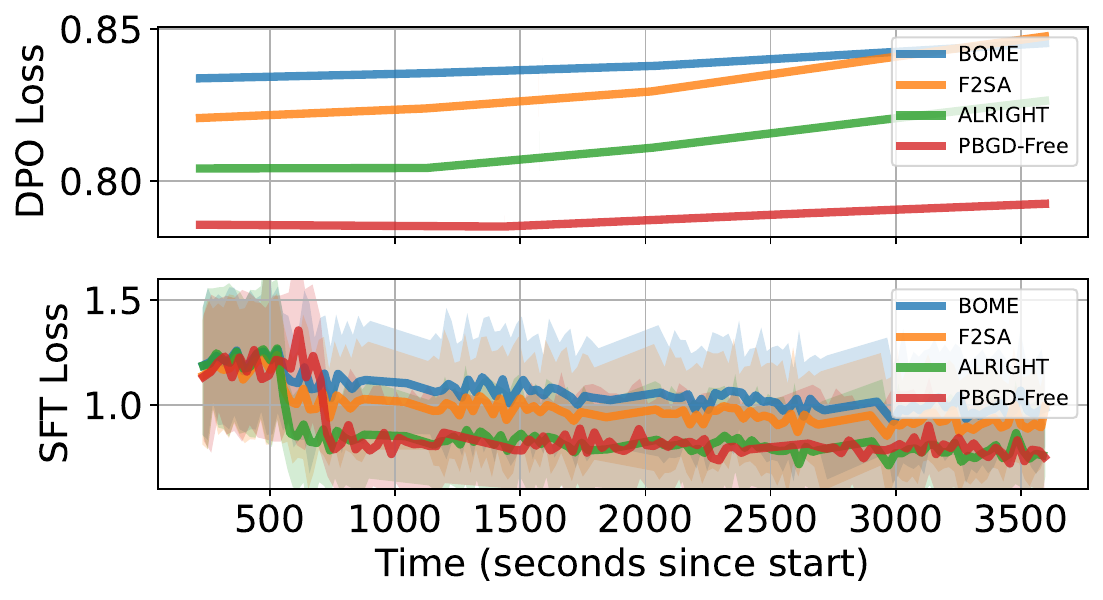}
        \vspace{-.2cm}
        \caption{Evaluation of DPO losses and train SFT loss vs. time with \textsc{stride = 50} for different algorithms in Post SFT-tuning phase on for PEFT \textsc{Llamma-3-3b} \citep{grattafiori2024llama}.}
        \label{fig:Post_SFT_tunin_DPO_SFT_loss_time_3b}
    \end{minipage}
    \vspace{-0.1cm}
\end{figure}

\subsection{SVM hyper-parameter training}
\label{sec: SVM exp}

\begin{table}[tb]
\centering
\begin{tabular}{lcc}
\hline\hline
\textbf{Method} & \textbf{Test Accuracy} & \textbf{Time (s)} \\
\hline\hline
LVHBA      & $0.7593 \pm 0.0397$ & ($2.47 \pm 0.23$) \\
GAM        & $0.7182 \pm 0.0360$ & ($4.48 \pm 0.22$) \\
BIC-GAFFA  & $0.7568 \pm 0.0366$ & ($0.31 \pm 0.01$) \\
PBGD-BLOCC & $\bf{0.7758 \pm 0.0292}$ & ($2.64 \pm 0.59$) \\
\textbf{PBGD-Free}  & $0.7699 \pm 0.0313$ & ($\bf{0.17 \pm 0.09}$) \\
\hline\hline
\end{tabular}
\vspace{0.2cm}
\caption{Test accuracy and training time cost for the SVM problem on the diabetes dataset of PBGD-Free (Algorithm~\ref{alg: PBGD-Free}) in comparison with BLOCC~\citep{jiang2024primal}, LV-HBA~\citep{yao2024constrained}, Bic-GAFFA~\citep{yao2025overcoming}, and GAM~\citep{xu2023efficient}. The first row shows the test accuracy (mean $\pm$ std), and the second row (in brackets) shows the running time until the validation loss change is smaller than $10^{-5}$ or a maximum of 50 epochs.}
\label{tab:final_results_valstop}
\end{table}

\begin{figure}[tb]
\centering
    \begin{minipage}{0.98\textwidth}
        \centering
        \begin{minipage}{0.3\textwidth}
            \centering
            \includegraphics[width=\textwidth]{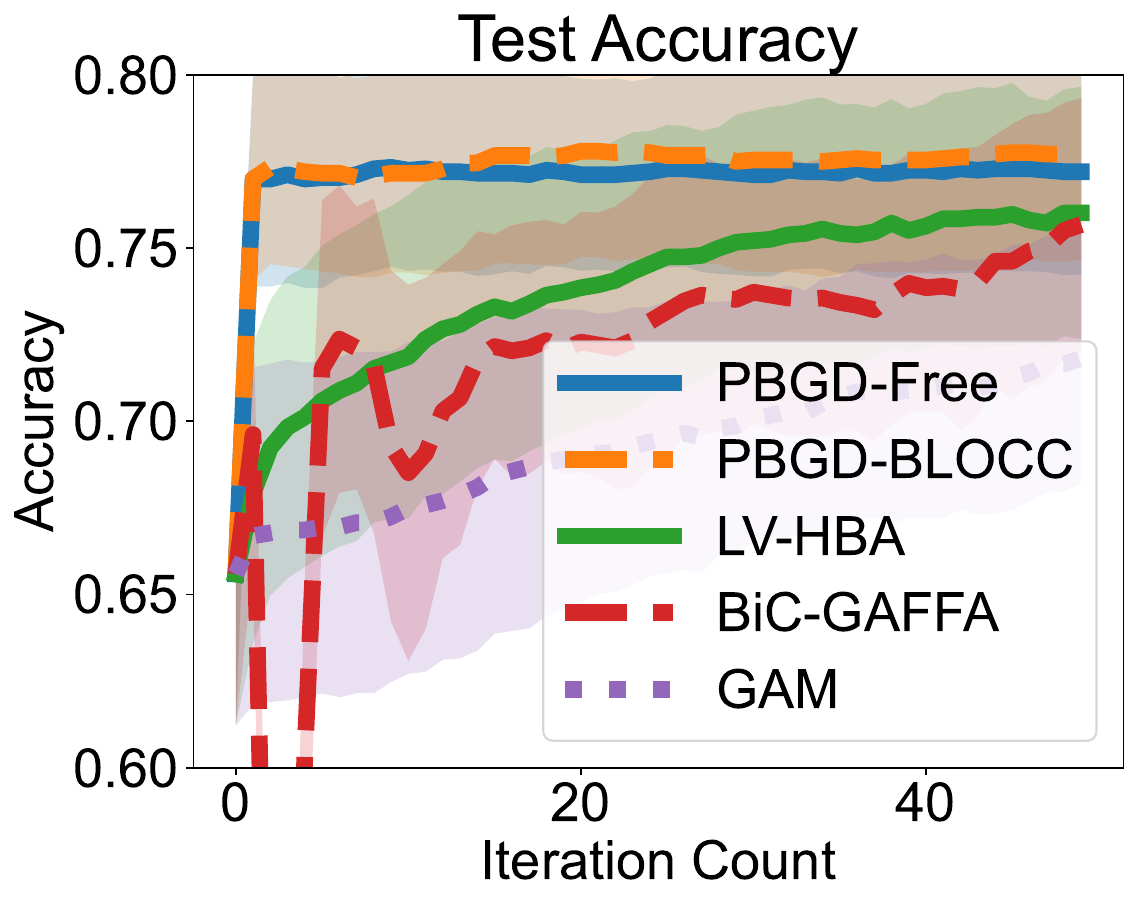}
        \end{minipage}%
        \hfill
        \begin{minipage}{0.33\textwidth}
            \centering
            \includegraphics[width=\textwidth]{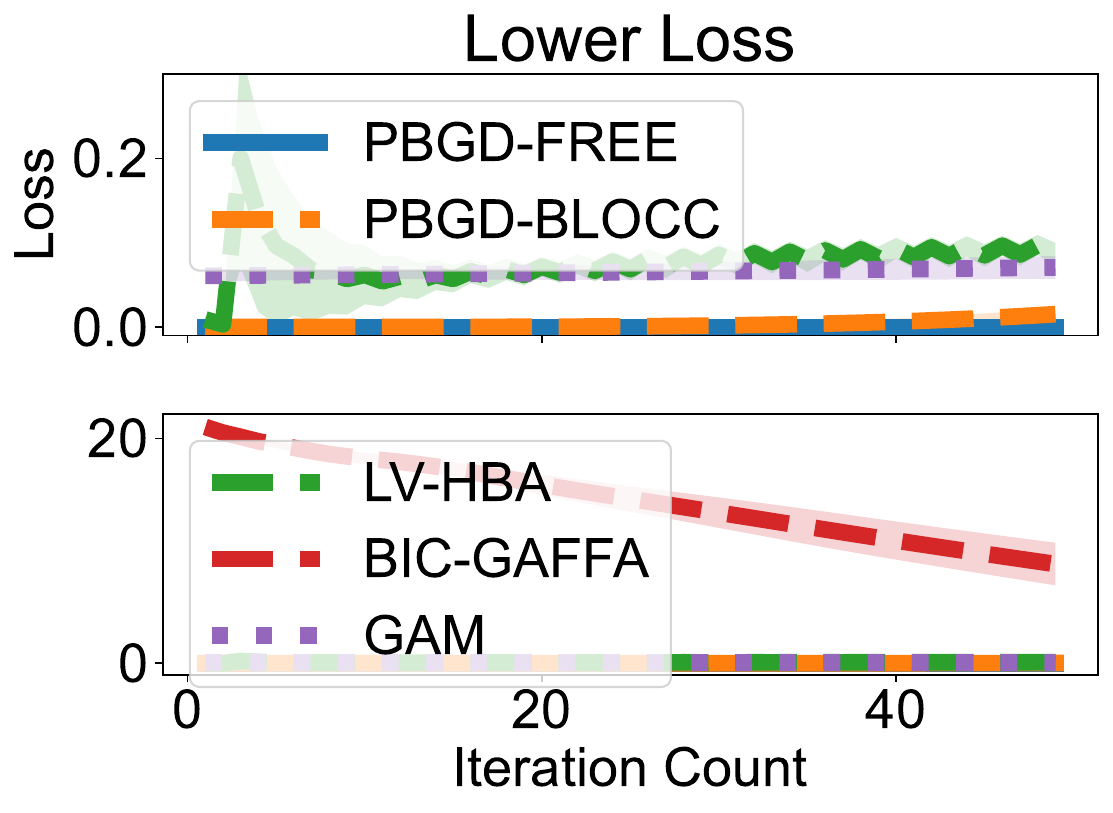}
        \end{minipage}%
        \hfill
        \begin{minipage}{0.33\textwidth}
            \centering
            \includegraphics[width=\textwidth]{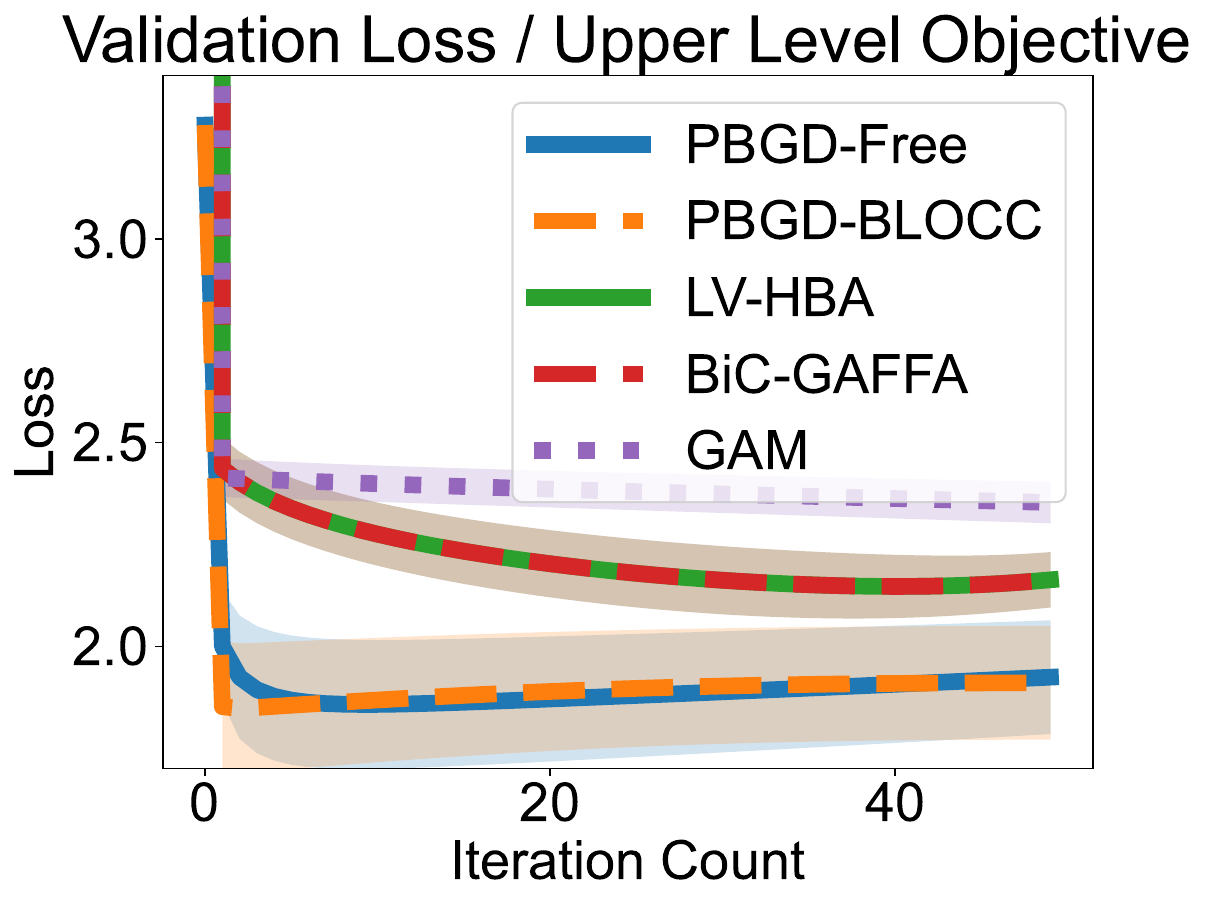}
        \end{minipage}
            \vspace{0.2cm}
        \caption{Iteration performance of PBGD-Free (Algorithm~\ref{alg: PBGD-Free}) in comparison with BLOCC~\citep{jiang2024primal}, LV-HBA~\citep{yao2024constrained}, Bic-GAFFA~\citep{yao2025overcoming}, and GAM~\citep{xu2023efficient}. Test accuracy (left), lower-level loss $g(x,y)$ (middle), and upper-level loss $f(x,y)$ (right) are reported for the SVM problem on the diabetes dataset, averaged over 20 random train–validation–test splits. Bold lines denote the mean, and shaded regions represent the standard deviation.}
        \label{fig:hyperparam_opt}
    \end{minipage}%
    \hfill
\end{figure}

Selecting the violation upper bound for the soft-margin SVM is naturally formulated as a bi-level optimization problem with coupled constraints \citep{xu2023efficient, yao2024constrained, jiang2024primal}. In this section, we present numerical experiments on solving the following problem.
\begin{align} 
    & \min_c \frac{1}{|\mathcal{D}_{val}|} \sum_{(z_{{\rm val}}, l_{{\rm val}}) \in \mathcal{D}_{\rm val}} \exp\left({1-l_{{\rm val}} \left(z_{{\rm val}}^\top w^*+b^* \right)} \right) +  \frac{1}{2}\| c \|^2 \nonumber\\
    &\text{with}~   w^* ,b^* = \arg\min_{w ,b} \frac{1}{2}\|w\|^2  \quad\text{s.t.} \quad 1-l_{{\rm tr}, i}(z_{{\rm tr}, i}^\top w +b) \leq c_i, ~ \forall  (z_{tr,u},l_{tr,i})\in \mathcal{D}_{\rm tr}.  \label{eq: svm_reform_llcons}\\\nonumber
\end{align}
The upper-level objective is the validation loss evaluated on the validation set $\mathcal{D}_{\text{val}}$, regulated by the violation magnitude $\|c\|^2$. Here, we empirically verify that $f(x,\cdot)$ satisfies $(\alpha = 1.25, \delta = 10^{-10})$-flatness around $(w^*,b^*)$. The lower-level trains hyperplane parameters on the training set $\mathcal{D}_{\text{tr}}$, where the soft margin violation $c_i$ is upper-bounded by a hyperparameter $c_i$.

We evaluate our value-function-free algorithm (Algorithm~\ref{alg: PBGD-Free}) with $\gamma = 20$ on training a linear SVM using the diabetes dataset \citep{Dua:2017}, and compare its performance with methods designed for BLO with coupled constraints, namely PBGD-BLOCC \citep{jiang2024primal}, GAM \citep{xu2023efficient}, LV-HBA \citep{yao2024constrained}, and BiC-GAFFA \citep{yao2025overcoming}.

Table~\ref{tab:final_results_valstop} shows the test accuracy and training time for different methods on the diabetes dataset: PBGD-BLOCC achieves the highest accuracy, PBGD-Free reaches similar accuracy with much lower training time, and both PBGD-BLOCC and PBGD-Free significantly outperform the other methods in accuracy.
Figure~\ref{fig:hyperparam_opt} shows that PBGD-Free achieves performance comparable to PBGD-BLOCC. Both PBGD-Free and PBGD-BLOCC attain strong test performance, minimal lower-level violation, and the lowest upper-level objective, demonstrating their effectiveness in solving BLO problems with coupled constraints. Moreover, PBGD-Free has an advantage over PBGD-BLOCC in terms of per-iteration cost, since BLOCC requires double inner loops while PBGD-Free uses only one. These results are consistent with the theoretical insights established in Lemma~\ref{lemma: tighter bound for y gam yg} and Theorem~\ref{thm: no value function for BLO with CC}.

\section{Conclusions}

In summary, we address the challenges of BLO with coupled constraints by providing a tighter smoothness anlysis of value-function-based penalty reformulations. Building upon this, we propose the alternating PBGD scheme and develop an efficient PBGD-Free algorithm, which eliminates the costly value-function computation while maintaining practical efficiency. To establish the theoretical guarantees of the PBGD-Free algorithm, we introduce a novel flatness condition that relaxes the traditional Lipschitz continuity assumption, particularly when the Lipschitz constant is small. We provide both theoretical and empirical insights into this flatness condition and demonstrate its validity. Under this condition, we further prove that PBGD-Free achieves a convergence rate comparable to that of standard gradient descent. Our theoretical insights, complemented by empirical validation on both toy and large-scale tasks, demonstrate that efficient and scalable BLO methods are achievable even in the presence of complex constraint structures. 









\section*{Declarations}

\paragraph{Funding}

The work was supported by the National Science Foundation Projects 2401297, 2532349 and 2532653, and by the Cisco Research Award.

\paragraph{Competing interests}
The authors declare that they have no competing interests.

\paragraph{Ethics approval and consent to participate}
Not applicable.

\paragraph{Consent for publication}
Not applicable.

\paragraph{Data availability}
The datasets and models analyzed during the current study are fully described in the main manuscript. No additional datasets or code are required to reproduce the results reported in this paper.

\paragraph{Materials availability}
Not applicable.

\paragraph{Code availability}
No additional code is required to reproduce the results; all algorithms and procedures are fully described in the main manuscript. The implementations are based on publicly available repositories, which are cited in the manuscript.

\paragraph{Author contribution}
Liuyuan Jiang (first author) developed the algorithms, performed the theoretical and empirical analysis, and wrote the majority of the manuscript.  
Quan Xiao (second author) provided guidance on the research approach and discussed theoretical ideas throughout the project.  
Lisha Chen (third author) supervised the research, edited the manuscript, and verified the correctness of the results.  
Tianyi Chen (fourth author, corresponding author) provided overall supervision, guidance on the research direction, and oversaw the manuscript preparation.  
All authors read and approved the final manuscript.

\paragraph{Previous versions}


A previous version of this work was presented at the Thirty-Ninth Annual Conference on Neural Information Processing Systems (NeurIPS), 2025, titled ``Beyond Value Functions: Single-Loop Bilevel Optimization under Flatness Conditions'' \cite{jiang2025beyond}. That version focused on the analysis of the PBGD-Free algorithm for unconstrained bilevel optimization problems under the lower-level PL condition.  
The current manuscript is an extended journal version that:  
i) extends the algorithm from the NeurIPS paper to the coupled constraint setting;  
ii) provides a more general proof for the development of the second-order directional derivative of the penalty problem; and,  
iii) includes additional theoretical results, empirical experiments, and discussions not present in the previous conference versions.









\bibliography{reference,bilevel}
\bibliographystyle{spbasic} 

\appendix
\section{Proof of Theorem~\ref{theorem: smoothness}}
\label{appendix: Theorem smoothness}

Before proceeding to the full proof of Theorem~\ref{theorem: smoothness}, we present two useful lemmas.

\begin{lemma}
\label{lemma: implicit directional derivative lipschitz}
    Suppose all assumptions in Theorem~\ref{theorem: smoothness} hold. Denote $\tilde{g}(x,y;\delta)=\delta f(x,y)+g(x,y)$. For arbitrary $d$, the directional derivative
    \begin{align*}
    D_d(y_{1/\delta}^*(x);\delta) = \arg\min_{\eta \in \mathcal{C}_{\tilde{\mathcal{Y}}}(y_{1/\delta}^*(x))} \frac{1}{2} \eta^\top \nabla_{yy} \tilde{g}_\delta(x,y_{1/\delta}^*(x)) \eta + (\nabla_{yx} \tilde{g}_\delta(x,y_{1/\delta}^*(x)) d)^\top \eta
    \end{align*}
    is Lipschitz in $\delta$.
\end{lemma}
\begin{corollary}
Under the same assumptions as in Lemma~\ref{lemma: implicit directional derivative lipschitz}, for $\gamma > 0$, the directional derivatives satisfy
\begin{align*}
    \| D_d(y_{\gamma}^*(x)) - D_d(y_g^*(x)) \| = \mathcal{O}(\gamma^{-1}).
\end{align*}
\end{corollary}
\begin{proof}
    The formulation of $D_d(y_{1/\delta}^*(x);\delta)$ follows directly from Lemma~\ref{lemma: implicit gradient}. 
    By \cite[Theorem 2.87]{bonnans2013perturbation}, Robinson CQ, which is guaranteed by LICQ for convex optimization, ensures that
    the critical cone mapping $y \rightarrow \mathcal{C}_{\tilde{\mathcal{Y}}}(y)$ is graphically regular.
    This ensures Hausdorff-Lipschitz of $\mathcal{C}_{\tilde{\mathcal{Y}}}(y)$ \cite[Theorem 9.40]{rockafellar1998variational}. 
    Therefore, all conditions in \cite[Theorem 5.4.4]{facchinei2003finite} are checked and it directly gives the Lipschitzness result.
\end{proof}

\begin{lemma}
\label{lemma: directional derivative bound}

Suppose all assumptions in Theorem~\ref{theorem: smoothness} hold. The norm of implicit directional derivative $D_d(y_g^*(x))$ for arbitray unit direction $d\in \mathbb{R}^{d_x}$ in \eqref{eq: implicit directional derivative} is bounded by
\begin{align*}
    \|D_d(y_g^*(x))\| \leq \frac{l_{g,1}}{\mu_g}.
\end{align*}
\end{lemma}
\begin{proof}

From Lemma~\ref{lemma: implicit gradient}, we know
\begin{align*}
    D_d(y^*_g(x)) = \Proj_{\mathcal{C}_{\tilde{\mathcal{Y}}}(y^*_g(x))}^{\nabla_{yy} g(x,y^*_g(x))} \Big( - (\nabla_{yy} g(x,y^*_g(x)))^{-1} \nabla_{yx}g(x,y^*_g(x)) d \Big).
\end{align*}
Here, $\nabla_{yy} g(x,y_g^*(x))$-norm, which is given by $\|v\|_{\nabla_{yy} g(x,y_g^*(x))}^2 := v^\top \nabla_{yy} g(x,y_g^*(x)) v$, is well defined as $\nabla_{yy} g(x,y_g^*(x))\succeq \mu_g I$ is positive definite as $\mu_g>0$ from strong convexity. In this way, we obtain
\begin{align*}
    \mu_g \|D_d(y^*_g(x))\|^2 \leq & \|D_d(y^*_g(x))\|_{\nabla_{yy} g(x,y_g^*(x))}^2\\
    \leq &\| - (\nabla_{yy} g(x,y^*_g(x)))^{-1} \nabla_{yx}g(x,y^*_g(x)) d \|_{\nabla_{yy} g(x,y_g^*(x))}^2\\
    = & \| d^\top \nabla_{xy}g(x,y^*_g(x))(\nabla_{yy} g(x,y^*_g(x)))^{-1} \nabla_{yx}g(x,y^*_g(x)) d \|^2
    \leq \frac{l_{g,1}^2}{\mu_g},
\end{align*}
where the first inequality follows $\nabla_{yy} g(x,y_g^*(x))\succeq \mu_g I$, the second uses the non-expansiveness as $\mathcal{C}_{\tilde{Y}(y_g^*(x))}$ is closed and convex, and the last uses the $l_{g,1}$-smoothness of $g$ and $\mu_g$-strongly-convexity. 
Rearranging thus proves the Lemma.

\end{proof}

In this way, we can proceed to prove Theorem~\ref{theorem: smoothness}.
\begin{proof}[Proof of Theorem~\ref{theorem: smoothness}]
    As $v(x)$ is differentiable (cf. Lemma~\ref{lemma: gradient of v}), the directional derivative for $v(x)$ is 
    \begin{align*}
        D_d(v(x))=\nabla v(x)^\top d = \nabla_x g(x,y_g^*(x))^\top d,
    \end{align*}
    for all unit direction $d$.
    By definition, the second-order directional derivative for $v(x)$ is
    \begin{align}
        D_{dd}^2(v(x)) =  & \lim_{r\downarrow 0}\frac{1}{r} \Big( D_d(v(x+rd))-D_d(v(x)) \Big) \nonumber\\ 
        = & \lim_{r\downarrow 0}\frac{1}{r} d^\top \Big( \nabla_x g\Big(x+rd,y_g^*(x+rd) \Big)- \nabla_x g\Big(x,y_g^*(x)\Big) \Big) \nonumber\\
        =& d^\top \nabla_{xx} g(x,y_g^*(x)) d  +d^\top \nabla_{xy} g(x,y_g^*(x)) D_d(y_g^*(x)). \nonumber 
    \end{align}
    where the first equality is the definition of second-order directional derivative; the second is by plugging the first-order directional derivative; the third follows Taylor's expansion.
    Denote $\tilde{g}(x,y;\delta)=\delta f(x,y)+g(x,y)$, similarly,
    \begin{align*}
        D_{dd}^2( v_\gamma(x))= & d^\top \nabla_{xx} \tilde{g}(x,y_\gamma^*(x);\gamma^{-1}) d +d^\top \nabla_{xy} \tilde{g}(x,y_\gamma^*(x);\gamma^{-1})D_d(y_\gamma^*(x))
    \end{align*}
    Additionally, we know
    \begin{align}
    \|\nabla^2 g(x,y_g^*(x) )-\nabla^2 \tilde{g}(x,y_\gamma^*(x);\gamma^{-1})\|=&\|\nabla^2 g(x,y_g^*(x) )-(\gamma^{-1}\nabla^2 f(x,y_\gamma^*(x))+\nabla^2 g(x,y_\gamma^*(x) ))\| \nonumber \\
    \leq & \|\nabla^2 g(x,y_g^*(x) )-\nabla^2 g(x,y_\gamma^*(x))\| + \gamma^{-1}\|\nabla^2 f(x,y_\gamma^*(x))\| \nonumber  \\
    \leq & l_{g,2} \| y_g^*(x)-y_\gamma^*(x)\|+\gamma^{-1} l_{F_\gamma,1} \nonumber\\
    \leq & \left(\frac{l_{f,0}l_{g,2}}{\mu_g}+ l_{F_\gamma,1} \right)\gamma^{-1} \label{eq: nabla xx diff},
    \end{align}
    where the first inequality follows the triangle inequality, the second uses the $l_{g,2}$-Lipschitzness of $\nabla^2 g$ and the $l_{F_\gamma,1}$-smoothness of $f$, and the third follows directly from Lemma \ref{lemma: distance of yg ygam}. 

In this way, knowing that the second order directional derivative of $F_\gamma(x)$ is in the form of \eqref{eq: 2nd order directional derivative of F}, we can bound its norm following the triangle inequality, Cauchy-Schwarz inequality, and the prepared lemmas:
\begin{align*}
    & \|  D_{dd}^2( F_\gamma(x)) \|=  \gamma \|   D_{dd}^2( v_\gamma(x))- D_{dd}^2( v(x)) \| \nonumber \\
    \leq &  \gamma \big(\underbrace{\|\nabla_{xx} g(x,y_g^*(x) )-\nabla_{xx} \tilde{g}(x,y_\gamma^*(x);\gamma^{-1})\|}_{\eqref{eq: nabla xx diff}}+\underbrace{\|\nabla_{xy} g(x,y_g^*(x))\|}_{l_{g,1}\text{-smoothness of }g}\underbrace{\| D_d(y_g^*(x))-D_d(y_\gamma^*(x))\| }_{\text{Lemma}~\ref{lemma: implicit directional derivative lipschitz}}\nonumber\\
    &+\underbrace{\|D_d(y_g^*(x))\|}_{\text{Lemma}~\ref{lemma: directional derivative bound}}\underbrace{\|\nabla_{xy} g(x,y_g^*(x))-\nabla_{xy} \tilde{g}(x,y_\gamma^*(x);\gamma^{-1})\|}_{\eqref{eq: nabla xx diff}}\big) \nonumber\\
    \leq & \gamma \left(
    \left(\frac{l_{f,0}l_{g,2}}{\mu_g}+ l_{F_\gamma,1} \right)\gamma^{-1} + l_{g,1}\Oc(\gamma^{-1}) +\frac{2 l_{g,1}}{\mu_g} \left(\frac{l_{f,0}l_{g,2}}{\mu_g}+ l_{F_\gamma,1} \right)\gamma^{-1}
    \right) = \Oc(1).
\end{align*}

\end{proof}

\section{Proof of Theorem~\ref{thm: ALT-PBGD}}
\label{appendix: proof of ALT-PBGD FSL}

Denote the gradient estimate $g_t = \nabla_x f(x_t,y_{t}^\gamma) +\gamma\nabla_x g(x_t,y_{t}^\gamma)-\gamma\nabla_x g(x_t,y_{t}^g)$. 
Similar to \eqref{eq: VaFF bias bound intermediate step}, we prepare
\begin{align}
    \|b_t\|^2 = \| \nabla F_\gamma(x_t)-g_t\|^2 
    \leq & 2\gamma^2 l_{\tilde{g},1}^2 \frac{2}{\mu_g} (1-\eta^y \mu_g)(\tilde{g}_\gamma(x_t,y_{t-1}^\gamma)-v_\gamma(x_t)) \nonumber\\
    & + 2 \gamma^2 l_{g,1}^2  \frac{2}{\mu_g}  (1-\eta^y \mu_g)  (g(x_t,y_{t-1}^g)-v(x_t))
    \label{eq: bias bound intermediate step}
\end{align}
where $l_{\tilde{g},1}$ is the smoothness modulus for $\tilde{g}(x,y)=\gamma^{-1}f(x,y)+g(x,y)$.
Following \eqref{eq: VaFF fully single-loop intermediate step} and plugging in \eqref{eq: bias bound intermediate step}, 
there is
\begin{align}
    & F_\gamma(x_{t+1})- F_\gamma(x_t) \nonumber\\
    \leq & -\frac{\eta}{4} \|\frac{x_{t+1}-x_t}{\eta}\|^2  + \eta\frac{4}{\mu_\gamma^*}\gamma^2 l_{\tilde{g},1}^2 (1-\eta^\gamma \mu_g)(\tilde{g}_\gamma(x_t,y_{t-1}^\gamma)-v_\gamma(x_t)) \nonumber \\
    & +\eta \frac{4}{\mu_g}\gamma^2 ( l_{g,1})^2 (1-\eta^g \mu_g)(\tilde{g}_\gamma(x_t,y_{t-1}^g)-v_\gamma(x_t)). \label{eq: bias bound intermediate step 2}
\end{align}
Moreover, similar to \eqref{eq: h -vh update}, we obtain
\begin{align}
    \tilde{g}_\gamma(x_{t+1},y_{t}^\gamma)-v_\gamma(x_{t+1}) 
    \leq & (1+\frac{\eta l_{\tilde{g},1}z}{2})(1-\eta^\gamma \mu_g)(\tilde{g}_\gamma(x_t,y_{t-1}^\gamma)-v_\gamma(x_t)) \nonumber\\
    &+ (\frac{\eta l_{\tilde{g},1}}{2z}+ \frac{\eta^2(l_{\tilde{g},1}+l_{v_\gamma,1})}{2})\|\frac{ x_{t+1}-x_t}{\eta}\|^2, \quad \forall z>0. \label{eq: h -vh update 2} 
\end{align}
and
\begin{align}
    g(x_{t+1},y_{t}^g)-v(x_{t+1})\leq & (1+\frac{\eta l_{g,1}z'}{2})(1-\eta^g \mu_{g^*})(g(x_t,y_{t-1}^g)-v (x_t))  \nonumber \\
    & + (\frac{\eta l_{g,1}}{2z'}+ \frac{\eta^2(l_{g,1}+l_{v,1})}{2})\|\frac{ x_{t+1}-x_t}{\eta}\|^2. \label{eq: g - v update} 
\end{align}

In this way, adding $c (\tilde{g}_\gamma(x_{t+1},y_{t}^\gamma)-v_\gamma(x_{t+1}))$ and $c'(g(x_{t+1},y_{t}^g)-v(x_{t+1}))$ 
to both sides of \eqref{eq: bias bound intermediate step 2}, there is
\begin{align*}
    & F_\gamma(x_{t+1})- F_\gamma(x_t) + c (\tilde{g}_\gamma(x_{t+1},y_{t}^\gamma)-v_\gamma(x_{t+1})) + c' (g(x_{t+1},y_{t}^g)-v(x_{t+1}))\\
    \leq & \left( -\frac{\eta}{4} +c \Big(\frac{\eta l_{\tilde{g},1}}{2z}+ \frac{\eta^2(l_{\tilde{g},1}+l_{v_\gamma,1})}{2} \Big) +c'  \Big(\frac{\eta l_{g,1}}{2z'}+ \frac{\eta^2(l_{g,1}+l_{v,1})}{2} \Big) \right)\Big\|\frac{x_{t+1}-x_t}{\eta}\Big\|^2 \\
    &+ c \Big( \big(1+\eta \big(\frac{ l_{\tilde{g},1}z}{2} + \gamma^2 l_{\tilde{g},1}^2 \frac{4}{\mu_\gamma^* c} \big) \big) (1-\eta^\gamma \mu_g)(\tilde{g}_\gamma(x_t,y_{t-1}^\gamma)-v_\gamma(x_t))  
    \Big) \\
    &+ c' \Big( \big(1+\eta \big(\frac{ l_{g,1}z'}{2} +  \gamma^2 l_{g,1}^2  \frac{4}{\mu_g c'} \big) \big) (1-\eta^g \mu_{g})(g(x_t,y_{t-1}^g-v(x_t))  
    \Big).
\end{align*}

Choose the following hyper-parameter,
\begin{align*}
    \begin{cases}
        c =\gamma \mu_g^{-\frac{1}{2}}\\
        c' =\gamma \mu_g^{-\frac{1}{2}}\\
        z  = 16 c l_{\tilde{g},1} \\
        z' = 16 c' l_{g,1}\\
        \eta^g \leq l_{g,1}^{-1}\\
        \eta^\gamma \leq l_{\tilde{g},1}^{-1}\\
        \eta \leq \min \left\{\frac{1}{16  c(l_{\tilde{g},1}+l_{v_\gamma,1})},\frac{1}{16  c' (l_{g,1}+l_{v,1})},\frac{\eta^\gamma \mu_g/(1-\eta^\gamma \mu_g)}{\frac{l_{\tilde{g},1}z}{2}+ \frac{4\gamma^2 l_{\tilde{g},1}^2}{\mu_g c}},\frac{\eta^g \mu_{g}/(1-\eta^g \mu_{g})}{\frac{l_{g,1}z'}{2}+ \frac{4\gamma^2 l_{g,1}^2}{\mu_{g} c}} \right\}
    \end{cases}
\end{align*}
i.e. $c,c' = \Oc(\gamma)$, $\eta= \Oc(\gamma^{-1})$, there is
\begin{align*}
    & F_\gamma(x_{t+1})- F_\gamma(x_t) + c (\tilde{g}_\gamma(x_{t+1},y_{t}^\gamma)-v_\gamma(x_{t+1})) + c' (g(x_{t+1},y_{t}^g)-v(x_{t+1}))\\
    \leq &  -\frac{\eta}{8} \|\frac{x_{t+1}-x_t}{\eta}\|^2 +c (\tilde{g}_\gamma(x_t,y_{t-1}^\gamma)-v_\gamma(x_t))  
    + c' l_{g,1}^2 (g(x_t,y_{t-1}^g-v(x_t)).  
\end{align*}
Denote $D_1 = F_\gamma(x_0)- F_\gamma(x_T)$, $D_2 = (\tilde{g}_\gamma(x_0,y_{-1}^\gamma)-v_\gamma(x_0)) - (\tilde{g}_\gamma(x_T,y_{T-1}^\gamma)-v_\gamma(x_T))$, and $D_3 = (g(x_0,y_{-1}^g)-v(x_0)) - (g(x_T,y_{T-1}^g)-v(x_T))$ Rearranging and telescoping gives
\begin{align*}
    \frac{1}{T}\sum_{t=0}^{T-1} \Big\|\frac{x_{t+1}-x_t}{\eta} \Big\|^2 \leq  \frac{8 \left(  D_1 + c D_2 + c' D_3\right)}{\eta T} = \Oc(\gamma^2 T^{-1})
\end{align*}
where the last equality is because $c,c' = \Oc(\gamma)$, and $\eta= \Oc(\gamma^{-1})$.
By plugging this into \eqref{eq: h -vh update 2} and \eqref{eq: g - v update}, and following a similar analysis as used when deriving \eqref{eq: bias update bound}, we obtain
\begin{align*}
    \frac{1}{T}\sum_{t=0}^{T-1}(g(x_t,y_{t-1}^g)-v (x_t))
    \leq &  \Oc \left( \left(\tfrac{\eta l_{g,1}}{2z'}+ \tfrac{\eta^2(l_{g,1}+l_{v,1})}{2} \right) \frac{1}{T}\sum_{t=0}^{T-1}\Big\|\tfrac{ x_{t+1}-x_t}{\eta}\Big\|^2 \right)
    = \Oc( T^{-1}), \nonumber \\
    \frac{1}{T}\sum_{t=0}^{T-1}(\tilde{g}_\gamma(x_t,y_{t-1}^\gamma)-v_\gamma(x_t)) \leq & \Oc \left( \left(\tfrac{\eta l_{\tilde{g},1}}{2z}+ \tfrac{\eta^2(l_{\tilde{g},1}+l_{v_\gamma,1})}{2} \right) \frac{1}{T}\sum_{t=0}^{T-1}\Big\|\tfrac{ x_{t+1}-x_t}{\eta}\Big\|^2 \right)
    = \Oc(T^{-1}),
\end{align*}
as $\left(\frac{\eta l_{\tilde{g},1}}{2z}+ \frac{\eta^2(l_{\tilde{g},1}+l_{v_\gamma,1})}{2}\right), \left( \frac{\eta l_{g,1}}{2z'}+ \frac{\eta^2(l_{g,1}+l_{v,1})}{2} \right)= \Oc(1/\gamma^2)$.
Moreover, plugging these back to \eqref{eq: bias bound intermediate step}, we have
$\frac{1}{T}\sum_{t=0}^{T-1}\| b_t\|^2\leq \Oc \big(\gamma^2 T^{-1} \big)$. Therefore, following a similar analysis as in \eqref{eq: G eta and x update}, there is
\begin{align*}
   \frac{1}{T}\sum_{t=0}^{T-1}  \| G_{F_\gamma,\mathcal{X}}(x_t)\|^2 \leq & \frac{1}{T}\sum_{t=0}^{T-1}  \Big\|\frac{x_t-x_{t+1}}{\eta} \Big\|^2 + \frac{1}{T}\sum_{t=0}^{T-1} \| b_t\|^2 \leq  \Oc\big(\gamma^2 T^{-1} \big).
\end{align*}

\section{Proof of Lemma~\ref{lemma: gradient equals zero generalization CC}}
\label{appendix: proof of Lemma Hessian of v CC}
Before proceeding to the proof of Lemma~\ref{lemma: gradient equals zero generalization CC}, we first present a more general lemma. 

\begin{lemma}
\label{lemma: gradient equals zero generalization analytic}
Suppose $\mathcal{Y}\subseteq \mathbb{R}^{d_x}$ is a closed and convex set with smooth boundary, $\nabla_y g(x,y)$ is continuous in $y$, and Lipschitz in $x$, and $g(x,y)$ satisfies $\mu_g$-Proximal PL in $y$ on $\mathcal{Y}$ and on a disturbed domain $\mathcal{Y}_{\delta}$ such that (i) $\mathcal{Y}\subset \mathcal{Y}_\delta$, and (ii) $d_{\mathcal{Y}}(y)=\delta$ for any $y\in \text{bd}(\mathcal{Y}_\delta)$. 
Fix any $y_g^*(x)\in S_g(x)$, and for arbitrary unit direction $d\in \mathbb{R}^{d_x}$, there exists a trajectory of $y_g^*(x+rd)$ such that
\begin{align*}
    \langle \nabla_y g(x,y_g^*(x)),\lim_{r\downarrow 0} \frac{y_g^*(x+rd)-y_g^*(x)}{r} \rangle = 0
\end{align*}
\end{lemma}
\begin{proof}
    We prove this by considering two cases.
    \textbf{Case 1) $\nabla_y g(x,y_g^*(x)) = 0$.} In this case, the equality automatically hold. We consider in the following the case $\nabla_y g(x,y_g^*(x)) \neq 0$.
    
    \textbf{Case 2) $\nabla_y g(x,y_g^*(x)) \neq 0$.} In this case, $y_g$ is on the boundary of $\mathcal{Y}$, i.e. $y_g^*(x)\in \text{bd}(\mathcal{Y})$, as the (local) optimum of a differentiable function happens either at stationary point or at extreme (i.e. boundary).

    Additionally, the continuity of $\nabla_y g(x,\cdot)$ gives $\{y^* \in \mathbb{R}^{d_y}:\nabla_y g(x,y_g^*(x))=0\}$ being a closed set, as $\{0\}$ is closed and the closeness is preserved by continuous mapping. As $\nabla_y g(x,y_g^*(x)) \neq 0$, there exists $\delta_0>0$ such that $\| y_g^*(x) -y^*\|\geq \delta_0$ for any $y_g^*(x)$ and any $y^* \in \{y^* \in \mathbb{R}^{d_y}:\nabla_y g(x,y_g^*(x))=0\}$. Moreover, since $g(x,\cdot)$ is PL and $\nabla_y g(x,y_g^*(x))\neq0$, no unconstrained stationary point lies in $\mathcal Y$; i.e. $\{y^* \in \mathbb{R}^{d_y}:\nabla_y g(x,y_g^*(x))=0\}\subseteq\mathcal Y^{c}$.

    Denote by $y_g^\delta(x) \in \arg\min_{y\in \mathcal{Y}_\delta} g(x,y)$.  
    By \citep[Theorem 2.87]{bonnans2013perturbation}, we have $\|y_g^\delta(x) - y_g^*(x)\| = \mathcal{O}(\delta)$.
    Recall that $y_g^*(x)$ has positive distance $\delta_0$ to any unconstrained stationary point $y^*$, i.e., $\|y_g^*(x)-y^*\|\geq \delta_0$.  
    Hence, for sufficiently small $\delta < \delta_1 \sim \delta_0$, we have $\|y_g^\delta(x) - y^*\| \geq \|y_g^*(x)-y^*\| - \|y_g^\delta(x) - y_g^*(x)\| > 0$, which implies $\nabla_y g(x,y_g^\delta(x)) \neq 0$. Therefore, $y_g^\delta(x)$ must lie on the boundary of $\mathcal{Y}_\delta$, i.e. $y_g^\delta(x) \in \operatorname{bd}(\mathcal{Y}_\delta)$.

    Moreover, under the condition that Proximal PL condition for $g(x,y)$ also holds for $y\in \mathcal{Y}_\delta$, there exist $y_g^\delta(x+rd) \in \arg \min_{y\in \mathcal{Y}_\delta}g(x+rd,y)$ such that $\| y_g^\delta(x)-y_g^\delta(x+rd)\|\leq L_y^g r$ for unit direction $d$, according to Remark~\ref{remark: Lipschitz of S(x)}. In this way, there exists some $r_0>0$ such that for all $0<r<r_0$, $\| y_g^\delta(x)-y_g^\delta(x+rd)\|<\delta$, i.e. $y_g^\delta(x+rd) \in \mathcal{Y}_\delta \backslash \mathcal{Y}$ as $y_g^\delta(x)$ is on the boundary of $\mathcal{Y}_\delta$, which is in $\delta$ distance to $\mathcal{Y}$. This directly implies $\nabla_y g(x,y_g^*(x+rd))\neq 0$ as otherwise $y_g^\delta(x+rd) =y_g^*(x+rd)\in \mathcal{Y}$. This gives $y_g^*(x+rd)\in \text{bd}(\mathcal{Y})$, for all $0<r<r_0$.

    In this way, we know that $y_g^*(x)$ and $y_g^*(x+rd)$ are on $\operatorname{bd}(\mathcal{Y})$ and $y_g^*(x+rd)$ approaches to $ y_g^*(x)$ in a Lipschitz way.
    Additionally, by the smoothness of $\mathcal{Y}$ on its boundary, $\mathcal{T}_{\partial \mathcal{Y}}(y_g^*(x))$ coincides with the critical cone $C_\mathcal{Y}(y_g^*(x))$. Hence, according to \citet[Theorem 2.5.6]{clarke1990optimization},
    \begin{align*}
        \lim_{r\downarrow 0} \frac{y_g^*(x+rd)-y_g^*(x)}{r} \in C_\mathcal{Y}(y_g^*(x)).
    \end{align*}
\end{proof}

In this way, we are ready to proceed to the proof of Lemma~\ref{lemma: gradient equals zero generalization CC}.
\begin{proof}[Proof of Lemma~\ref{lemma: gradient equals zero generalization CC}]
The proof follows directly from Lemma~\ref{lemma: gradient equals zero generalization analytic}.
Here, $g^\lambda(x,y)$ in \eqref{eq: g lambda x y} is proximal-PL in $y\in \mathcal{Y}$ and any disturbed domain $\mathcal{Y}_{\delta}$ defined in Lemma~\ref{lemma: gradient equals zero generalization analytic}, as $g(x,y)$ is strongly convex in $y$ and $c(x,y)$ is convex in $y$, according to \citep{karimi2016linear}. Moreover, as $\nabla_y g^\lambda(x,y)=\nabla_y g(x,y)+\langle \lambda_g^*(x),\nabla_y c(x,y)\rangle$, and as assuming $\|\lambda_g^*(x) \|<B_\lambda$ for all $x\in \mathcal{X}$, there is, for any $x_1,x_2\in \mathcal{X}$, 
\begin{align*}
    \|\nabla_y g^\lambda(x_1,y)-\nabla_y g^\lambda(x_2,y)\|\leq (l_{g,1}+ l_{c,1}\|\lambda_g^*(x) \|)\|x_1-x_2 \| \leq (l_{g,1}+ l_{c,1}B_\lambda)\|x_1-x_2 \|.
\end{align*}
i.e. $\nabla_y g^\lambda(x,y)$ is Lipschitz in $x$. So all conditions in Lemma~\ref{lemma: gradient equals zero generalization analytic} are checked and therefore proved.
\end{proof}


\end{document}